\documentclass[11pt]{amsart}
\pdfoutput=1

\usepackage{lineno}

\usepackage{float}

\usepackage{amsmath}%\
\usepackage{amssymb}%
\usepackage{amsthm}

\usepackage{multirow}% http://ctan.org/pkg/multirow
\usepackage{graphicx, import}
\usepackage{tikz-cd}
\usepackage{color}   %May be necessary if you want to color links

\usepackage{a4wide}

\usepackage{subfig}

\usepackage{parskip}

\usepackage{hyperref}
\hypersetup{
linktoc=page
}
\usepackage{cleveref}

\crefrangelabelformat{equation}%Writing multiple references as (1a-1c)
{(#3#1#4--#5\crefstripprefix{#1}{#2}#6)}
\crefrangelabelformat{subequation}%
{(#3#1#4--#5\crefstripprefix{#1}{#2}#6)}

\newtheorem{theorem}{Theorem}[section]
\newtheorem{theorem+definition}{Theorem\,+\,Definition}[theorem]

\numberwithin{case}{theorem}

\newtheorem*{claim*}{Claim}

\newtheorem{corollary}[theorem]{Corollary}

\newtheorem{lemma}[theorem]{Lemma}
\newtheorem{notation}[theorem]{Notation}
\newtheorem{convention}[theorem]{Convention}

\newtheorem{proposition}[theorem]{Proposition}
\newtheorem{assumption}[theorem]{Assumption}

\newtheorem{def+prop}[theorem]{Definition\,+\,Proposition}
\theoremstyle{definition}
\newtheorem{definition}[theorem]{Definition}
\theoremstyle{remark}
\newtheorem{remark}[theorem]{Remark}
\newtheorem{example}[theorem]{Example}
\DeclareMathOperator{\SL}{SL}
\DeclareMathOperator{\GL}{GL}
\DeclareMathOperator{\PGL}{PGL}
\DeclareMathOperator{\CC}{\mathbb{C}}
\DeclareMathOperator{\RR}{\mathbb{R}}
\DeclareMathOperator{\PP}{\mathbb{P}}
\DeclareMathOperator{\ZZ}{\mathbb{Z}}

\DeclareMathOperator{\LL}{\mathbb{L}}
\DeclareMathOperator{\HH}{\mathbb{H}}

\DeclareMathOperator{\im}{Im}

\DeclareMathOperator{\Hol}{Hol}

\DeclareMathOperator{\lcm}{lcm}
\DeclareMathOperator{\res}{res}
\DeclareMathOperator{\ord}{ord}

\DeclareMathOperator{\GM}{GM}

\DeclareMathOperator{\cd}{codim}

\newcommand{\codim}[1][M]{\cd(#1)}

\DeclareMathOperator{\ANN}{Ann}
\DeclareMathOperator{\GRC}{GRC}
\DeclareMathOperator{\MRH}{MRH}

\DeclareMathOperator{\interior}{int}
\DeclareMathOperator{\Sec}{Sec}

\newcommand\calA{\mathcal{A}}
\newcommand\calB{\mathcal{B}}

\newcommand\calU{\mathcal{U}}
\newcommand\calQ{\mathcal{Q}}
\newcommand\calT{\mathcal{T}}

\newcommand\calC{\mathcal{C}}

\newcommand\calM{\mathcal{M}}

\newcommand\calH{\mathcal{H}}
\newcommand\calX{\mathcal{X}}
\newcommand\calZ{\mathcal{Z}}
\newcommand\calS{\mathcal{S}}
\newcommand\calP{\mathcal{P}}
\newcommand\calY{\mathcal{Y}}

\newcommand{\Mgn}[1][g,n]{\calM_{#1}}

\newcommand{\Stra}[1][(\mu)]{\calH {#1}}
\newcommand{\WYSIC}[1][(\mu)]{\widetilde{\calH}{#1}}

\newcommand{\PStra}[1][(\mu)]{\PP\!\Stra[#1]}

\newcommand{\MSDS}[1][(\mu)]{\Xi\Mgn {#1}}

\newcommand{\PMSDS}[1][(\mu)]{\PP\MSDS[#1]}
\newcommand{\LPS}[1][\sigma]{\operatorname{LPS}_{#1}}% the log period space
\newcommand{\LPSpre}[1][\sigma]{\operatorname{LPS}^{\circ}_{#1}}% the log period space
\newcommand{\LPSmap}[1][\sigma]{\pi_{#1}}% the log period space
\newcommand{\Vext}[1][\sigma]{\widetilde{V}_{#1}}

\newcommand{\LogP}[1][\gamma]{\psi_{#1}}

\newcommand{\bdComp}[1][\Gamma]{D_{\enhancG[#1]}}
\newcommand{\bdCompP}{\PP\!D_{\enhancG}}

\newcommand{\LVLF}[1][i]{L_{#1}}
\newcommand{\GRCF}[1][i]{W_{#1}}

\newcommand{\Ri}[1][1]{\operatorname{R}^{#1}}
\newcommand{\Hi}[1][(i)]{\operatorname{H}^{1}_{#1}}

\newcommand{\twistD}{\mathbf{\eta}}%twisted differential
\newcommand{\stabC}{X}% stable curve
\newcommand{\modD}[1][\xi]{#1}% standard modification differential
\newcommand{\starD}[2][\parLevel]{#1\star #2}

\newcommand{\resleq}[2][i]{{#2}_{(\leq #1)}}
\newcommand{\resgeq}[2][i]{{#2}_{(\geq #1)}}
\newcommand{\reseq}[2][i]{{#2}_{(#1)}}
\newcommand{\resg}[2][i]{{#2}_{(> #1)}}
\newcommand{\resl}[2][i]{{#2}_{(< #1)}}
\newcommand{\pa}{\gamma}
\newcommand{\cycle}[1][]{[\pa_{#1}]}
\newcommand{\invpa}[1][b]{\hat{\pa}(#1)}
\newcommand{\invcycle}[1][b]{[{\invpa[#1]}]}

\newcommand{\dimU}{M}
\newcommand{\dimPar}{N}

\newcommand{\Nearby}[1][e]{\varsigma_{#1}}

\newcommand{\stdFix}[1][e]{\mathbb{V}_{#1}}
\newcommand{\stdDiff}[1][e]{\Omega_{#1}}
\newcommand{\stdAnn}[1][e]{ \calA_{#1}}
\newcommand{\imAnn}[1][e]{ \calB_{#1}}
\newcommand{\imDisk}[1][e]{ \calU_{#1}}
\newcommand{\wideDisk}[1][e]{ \widetilde{\calU}_{#1}}
\newcommand{\diskh}[1][\calZ_{k}]{\mathbb{D}_{#1}}
\newcommand{\stdFormA}[1][e]{\phi_{#1}}
\newcommand{\stdFormB}[1][e]{\upsilon_{#1}}
\newcommand{\gluing}[1][e]{\Upsilon_{#1}}

\newcommand{\intSum}[1][\pa]{\sum_{e\in E}\IP[#1,\Van]}

\newcommand{\parLevel}{t}% the levelwise scaling parameter
\newcommand{\parHor}{h}% the plumbing parameter at horizontal nodes

\newcommand{\stdpar}{b}

\newcommand{\Per}{\varphi}

\newcommand{\WYSI}{``WYSIWYG'' }

\newcommand{\Van}[1][e]{\lambda_{#1}}% Standard notation for vanishing cycles

\newcommand{\hcc}{horizontal-crossing cycle }

\newcommand{\Hsigma}{H_1(\Sigma\!\setminus\! P,Z)}
\newcommand{\Hcut}[1][i]{
H_1(\Sigma^{cut}_{(#1)}\!\setminus\! P,Z\cup \Lambda_{(#1)}^{ver,+})}
\newcommand{\Hver}[1][i]{
H_1(\Sigma_{(#1)}\!\!\setminus\! P,Z\cup \Lambda_{(#1)}^{ver,+})}

\newcommand\IP[1][\cdot,\cdot]{\langle #1\rangle}

\newcommand{\universal}{\calY}

\newcommand{\eq}{\calX}
\newcommand{\norm}{\tilde{\calX}}
\newcommand{\universalFam}[1][B, \omega]{(\universal\to {#1})}

\newcommand{\eqFam}[1][B, \twistD]{(\eq\to #1)}

\newcommand{\normFam}[1][B, \twistD]{(\norm\to {#1})}

\newcommand{\dualG}[1][\Gamma]{#1}% standard symbol for dual graph
\newcommand{\enhancG}[1][\Gamma]{\overline{\dualG[#1]}}% standard notation for enhanced dual graph
\newcommand{\lvl}[1][(i)]{\ell {#1}}% standard level function

\newcommand{\stdlvl}{i}% Standard level
\newcommand{\scl}[1][\stdlvl]{t_{\lceil {#1}\rceil}}% scaling parameter

\newcommand{\lvlset}{L^{\bullet}(\enhancG)}
\newcommand{\lvlsetb}{L(\enhancG)}

\newcommand{\topl}{{\top}}%Top level of a class or a cycle

\newcommand{\sharc}[1][B]{\arc:\Delta \rightarrow {#1}}
\newcommand{\arc}{f}

\tikzset{
  symbol/.style={
    draw=none,
    every to/.append style={
      edge node={node [sloped, allow upside down, auto=false]{$#1$}}}
  }
}

\newlength{\mylen}
\setbox1=\hbox{$\bullet$}\setbox2=\hbox{\tiny$\bullet$}
\numberwithin{equation}{section}

\title{The boundary of linear subvarieties}
\author{Frederik Benirschke}
\address{Mathematics Department, The University of Chicago,
Chicago, IL, 60637, USA}
\email{benirschke@uchicago.edu}

\begin{document}

\begin{abstract} We describe the boundary of linear subvarieties in the moduli space of multi-scale differentials. Linear subvarieties are algebraic subvarieties of strata of (possibly) meromorphic differentials that in local period coordinates are given by linear equations. The main example of such are affine invariant submanifolds, that is, closures of $\SL(2,\RR)$-orbits. We prove that the boundary of any linear subvariety is again given by linear equations in {\em generalized period coordinates} of the boundary. Our main technical tool is an asymptotic analysis of periods near the boundary of the moduli space of multi-scale differentials which yields further techniques and results of independent interest.
\end{abstract}
\maketitle
\setcounter{tocdepth}{1}
\tableofcontents
\newpage
\section{Introduction}
Let $\mu=(\mu_1,\ldots,\mu_n)\in \ZZ^n,\, \sum_{i=1}^n \mu_i=2g-2$. The stratum $\Stra$ is the moduli space consisting of pairs $(X,\omega)$ where $X$ is a Riemann surface of genus $g$ and $\omega$ is a meromorphic differential with multiplicities of zeroes and poles prescribed by $\mu$. The projectivized stratum $\PStra$ is the quotient of $\Stra$ by $\CC^*$, where $\CC^*$ acts on a differential by rescaling.
Strata have a natural linear structure, i.e. a set of coordinates, distinguished up to the action of the linear group, called period coordinates, such that the transition functions are linear. A special class of subvarieties of strata is given by linear subvarieties.

\begin{definition}
A $(\CC)$-\textit{linear subvariety} $M$ is an irreducible algebraic subvariety of a stratum $\Stra$ that, at any point, is given by a finite union of linear subspaces in local period coordinates.
\end{definition}

A particularly important class of linear subvarieties are {\em affine invariant submanifolds}. Those are linear subvarieties in strata of holomorphic differentials where the linear subspaces are defined over the real numbers. By a combination of \cite{EskinMirzakhani} and \cite{FilipAlgebraicity}, affine invariant submanifolds are exactly orbit closures of the natural $\SL(2,\RR)$-action.

Linear subvarieties  in projectivized strata are usually not compact. For example, affine invariant submanifolds are never compact since one can use cylinder deformations to degenerate to a stable curve.

 Recently in \cite{BCGGMsm} the authors constructed a modular compactification $\PMSDS$ of the projectivized stratum $\PStra$, the {\em moduli space of projectivized multi-scale differentials}. The goal of this paper is to study the boundary of a linear subvariety in $\PMSDS$.

 The boundary $\PMSDS\!\setminus\!\PStra$ parametrizes {\em multi-scale differentials}, i.e. stable curves together with a collection of meromorphic differential forms on the irreducible components, subject to several technical conditions, which we recall in \Cref{section:TwistDifferentials}. Furthermore, the boundary decomposes into a union of open boundary strata, each of which possesses a natural linear structure induced by \textit{generalized period coordinates}. We will explain the structure of the boundary in more detail in~\Cref{section:LinStructure}.

For technical reasons we work with the ``unprojectivized" \textit{moduli space of multi-scale differentials} $\MSDS$. The group $\CC^*$ acts on $\MSDS$ by rescaling and $\PMSDS=\MSDS/\CC^*$ is the quotient.
Our main result is as follows.
\begin{theorem}[Main theorem]\label{thm:Main} Let $M\subseteq\Stra$ be a $\CC$-linear subvariety.
Then the intersection of the closure $\overline M\subseteq \MSDS$ with any open boundary stratum $\bdComp$ of the moduli space $\MSDS$ of multi-scale differentials is a {\em levelwise} linear subvariety, for the natural linear structure on the boundary stratum $\bdComp\subset\partial\MSDS$.

Furthermore, the linear equations for $\partial M\cap \bdComp$ are {\em explicitly} computable from the linear equations for $M$ near the boundary.
\end{theorem}

For this statement, we recall that the irreducible components of stable curves in the boundary of $\MSDS$ are stratified by levels, depending on the vanishing order of the differential on each component along one-parameter families. By a levelwise linear subvariety we mean that each linear equation only relates periods of the differential along curves contained in the same level.

This version of the main theorem is only a preliminary qualitative result. In the course of the paper we state several more precise versions. Once we define the linear structure of the boundary, we can make a more precise, but still qualitative statement, given in~\Cref{thm:MainLvl}. Later, in~\Cref{section:MonodromyLinearVariety,section:CuttingOut,section:LinearStructures}, we will be able to determine the explicit equations defining the boundary $\partial M \cap\bdComp$ provided we know the linear equations defining $M$ at a point near the boundary.
 In~\Cref{prop:LinContainment} we give an explicit formula in local coordinates, while~\Cref{prop:CoordFree} gives a coordinate-free description of the linear equations defining $\partial M$.

Our main technical tool is a detailed asymptotic analysis of the behavior of periods near the boundary of $\MSDS$. When integrating differentials over cycles passing through nodes of the limiting stable curve, the period might diverge logarithmically. In particular, periods do not extend as holomorphic functions to the boundary $\partial\MSDS$, but  they do extend after subtracting their logarithmic divergences. We call the resulting functions \textit{log periods}.
The first part of the paper is devoted to properly defining log periods and computing their limits at the boundary.

Log periods can also be viewed naturally in the context of Hodge theory.
Stated in this language, the extension of log periods can be seen as an analogue of Schmid's nilpotent orbit theorem \cite{SchmidVHS} in the flat setting. We discuss the relations to Hodge theory in~\Cref{section:MonodromyLinearVariety}.
See also \cite{ecStrata} for similar discussions.

\subsection*{The linear equations of the boundary}
We now explain how to obtain the linear equations defining the intersection of a boundary stratum with the closure of a linear subvariety $M$ from the linear equations defining $M$ near the boundary.

Let $(X,\eta)$ be a multi-scale differential in the boundary of $M\cap \bdComp$.
% Every irreducible component of $X$ has an assigned level, depending on the vanishing order of the differential on this irreducible component along one-parameter families in the stratum $\Stra$ degenerating to $(X,\eta)$.
The stable curve $X$ has two types of nodes: {\em vertical} nodes connect irreducible components of different levels, while {\em horizontal} nodes connect components of the same level. Near the boundary stratum $\bdComp$ of $\MSDS$, every smooth surface can be cut by simple closed curves into subsurfaces of different levels. The subsurface of level $i$ specializes to the irreducible components of $X$ of level $i$ under degeneration.

In a period chart inside $\Stra$, a linear equation for $M$ is a homology class $F=\sum_{l} A_{l} [\pa_l]$ where the collection $\{\pa_l\}$ is a suitable basis for relative homology, called a {\em $\enhancG$-adapted} homology basis. We define the notion of $\enhancG$-adapted basis in \Cref{section:CohBasis}. Roughly speaking one starts by choosing a homology basis for each subsurface of level $i$ and extends those to a basis on the whole surface by only passing through lower levels.

We say the homology class $F$ is of {\em top level at most $i$}, if it can be represented by a sum of paths, each of which is completely contained in the subsurface of level $\leq i$. To obtain the equations for the boundary proceed as follows. Start with the defining equations $F_1,\ldots,F_k$ for $M$, written in terms of a $\enhancG$-adapted bases and put into reduced row echelon form. Then for each $F_l$ repeat the following steps.
\begin{enumerate}
\item Determine the top level $\topl(F_l)$ of $F_l$.
\item If the equation $F_l$ crosses  horizontal nodes of level $\topl(F_l)$, delete it.
\item Otherwise, restrict $F_l$ to each irreducible component of $X$ of level $\topl(F_l)$. The resulting cycle then defines an equation for $\partial M\cap \bdComp$.
\end{enumerate}
 We describe the restriction procedure more explicitly in \Cref{section:Filtr}. The collection of linear equations obtained in this way are  then the linear equations defining the boundary of $M$.
In \Cref{sec:Example} we give an explicit example illustrating the above process.

\subsection*{Potential applications}
The main theorem gives a novel tool to study  the classification problem for affine invariant submanifolds.
Let $M\subseteq \Stra$ be an affine invariant submanifold. Then by~\Cref{thm:Main} the intersection of $\partial M$ with {\em any} boundary stratum is a lower dimensional linear subvariety. One can now try to iterate this process inductively.
A useful consequence of~\Cref{thm:Main} for this approach is the following corollary.
\begin{corollary}
If the linear equations for $M$ are defined over a field $K\subseteq \CC$, then the linear equations  $\partial M$ intersected with any open  boundary stratum of $\MSDS$ are defined over a subfield of $K$.
\end{corollary}

Another consequence of the proof of \Cref{thm:Main} are restrictions on the possible linear equation defining $M$ inside $\Stra$ arising from considerations of invariance under monodromy.
The precise statement is given in \Cref{rem:MonInv}. This should be compared to the cylinder deformation theorem \cite[Thm. 5.1]{WrightCylinder} which also restricts the possible linear equations, albeit in a slightly different language and thus the results are not directly comparable. In  \cite{BDGCylinder} we investigate in detail the relation of our approach to cylinder deformation results, together with applications to describing the geometry and combinatorics of possible degenerations of affine invariant manifolds. More precisely, we will show that the cylinder deformation theorem for affine invariant submanifold is a direct consequence of algebraicity. Furthermore we determine the explicit analytic equations for the closure of linear subvarieties in plumbing coordinates  in a neighborhood of the boundary, rather than just describing the boundary.

\subsection*{Algebraicity}
We stress that our setup only works for algebraic subvarieties that are locally given by finitely many linear subspaces, and does not apply to merely analytic subvarieties. By \cite{FilipAlgebraicity} affine invariant submanifolds, i.e. analytic subvarieties given by subspaces defined over the {\em real} numbers, are always algebraic. On the other hand, \cite{BMpriv} have communicated to us an example of an analytic subvariety of a meromorphic stratum which is locally defined by linear equations with rational coefficients, which is not algebraic.

Algebraicity is only used once in the argument, in~\Cref{section:SetupComplexLinear}, where we use the classical fact that the Euclidean closure of an algebraic variety in an algebraic compactification coincides with the Zariski closure and in particular is an analytic variety.

Afterwards, we use the fact that every boundary point of an analytic variety is the limit of a holomorphic one-parameter family, and not just of some sequence. This will ultimately allow us to avoid the cautionary example from \cite[Section 4]{ChenWrightWYSIWYG} and take limits of linear  equations. We will discuss the cautionary example in more detail in \Cref{rem:AvoidEx}.

\subsection{Relationship to previous work}
Degenerations of affine invariant submanifolds have been considered in \cite{MirzakhaniWrightWYSIWYG, ChenWrightWYSIWYG}. If we consider a family of differentials inside the Hodge bundle, the limit on a stable curve is a collection of differentials on each irreducible component with at most simple poles at the nodes and opposite residues at each node.
In \cite{MirzakhaniWrightWYSIWYG, ChenWrightWYSIWYG} the authors consider a partial compactification $\WYSIC$ of $\Stra$ which is constructed by removing all nodes and filling them in with marked points, and contracting all components where the differential vanishes. Thus they only consider the top level part of a limit multi-scale differential in the boundary. Each differential $(X_{\infty},\omega_{\infty})\in \WYSIC$ is contained in a stratum $\Stra[(\omega_{\infty})]$ of possibly disconnected differentials with at most simple poles. The resulting partial compactification is called \WYSI compactification because only the parts of the limit are considered that are represented by flat surfaces of positive area. The following is a description of the boundary of an affine invariant submanifold inside $\WYSIC$.
\begin{theorem}[{\cite[Thm. 1.2]{ChenWrightWYSIWYG}}\label{thm:WYSImain}
]Let $M$ be an affine invariant submanifold and $(X_\infty,\omega_{\infty})\in \WYSIC$  with no simple poles.
The intersection of the boundary $\partial M\subseteq \WYSIC$ with the stratum ~$\Stra[(\omega_{\infty})]\subseteq \WYSIC$ is an algebraic variety, locally given by finitely many subspaces in the period coordinates of $\Stra[(\omega_{\infty})]$.
Furthermore, assume that a sequence $(X_n,\omega_n)$ of points of~$M$ converges to $(X_\infty,\omega_{\infty})$.
After removing finitely many terms, the sequence $(X_n,\omega_n)$ may be partitioned into finitely many subsequences  such that for each subsequence the tangent space to a branch of $\partial M \cap \Stra[(\omega_{\infty})]$  at $(X_\infty,\omega_{\infty})$, inside $\WYSIC$,
 is equal to the intersection of the tangent space
of a branch of $M$ at $(X_n, \omega_n)$ and the tangent space of $\Stra[(\omega_{\infty})]$, for $n$ sufficiently large.
Here we use the fact that, since $(X_{\infty},\omega_{\infty})$ has no simple poles, the tangent space to $\Stra[(\omega_{\infty})]$ is naturally a subspace of $\Stra$ at $(X_n,\omega_n)$.
\end{theorem}

\Cref{thm:Main} should then be seen as an analogue of \Cref{thm:WYSImain} for the moduli space of multi-scale differentials. Roughly speaking, \Cref{thm:WYSImain} says that the boundary of an affine invariant submanifold is given by linear equations on all components of the limit where the differential does not vanish. After suitable rescaling, the limits become non-zero on the remaining components, and we show that, after rescaling, the whole boundary is given by linear equations.

There exists a forgetful map $p:\MSDS\to\WYSIC$ by sending a multi-scale differential to its top level piece.  In \Cref{sec:bdWYSI} we will see that our results quickly imply \Cref{thm:WYSImain}. The crucial observation is that $p$ has compact fibers.
In the presence of simple poles and multiple levels, the description of the tangent space to the boundary in $\MSDS$ is much more involved than \Cref{thm:WYSImain}, and the complete description is given by ~\Cref{prop:CoordFree}.

The proofs in \cite{MirzakhaniWrightWYSIWYG, ChenWrightWYSIWYG} use the theory of cylinder deformations and thus only work for affine invariant submanifolds in strata of holomorphic differentials, our results on the other hand work for linear subvarieties with arbitrary coefficients and in meromorphic strata, provided that they are algebraic. In particular  \Cref{thm:WYSImain} is true for arbitrary linear subvarieties in meromorphic strata.

\subsection{Outline of the proof}
The proof of~\Cref{thm:Main} can be roughly divided into two parts. The first part is to determine a set of linear equations that are satisfied by any boundary point of $\partial M$. Afterwards we need to show that every point on the boundary satisfying those linear equations is indeed in $\overline{M}$.

After choosing a homology basis $\{\pa_1,\dots,\pa_d\}$, $M$ 
 can locally near  $x_0\in M$ be written as the zero locus of $k_0=\cd_{\Stra}(M)$ linear equations. In particular we can find a a matrix $A=(A_{kl})_{kl}$ in reduced row echelon form such that
\[
M= \left\{ (X,\omega)\in\Stra\,:\, \sum_{l=1}^{d} A_{kl}\int_{\pa_l}\omega =0 \text{ for } k=1,\ldots,k_0\right\}.
\]
 Na\"\i vely one would now take the limit of these equations as $\omega$ approaches the boundary of $\MSDS$, but the periods $\int_{\pa_l}\omega$, which are locally holomorphic functions on $\Stra$, cannot be extended holomorphically to the boundary. Firstly, due to monodromy one cannot continuously extend the cycles $[\pa_l]$ to a whole neighborhood of the boundary and secondly, along a sequence converging to the boundary the period might diverge.
In~\Cref{section:LogPeriods} we thus study the asymptotic behavior of periods as they approach the boundary of $\MSDS$.  The main result of that section, ~\Cref{thm:PerThm}, says that after subtracting suitable, explicitly given, multivalued, logarithmic terms, the period $\int_{\pa} \omega$ becomes monodromy invariant and extends holomorphically to $\MSDS$. The resulting extended ``periods'' are called \textit{log periods}.
In~\Cref{thm:PerThm} we additionally compute the limit of the log periods at the boundary $\partial\MSDS$.

We can now describe our strategy to produce linear equations satisfied on the boundary of $M$, which is the content of~\Cref{section:MonodromyLinearVariety}.  Let $b_0\in\partial M$ be a boundary point of the linear subvariety $M$ contained in an open boundary stratum. We can choose a one-parameter family $\arc:\Delta\to \overline{M}$ which is generically contained in $M$ and such that $f(0)=b_0$.  Since $M$ is a linear subvariety, the linear equations are invariant under the monodromy of the Gauss-Manin connection, and this forces the linear equations to be of a special form, see~\Cref{prop:Vdisk}.

The special form of the linear equations together with the explicit formula for the limit of log periods then immediately implies that, at least along one-parameter families, we can take the limit of the linear equations defining $M$. Thus we get necessary linear equations satisfied on $\partial M$. The precise statement is~\Cref{cor:LimEq}.

In~\Cref{section:CuttingOut} we then show that the linear equations obtained in~\Cref{section:MonodromyLinearVariety} are actually the defining equations for the boundary $\partial M$  intersected with an open boundary stratum.
On $\MSDS$ the linear equations for $M$ cannot be extended to the boundary, even if rewritten in log periods, but on a suitable cover of $\MSDS$, that we call the \textit{log period space} $\LPS[]$, they do extend. The proof of~\Cref{thm:Main} is then obtained by a detailed analysis of the extended linear equations on $\LPS[]$.
For technical reasons, instead of a single cover $\LPS[]$, we need to consider a countable collection $\LPS$ of such. The indexing set corresponds roughly to different monodromies of the periods along one-parameter families.

As a result of the proof of~\Cref{thm:Main} we obtain an explicit formula, in local coordinates, for how obtain linear equations for $\partial M$ intersected with an open boundary stratum, given the linear equations for $M$. In~\Cref{section:LinearStructures} we interpret these results in a coordinate-free way by constructing natural maps in relative homology relating the tangent spaces of the stratum $\Stra$ and the boundary $\bdComp$.

In \Cref{sec:bdWYSI} we apply the results of \Cref{section:LinearStructures} to prove  \Cref{thm:WYSImain}.

\subsection{Acknowledgments}
I would like to thank my advisor Samuel Grushevsky for suggesting and guiding me through this project.  I am also thankful for valuable conversations with Ben Dozier, Quentin Gendron, Martin M{\"o}ller, Chaya Norton and John Sheridan.
Furthermore, I am grateful to Alex Wright for suggesting to reprove \cite[Thm. 1.2]{ChenWrightWYSIWYG} (\Cref{thm:WYSImain})
 with our methods. I also thank the referee for many helpful comments.

\section{Basic setup and notation}\label{section:BasicSetup}

\subsection{Setup for families}\label{section:SetupFamilies}
We fix a stratum $\Stra$ of meromorphic differentials with
\[
\mu=(\mu_1,\dots,\mu_r,\mu_{r+1},\dots,\mu_{r+s})\in \ZZ^{r+s},\, r+s=n
\] where
\[
\sum_{i=1}^{r+s}\mu_i=2g-2; \quad\mu_1\geq \dots\geq \mu_r\geq 0 > \mu_{r+1}\geq \dots \geq \mu_{r+s}.
\]
Our setup for families of differentials is as follows.
A \textit{family of differentials} $(\pi:\calX\rightarrow B,\omega,\calS)$ is a family $(\pi:\calX\rightarrow B,\calS)$ of pointed stable curves with sections $\calS=(\calS_1,\dots,\calS_{r+s})$ over the base $B$, together with a section $\omega $ of $\omega_{\calX/B}(-\sum_{i=r+1}^{r+s}\mu_i\calS_i)$ defined on the complement of the nodes. By abuse of notation we will sometimes denote the family of differentials by $\omega$. For generic $b\in B$ we  require $\ord_{\calS_i(b)} \omega_b= \mu_i$  and additionally require $\omega$ to have no other zeroes or poles outside of the nodes.
A family of flat surfaces of type $\mu$ is a family of differentials where all fibers $X_b$ are smooth and $\omega_b\in\Stra$ for all $b\in B$.

We will often write
\[
\calZ=(\calZ_1:=\calS_1,\dots,\calZ_r:=\calS_r),\,\calP=(\calP_1:=\calS_{r+1},\dots,\calP_s:=\calS_{r+s})\]
 for the {\em zero} and {\em pole} sections, respectively.

 For equisingular families $(\calX,\omega)$  of differentials we let $(\widetilde{\calX},\omega)$ be the associated family which is obtained by fiberwise normalization. Here the differential  on $\widetilde{\calX}$ is simply the  pullback of $\omega$ from $X$ and by abuse of notation we denote it also by $\omega$. In this case we let $\calQ_e^{\pm}$ be the sections of the preimages of the nodes on $\widetilde{\calX}$.

We usually consider families over a smooth base $B=(\Delta^*)^{d}\times \Delta^{e}$ for non-negative integers $d,e$, most of the time arising as the complement of a simple normal crossing divisor. Our convention is that $\Delta^k$ is a polydisk in $\CC^k$ centered at the origin of sufficiently small radius, to be chosen, and possibly further shrunk.

{\em The moduli space of multi-scale differentials.}
%\label{section:MSDS}
We now start recalling the moduli space of multi-scale differentials $\MSDS$ and its projectivized version $\PMSDS$ constructed in \cite{BCGGMsm}.

The main features of interest for us are:
\begin{itemize}
\item $\MSDS$ and $\PMSDS$ are smooth algebraic orbifolds and their respective boundaries $\partial\MSDS,\, \partial\PMSDS$ are normal crossing divisors;
\item the boundary has a modular interpretation in terms of multi-scale differentials and assigned prong-matchings, which we will recall next;
\item $\PMSDS$ is compact.
\end{itemize}

{\em Orbifold structure.}
The moduli space $\MSDS$ of multi-scale differentials and its projectivization $\PMSDS$ are smooth, algebraic DM-stacks. All our results are true for linear algebraic substacks of $\Stra$. We usually omit the stack structure and only work with the underlying varieties.

\subsection{Enhanced level graphs}\label{section:LevelGraphs}
To describe the boundary of $\MSDS$, we need to  add additional decorations to the dual graph of a stable curve.  Our setup mostly follows the conventions from \cite{BCGGMgrc}, where we simplify some conventions to focus on the features that are important to us, avoiding some of the more technical notions. A \textit{level graph}  $\enhancG=(\dualG,\lvl[])$ is a stable graph $\dualG=(V,E,H)$ with half-edges $H$ corresponding to marked points of the stable curve, together with a total {\em pre}-order on the vertices $V$ defined by a \textit{level function}
\[
\lvl[]: V\rightarrow L^{\bullet}(\enhancG)
\]
where $L^{\bullet}(\enhancG):=\{0,-1,\dots, -\ell(\enhancG)\}$ is the set of levels. Following the convention of \cite{BCGGMsm} we write $L(\enhancG):=L^{\bullet}(\enhancG)\setminus\{0\}$, and refer to it as the set of \textit{lower levels}. An edge is called {\em horizontal} if it joins vertices of the same level, and {\em vertical} otherwise. We let $E^{ver},\,E^{hor}\subseteq E$ be the sets of all vertical and horizontal edges, respectively.
An \textit{enhancement} is an assignment of an integer $\kappa_e\geq 0$, called the {\em number of prongs}, to each edge $e$, so that $\kappa_e = 0$ if and only $e\in E^{hor}$.
If an edge $e$ joins the vertices $v$ and $v'$ such that $\ell(v)\geq \ell (v')$ then we let $\lvl[(e+)]$ be the level of $v$ and similarly $\lvl[(e-)]$ the level of $v'$. Furthermore we set $v(e+):=v$ and $v(e-):=v'$. At horizontal nodes we make a random choice. Similarly, for a half-edge $h$ we let $v(h)$ be the vertex adjacent to $h$ and $\ell(h)$ the level of $v(h)$.

We let $\resleq{\enhancG}$ be the restriction of $\enhancG$ to levels at most $i$, i.e. we remove all vertices from $\enhancG$ with levels above $i$ and all edges and half-edges connecting to those vertices.
The restrictions $\reseq{\enhancG},\resg{\enhancG}$ are defined similarly.

For later use we define
\begin{equation}\label{eq:ak}
a_i:= \lcm(\kappa_e),\, m_{e,i}:=a_i/\kappa_e,
\end{equation}
where the $\lcm$ is taken over all edges connecting $\resleq{\Gamma}$ and $\resg{\Gamma}$ and $m_{e,i}$ is defined for any edge $e$ such that $\ell(e+)> i\geq\ell(e-)$.

\subsection{Stable curves and level graphs}
Let $\enhancG$ be an enhanced level graph and  $(X,S)$ be a stable curve with marked points $S$ and dual graph $\dualG$. Usually we omit the marked points in our notation. We denote by $X_v$ the irreducible component of $X$ corresponding to $v\in V$. Similarly we let $\reseq{X}$ be the subcurve consisting of all irreducible components of level $i$. We refer to $\reseq[0]{X}$ as the top level of $X$. There are analogous definitions for the subcurve $\resleq{X}$ consisting of components of level $\leq i$, for $\resgeq{X}$, and $\resg{X}$. For each node $e$ let $q_e^+$ and $q_e^-$ be the preimages of the node that are contained in $X_{v(e+)}$ and $X_{v(e-)}$, respectively.

Let $S=Z\cup P$ be the marked points, partitioned into marked zeroes and poles.

On the normalization $\widetilde{X}$ of $X$ we define
\[
\begin{split}
\widetilde{Z}&:= Z \cup \{q_e^+, e\in E^{ver}\},\\
\widetilde{P}&:=P \cup \{q_e^-, e\in E^{ver}\}\cup \{q_e^{\pm},\,e\in E^{hor}\},\\
\widetilde{S}&:=\widetilde{Z}\cup \widetilde{P}.
\end{split}
\]

We denote by $\reseq{\widetilde{X}}$ the normalization of $\reseq{X}$ and consider it as possible disconnected curve with marked points $\widetilde{S}_{(i)}$ where

\begin{equation}\label{eq:Ximarked}
\widetilde{Z}_{(i)}:=\widetilde{Z}\cap \reseq{X},\,
\widetilde{P}_{(i)}:=\widetilde{P}\cap \reseq{X},\,
\widetilde{S}_{(i)}:=\widetilde{S}\cap \reseq{X}.
\end{equation}

We define $\widetilde{Z}_{v}, \widetilde{P}_{v}, \widetilde{S}_{v}$  on the normalization $\widetilde{X}_v$ of $X_v$ analogously.

\subsection{Multi-scale differentials}\label{section:TwistDifferentials}
The boundary of $\MSDS$ can be described in terms of multi-scale differentials.
A \textit{multi-scale differential} $(\stabC,S,\twistD)$ compatible with an enhanced dual graph $\enhancG$ is a stable curve $(\stabC,S)$ and a collection $\twistD=(\twistD_v)_{v\in V}$ of meromorphic differentials on the normalization $\widetilde{X}_v$ of each irreducible component $X_v$ satisfying
\begin{itemize}
\item \textbf{(Prescribed vanishing)} Each differential $\twistD_v$ is non-zero, and has no zeroes and poles outside $\widetilde{S}_v$. Moreover, the order of vanishing at the marked point $S_k$ is $\mu_k$.
\item \textbf{(Matching orders)} For every node $e$ we have
\[
\ord_{q_e^+} \eta_{v(e+)}=\kappa_e-1,\, \ord_{q_e^-}\eta_{v(e-)}=-\kappa_e-1.
\]
\item \textbf{(Matching residues at horizontal nodes)} At horizontal nodes $e\in E^{hor}$, we have
\[
\res_{q_e^+}\twistD_{v(e+)} + \res_{q_e^-}\twistD_{v(e-)}=0.
\]
\item \textbf{(Global residue condition)}
For every level $i$ and every connected component $Y$ of $\resg{X}$ that does not contain a marked point with a prescribed pole, i.e. such that $P \cap Y=\emptyset$, the following condition holds. Let $\{e_1,...,e_b\}$ denote the set of all nodes where $Y$ intersects $\reseq{X}$. Then
\[
\sum_{j=1}^{b} \res_{q_{e_j}^-} \eta_{v(e_j-)}=0.
\]
\end{itemize}
Instead of grouping the differentials by irreducible components, it is often useful to group them level by level. In this case we write $\twistD=\left(\reseq{\twistD}\right)_{i\in L^{\bullet}(\enhancG)}$. We usually omit the marked points $S$ in the notation, since they are already encoded as the zeroes and poles of $\twistD$  away from the nodes.

\subsection{The structure of the boundary}\label{section:Boundary}
The boundary components of the moduli space of multi-scale differentials $\MSDS$ are indexed by the discrete data of enhanced  level graphs. The {\em open} boundary stratum corresponding to the enhanced level graph $\enhancG$ is denoted by $D_{\enhancG}$. A point of $D_{\enhancG}\subseteq \partial\MSDS$ corresponds to a pair $(X,\twistD)$ where $X$ is a stable curve with dual graph $\dualG$, and $\twistD$ is a multi-scale differential compatible with $\enhancG$.
Additionally there needs to be a choice of a prong-matching at every vertical node. Since we only work locally, we do not need to keep track of the prong-matching, and refer to \cite[Section 5.4]{BCGGMsm} for a proper discussion.

Two multi-scale differentials $(X,\twistD)$ and $(X',\twistD')$ correspond to the same boundary point of $\bdComp$ if they are related by the action of the level-rotation torus which acts simultaneously on the different levels by rescaling and on the prong-matchings; we refer to \cite[Section 6]{BCGGMsm} for the precise definitions.
For our purposes we can again mostly ignore the action: near a boundary point $(X,\twistD)$ with a chosen prong-matching, the boundary component $D_{\enhancG}$ can be parametrized by a small neighborhood in the space of  multi-scale differentials compatible with $\enhancG$, considered up to scaling each differential $\twistD_{(i)}$ on lower levels by an arbitrary non-zero complex number, one complex number for each level.

\begin{remark}
Suppose $(X,\twistD)$ and $(X,\twistD')$ are two multi-scale differentials with the same underlying differential but different prong-matchings. Since $\twistD$ and $\twistD'$ have the same periods it seems interesting to ask: if $(X,\twistD)$ is contained in the boundary of a linear subvariety is the same true for $(X,\twistD')$ and furthermore are the linear equations the same? Our methods are purely local inside the moduli space of multi-scale differentials, i.e. they only allow us to describe the linear subvariety in a small neighborhood of $(X,\twistD)$ which might not contain $(X,\twistD')$ and thus seem not applicable to this question.
\end{remark}

\subsection{Local coordinates on the boundary} \label{sec:GeneralizedPeriodCoordinates}
It is classically known that the stratum $\Stra$ has local coordinates given by the relative cohomology $H^1(X\setminus P,Z;\CC)$. The boundary stratum $D_{\enhancG}$ has a similar local description, which we now discuss.

The prescribed vanishing and matching orders conditions for multi-scale differentials imply that a multi-scale differential $\twistD$ is contained in the product of  strata $\prod_{v\in V}\Stra[(\mu_v)]$, where each $\mu_v$ is completely determined by $\mu$ and the enhanced level graph $\enhancG$.
Thus the space of  multi-scale differentials, i.e.~unprojectivized and without a choice of prong-matchings, can be identified with the subspace $\prod_{v\in V} \Stra[(\mu_v)]^{GRC}$ of $\prod_{v\in V} \Stra[(\mu_v)]$, constrained by the matching residues at horizontal nodes, as well as the global residue conditions.
To describe the boundary component $\bdComp$, we need to additionally projectivize the differential on lower levels, and choose the prong-matchings. This causes the stratum $\bdComp$ to be a cover of $\prod_{v\in V} \Stra[(\mu_v)]^{GRC}$, suitably projectivized.
 We can use this to describe local coordinates on $D_{\enhancG}$. For every level $i$ we set
\begin{equation}\label{eq:HiX}
\Hi(X;\ZZ):= H^1(\reseq{\widetilde{X}}\!\!\setminus \!\reseq{\widetilde{P}},\reseq{\widetilde{Z}};\ZZ),
\end{equation}
where we recall $\reseq{\widetilde{P}},\reseq{\widetilde{Z}}$ from~\cref{eq:Ximarked}. We sometimes simply write $\Hi(X)$ instead of $\Hi(X;\ZZ)$ and also write $\Hi(X;\CC):=\Hi(X;\ZZ)\otimes_{\ZZ}\CC$. We additionally denote $\Hi(X)^{\GRC}\subseteq \Hi(X)$ the subspace satisfying the global residue and matching residue conditions at horizontal nodes.
We will revisit the global residue condition in \Cref{sec:GRC}. See in particular \cref{eq:GRC} for an explicit definition of $\Hi(X)^{\GRC}$.
The boundary stratum $D_{\enhancG}$ then has \textit{local projective coordinates} given by
\[
\Hi[](X,\enhancG):=\Hi[(0)](X;\CC) \times
\prod_{i\in\lvlsetb} \PP\!\left (\Hi(X;\CC)^{GRC}\right).
\]
By this we mean that, after choosing local coordinates on each projective space, we get local coordinates on $\bdComp$. Note that this statement is only meaningful because the transition functions in those coordinates are given by projective linear maps. We will discuss the transition functions in more detail in~\Cref{section:LinStructure}.
We refer to those coordinates as \textit{generalized period coordinates}.
Similarly, the boundary $\bdCompP$ of $\PMSDS$ has projective local coordinates given by $\PP(H^1(X,\enhancG))$, where additionally  the homology  $\Hi[(0)](X)$ of the top level is projectivized.

Let $U\subseteq \bdComp$ be such a generalized period chart centered at $b_0=(\stabC_{b_0},\twistD_{b_0})$.
Then over $U$ there exists an equisingular family \begin{equation}\label{eq:ClutchingFamily}
\eqFam[U, \twistD]
\end{equation}
of stable curves with dual graph $\dualG$, where $\eta$ is the multi-scale differential determined by generalized period coordinates.
We define $(\eq_{(i)}\rightarrow U,\twistD_{(i)})$ to be the potentially disconnected family consisting only of irreducible components of level $i$.
From time to time it will be useful to consider the fiberwise normalization
\begin{equation}\label{eq:NormFamily}
\normFam[U, \twistD],
\end{equation}
which is a family of smooth, possibly disconnected, Riemann surfaces, where we make a choice of marking of the preimages of all nodes.
Notice that while a point in $U$ only parametrizes an equivalence class of multi-scale differentials, choosing local charts on each projective space $\PP\left(\Hi(X)^{\GRC}\right)$ allows us to choose  for each $u\in U$ a representative $(X_u,\eta_u)$, varying holomorphically in $u$.

\begin{convention}
From now on,  $b_0\in\bdComp$ will denote a boundary point, chosen once and for all, in a neighborhood of which in $\MSDS$ we will perform all of our constructions and computations.
We will usually write $(X,\twistD)$ instead of $(X_{b_0},\twistD_{b_0})$.
% We will often write $(X_0,\twistD_0)$ instead of $(X_{b_0},\eta_{b_0})$ for the fiber of $\eq$ over $b_0$.
Furthermore, from now on, $X$ always denotes a stable curve contained in $\bdComp$ and $\Sigma$ a smooth curve in $\Stra$.
\end{convention}

\subsection{The linear structure of the boundary}\label{section:LinStructure}
After choosing a homology basis on $\reseq{X}$ for each $i$, the changes of coordinates for $D_{\enhancG}$ are given by linear transformations in $\GL(\enhancG):=\GL(d_1,\ZZ)\times\left(\prod_{i\in\lvlsetb} \PGL(d_i,\ZZ)\right)$ where $d_i:=\dim \Hi(X)^{GRC}$.
%See~\Cref{section:MonodromyGM} for more details on the change of coordinates.
Thus $D_{\enhancG}$ possesses  a $\GL(\enhancG)$-structure or what we call a \textit{levelwise linear structure}. We call a subvariety of $D_{\enhancG}$ \textit{levelwise linear} if locally in generalized period coordinates it is given by subvarieties
\[
V_0 \times \prod_{i\in\lvlsetb} \PP(V_i)\subseteq \Hi[](X,\enhancG)
\]
where $V_i\subseteq \Hi(X;\CC)$ are some linear subspaces.
Since on $\bdCompP$ also the top level is projectivized, $\bdCompP$ admits a $\prod_{i\in\lvlset} \PGL(d_i,\ZZ)$-structure, i.e. coordinate changes live in $\prod_{i\in\lvlset} \PGL(d_i,\ZZ)$, and a subvariety is {\em levelwise linear} if locally it is  given by $\prod_{i\in\lvlset} \PP(V_i)$.
Note that if a subvariety $M\subseteq D_{\enhancG}$ is levelwise linear, the same is true for its image $\PP\!M\subseteq \PP\!D_{\enhancG}$.
We can now give a more precise, though still qualitative, version of our main result,~\Cref{thm:Main}.
\begin{theorem}[Main theorem, Levelwise version] \label{thm:MainLvl}
Let $M\subseteq \Stra$ be a linear subvariety.
For each open boundary component $\bdComp\subseteq \MSDS$ the intersection $
\partial M\cap \bdComp$
 is a levelwise linear subvariety of $\bdComp$.
The same is true for the projectivization $\PP\! M\subseteq \PStra$.
\end{theorem}

We stress that this statement already greatly restricts the possible linear equations of $\partial M\cap \bdComp$, since each linear equations only involves periods contained in the same level.

In~\Cref{section:LinearStructures} we will describe how to relate the levelwise linear structure on the boundary stratum $\bdComp$ with the linear structure on the stratum $\Stra$.
\begin{remark}
We stress that all levelwise projectivizations above are taken with respect to the standard action of $\CC^{\lvlsetb}$ on $\prod_{i\in\lvlsetb} H^1_{(i)}(X)$, not to be confused with the triangular action which will be introduced in~\cref{eq:TriangularAction}, following \cite[eq. (11.1)]{BCGGMsm}.
\end{remark}

\subsection{The model domain}\label{section:ModelDomain}
We now recall the local structure of the moduli space of multi-scale differentials near the open boundary stratum $\bdComp$.
In \cite[Section 8]{BCGGMsm}, the authors first introduce an auxiliary space, the model domain, and then show later in \cite[Section 10]{BCGGMsm} that it is locally biholomorphic to $\MSDS$. Local coordinates of the model domain $\MSDS$ near $b_0\in\bdComp$ can be given by
\begin{equation}\label{eq:ModelDomain}
B:=U\times \Delta^{\ell(\enhancG)-1}\times \Delta^{|E^{hor}|},
\end{equation}
where $U\subseteq \bdComp$ denotes a generalized period chart.
\begin{convention} From now on, unless stated otherwise, $U\subseteq \bdComp$ will always refer to a generalized period chart in $\bdComp$ centered at $b_0$, which we allow to be further shrunk as needed. Furthermore, we often implicitly identify $U$ with $U\times(0,\ldots,0)\subseteq B$.
\end{convention}
We call $B$ the  \textit{local model domain}
and denote its coordinates by $b=(\twistD,\parLevel,\parHor)$ with  {\em scaling parameters} $\parLevel=(t_i)_{i\in \lvlsetb}$,  {\em horizontal node parameters} ~$\parHor=(h_e)_{e\in E^{hor}}$, where $\eta$ is a multi-scale differential. We omit the stable curve $\stabC$ from the notation.

\begin{notation}
Throughout the text we denote
\begin{equation}\label{eq:mn}
\dimPar:=\ell(\enhancG)-1+|E^{hor}|,\quad \dimU:=\dim U=\dim \bdComp.
\end{equation}
In particular, $\dim B=\dimPar+\dimU$.
\end{notation}
Note that our notation here differs from \cite{BCGGMsm} where $N$ denotes the number of levels, not including the count of horizontal nodes. We recall that we denote the number of levels by $\ell(\enhancG)$.

Let $p:B\to U$  denote the projection onto the first factor. On $B$ we consider the pullback family $(p^*\eq\to B,\eta)$ where $\eqFam[U]$ is the family from~\cref{eq:ClutchingFamily}, which we call the \textit{model family}. We usually omit the projection map $p$ and only write
\begin{equation}\label{eq:ModelFam}
\eqFam.
\end{equation} for the model family.

We will explain the role of the parameters $t_i$ and $h_e$ more precisely in~\Cref{section:UniversalFamily}. For the moment we only define, following \cite{BCGGMsm},
\begin{equation}\label{eq:ScalingParameters}
\scl[\stdlvl]:=\prod_{k=i}^{-1}t_{k}^{a_{k}}\,,
\end{equation}
where the exponents $a_k$ are defined in~\cref{eq:ak}, and define the triangular action of $t$ on $\twistD$ by
\begin{equation}\label{eq:TriangularAction}
\starD{\twistD}:=( \scl[\stdlvl]\reseq{\twistD})_{i\in \lvlset}.
\end{equation}
We also define the \textit{plumbing parameters}
\begin{equation}\label{eq:PlumbingParameters}
s_e:= \begin{cases} \prod_{i=\ell(e-)}^{\ell(e+)-1}t_i^{m_{e,i}} & \text{ at vertical nodes},\\
 h_e & \text{ at horizontal nodes}.
 \end{cases}
\end{equation}
where the exponents $m_{e,i}$ are defined in~\cref{eq:ak}.
Note that in particular at vertical nodes we have the relation
\begin{equation}\label{eq:PlumbRelation}
\scl[\ell(e-)]=s_e^{\kappa_e}\scl[\ell(e+)].
\end{equation}

We define the (local) boundary $D\subseteq B$ as the normal crossing divisor given by the equations
\begin{equation}\label{eq:boundary}
D:=\left\{\prod_{i\in\lvlsetb} t_i\cdot \prod_{e\in E^{hor}} h_e=0\right\}
\end{equation}
The boundary component $U\simeq\bdComp\cap B=U\times (0,\dots,0)\subseteq D$ is called \textit{the most degenerate boundary stratum}, while the complement $D\setminus\bdComp$ corresponds to {\em partial undegenerations} of $\enhancG$. We will not need the precise definition of undegenerations, and instead refer the reader to \cite{BCGGMsm}.

\subsection{The universal family of multi-scale differentials}\label{section:UniversalFamily}
In \cite[Section 10]{BCGGMsm} the authors use plumbing to construct the \textit{universal family $\universalFam$  of multi-scale differentials} over the base $B$ defined in~\cref{eq:ModelDomain}.
We refer the reader  to \cite[Section 11]{BCGGMsm} for the precise definition of families of multi-scale differentials. For our purposes we only need the following properties of  the universal family $\universal$:
\begin{enumerate}
\item  For any $b\in B\setminus D$ the differential $\omega_b$ is a flat surface in the stratum $\Stra$;
\item  On the most degenerate stratum, i.e. for any $b\in U\times (0,\dots,0)$, the differential $\omega_b$ is a multi-scale differential in $\bdComp$;
\item There exist families of unions of disks $\wideDisk[]\subseteq \universal,\, \imDisk[]\subseteq\eq$ containing $\widetilde{S}$, and a biholomorphism
\begin{equation}\label{def:psi}
\Psi:\universal\setminus\wideDisk[]\simeq\eq\setminus\imDisk[].
\end{equation}
Near a marked point $\wideDisk[]$ and $\imDisk[]$ are homeomorphic to a disk, while at nodes they are homeomorphic to a union of two disks intersecting at the node;
\item Suppose $K\subseteq \universal\setminus\wideDisk[]$ is compact and $\Psi(K)\subseteq\reseq{\eq}$. Then
\[
\lim_{\parLevel,\parHor\to 0}  \dfrac{1}{\scl[\stdlvl]}\omega(b)_{|K}=\reseq{\twistD}
\]
uniformly, where $b=(\eta,t,h)$ and in the limit all $t_i$ and $h_e$ go to zero. In other words, as $b$ approaches a boundary point $\eta\in \bdComp$,  on the $i$-level  $\omega_{(i)}(b)$, rescaled by $\scl[i],$ converges uniformly to $\reseq{\twistD}$, away from the nodes and marked points.
\item Along the most degenerate boundary stratum $\bdComp$ the map $\Psi$ extends to an isomorphism
\[
\universal|_{\bdComp}\simeq \eq|_{\bdComp}.
\]
\end{enumerate}

\section{Constructing the universal family of \texorpdfstring{$\MSDS$}{}}\label{section:ConstrUniversal}
In this section we outline the construction of the universal family $\universal$. We follow \cite[Section 10]{BCGGMsm} in notation and setup, but we only highlight the features of the construction necessary for our discussion.

\subsection{Modification differentials}
\label{section:modD}
For a multi-scale differential, the residues match at horizontal nodes, while at vertical nodes  the multi-scale differential is holomorphic at $q_e^+$ and has a pole at $q_e^-$. On the other hand, for the plumbing construction in~\Cref{section:PlumbingSetup} it will be important to have differentials with matching residues at every node.
The solution, as found in \cite{BCGGMgrc}, is as follows. It is a consequence of the global residue condition that we can add a ``small" differential $\modD$ to $\twistD$ such that the residues  of $\starD{\twistD}+\modD$ match at all nodes. The precise definition is as follows, see also \cite[Def. 9.1]{BCGGMsm}. A family of \textit{modifying differentials} $\modD$ for the model family $\eqFam$  is a family of   meromorphic differentials $(\calX\rightarrow B,\modD)$ with $\modD=(\modD_v)_{v\in V}$ such that:
\begin{itemize}
\item $\modD$ is holomorphic except for possible simple poles along nodal and polar sections of $\twistD$. We allow $\modD$ to have residues at horizontal nodes.
\item $\reseq{\modD}$ is divisible by $\scl[\stdlvl-1]$ for each $i\in\lvlsetb$, and $\reseq[-\ell(\enhancG)]{\modD}\equiv 0$;
\item $\starD{\twistD}+\modD$ has matching residues at all nodes.
\end{itemize}

\subsection{Plumbing setup}
\label{section:PlumbingSetup}
In ~\cite[Section 10]{BCGGMsm} the authors introduced a plumbing setup for multi-scale differentials which we will now recall and then use subsequently.
We will subsequently work on a polydisk $B_{\varepsilon}\subseteq B$ of radius $\varepsilon=\varepsilon(b_0)>0$.
We define the standard annulus in $\CC$:
\[
A_{\delta_1,\delta_2}:=\{ \delta_1 < |z| <\delta_2\}.
\]
For $\delta=\delta(b_0)>0$, to be determined later,
we define the
\textit{standard plumbing fixture} to be
\[
\stdFix:=\{ (b,u,v)\in B_{\varepsilon}\times \Delta_{\delta}^2\,:\, uv=s_e(b)\},
\]
where $s_e(b)$  is defined by~\cref{eq:PlumbingParameters}. We consider $\stdFix\to B_{\varepsilon}$ as a family over $B_{\varepsilon}$ with fibers $(\stdFix)_{b}$. We equip $\stdFix$ with the relative one-form $\stdDiff$ given by
\[
\stdDiff:= (\scl[\ell(e+)]u^{\kappa_e}-r_e')\dfrac{du}{u}=-(\scl[\ell(e-)]v^{-\kappa_e}+r_e')\dfrac{dv}{v}
\]
with residue $r_e'$ to be determined later. We also consider the families of disjoint annuli $\stdAnn^{+},\stdAnn^{-}\subseteq \stdFix$ given by
\[\begin{split}
\stdAnn^{+}&:= \{ (b,u,v)\,:\, \delta/R < |u| < \delta \},\\
\stdAnn^{-}&:= \{ (b,u,v)\,:\, \delta/R < |v| < \delta \},
\end{split}
\]
for some constant $R>0$.
\begin{definition}\label{def:Van}
For $b\in B\setminus D$, we define the \textit{vanishing cycle} $\lambda_e\subseteq (\stdFix)_b$ to be the standard generator of the fundamental group of the annulus $(\stdFix)_b$ in $u$-coordinates, represented by a path encircling the origin once with  counterclockwise orientation.
\end{definition}
Our convention for the orientation on $\lambda_e$ has the following interpretation: If one chooses a tangent vector $\vec{v}$ on the curve pointing from lower levels to higher levels, then the frame $(\vec{u},\vec{v})$ is positively oriented where $\vec{u}$ is the tangent vector along $\lambda_e$, as seen in \Cref{fig:Orientation}. At horizontal vanishing cycles the orientation depends on the random choice of $e+$ and $e-$.

\begin{figure}[h]
\includegraphics[scale=0.1]{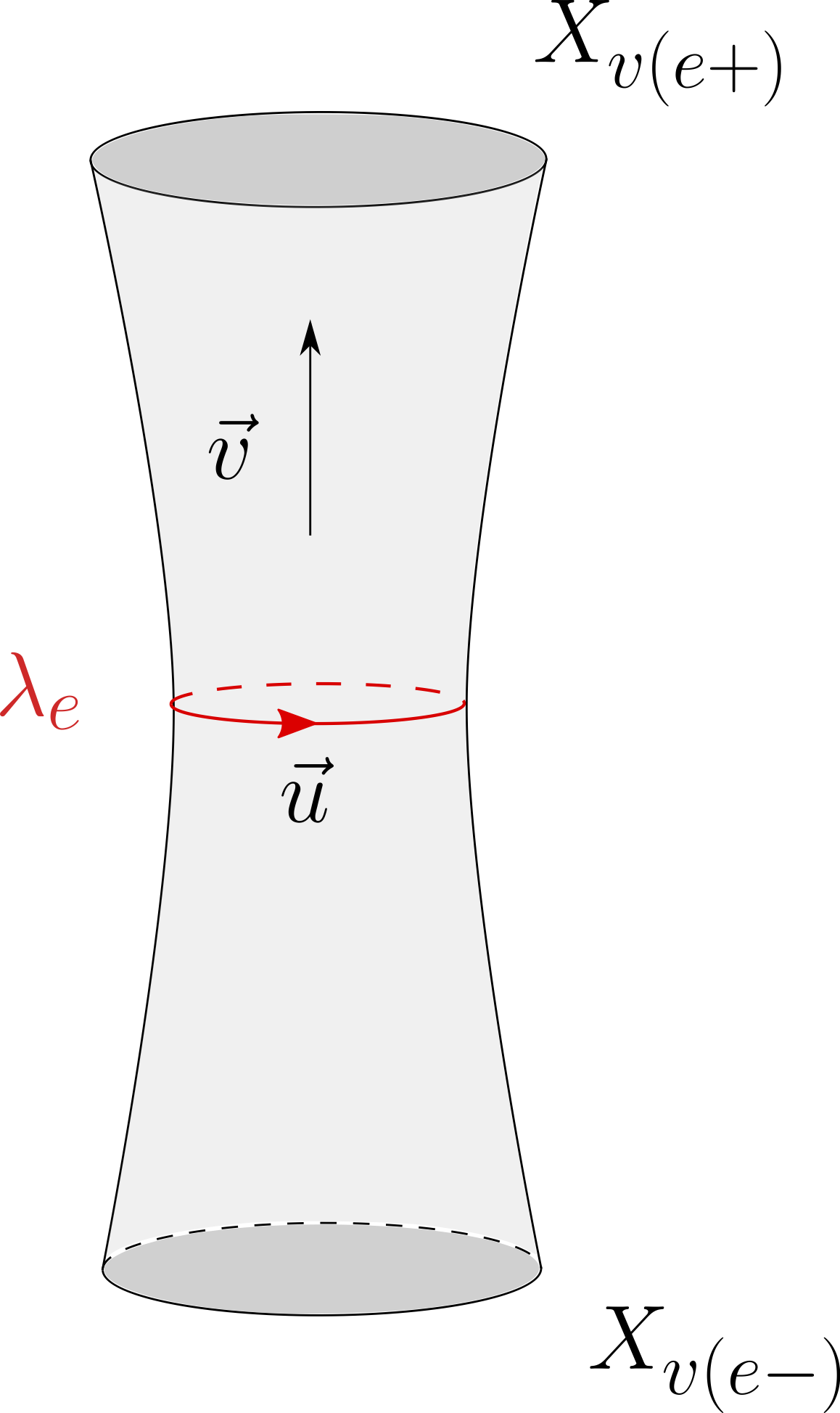}
\caption{Orientation on vanishing cycles}
\label{fig:Orientation}
\end{figure}

We stress that we consider $\lambda_e$ as an actual path and not just a homology class. If no confusion is possible we will not distinguish between $\lambda_e$ and its associated class.

For each marked zero $\calZ_{k}$ of order $m_k$, we define a family of disks, equipped with a relative one-form $\stdDiff[\calZ_{k}]$, by
\[
\diskh:= B_{\varepsilon}\times \Delta_{\delta},\quad \stdDiff[\calZ_{k}]:=z^{m_k}dz.
\]
We define a family of annuli $\stdAnn[\calZ_{k}]\subseteq\diskh$ by
\[
\stdAnn[\calZ_{k}]:= B_{\varepsilon}\times A_{\delta/R,\delta}.
\]

\subsection{Standard form coordinates for multi-scale differentials}
\label{section:StdFormCoord}
The idea of the plumbing construction for the universal family $\universal$ is to find local coordinates near marked zeroes and nodal sections in which the families $(\eq,\starD{\twistD})$ and ~$(\eq,\starD{\twistD}+\modD)$ have a simple form. The difference between the two families is that the order of vanishing at the nodes and marked points is constant for the first family $(\eq,\starD{\twistD})$ but it can jump for the second family $(\eq,\starD{\twistD}+\modD)$. In \cite{BCGGMsm} the authors introduce the following solution. For the family $(\eq,\starD{\twistD})$ we can find local coordinates near each nodal section and each marked point, in which the differential has a simple form, while for the second family $(\eq,\starD{\twistD}+\modD)$ this is only possible on an annulus such that the disc bounded by it contains the marked zeroes and the nodes. The results are as follows. We begin with the family $(\eq, \starD{\twistD})$.

\begin{theorem}[{\cite[Thm. 4.1]{BCGGMsm}}]
There exists a constant $\delta_1 >0$ such that for each edge $e$ and for each marked zero $\calZ_k$ of $\dualG$ there are families of conformal maps of disks
\[\begin{split}
\stdFormA^+&: B_{\varepsilon}\times \Delta_{\delta_1}\rightarrow \eq_{\ell(e+)},\\
\stdFormA^-&: B_{\varepsilon}\times \Delta_{\delta_1}\rightarrow \eq_{\ell(e-)},\\
\stdFormA[\calZ_{k}]&: B_{\varepsilon}\times \Delta_{\delta_1}\rightarrow \eq_{\ell(\calZ_k)},
\end{split}
\]
such that the following properties are satisfied:
\begin{enumerate}
\item The restrictions of these maps to $B_{\varepsilon}\times 0$ coincide with nodal sections  $\calQ_{e^+},\calQ_{e^-}$ and the marked sections $\calZ_k$, respectively.
\item The pullback of $\starD{\twistD}$ has standard form, that is
\[\begin{split}
(\stdFormA^+)^*&(\starD{\twistD})=  \scl[\ell(e+)]\left( z^{\kappa_e}-\res_{q_e^-}(\starD{\twistD})\right ) \dfrac{dz}{z},\\
(\stdFormA^-)^*&(\starD{\twistD}= -\scl[\ell(e-)]\left( z^{-\kappa_e}+\res_{q_e^-}(\starD{\twistD})\right) \dfrac{dz}{z},\\
(\stdFormA[\calZ_k])^*&(\starD{\twistD})= \scl[\ell(\calZ_k)]z^{m_k} dz.
\end{split}
\]
\end{enumerate}
\end{theorem}
The next result concerns the family $(\eq,\starD{\twistD}+\modD)$.

\begin{theorem}[{\cite[Thm. 10.4]{BCGGMsm}}]\label{thm:Plumbing}
For any $R >1$, there exist constants $\varepsilon,\delta >0$ such that for each edge $e$ and for each marked zero $\calZ_k$ of $\dualG$ there are families of conformal maps of annuli
\[\begin{split}
\stdFormB^+&: B_{\varepsilon}\times A_{\delta/R,\delta}\rightarrow \eq_{\ell(e+)},\\
\stdFormB^-&: B_{\varepsilon}\times A_{\delta/R,\delta}\rightarrow \eq_{\ell(e-)},\\
\stdFormB[\calZ_{k}]&: B_{\varepsilon}\times A_{\delta/R,\delta}\rightarrow \eq_{\ell(\calZ_{k})},
\end{split}
\]
such that the following properties are satisfied:
\begin{enumerate}
\item The images of $\stdFormB^+, \stdFormB^-, \stdFormB[\calZ_{k}]$ are families of annuli $\imAnn^+, \imAnn^-, \imAnn[\calZ_{k}]$ not containing any zeroes of  $(\eq,\starD{\twistD}+\modD)$. The families of annuli bound families of unions of disks $\imDisk^+, \imDisk^-, \imDisk[\calZ_k]$ containing the nodal sections $\calQ_e^+, \calQ_e^-$ and the section $\calZ_k$, respectively.
\item The pullback of $\starD{\twistD}+\modD$ has standard form, that is
\[\begin{split}
(\stdFormB^+)^*&(\starD{\twistD}+\modD)=  \scl[\ell(e+)]\left( z^{\kappa_e}-\res_{q_e^-}(\starD{\twistD}+\modD)\right ) \dfrac{dz}{z},\\
(\stdFormB^-)^*&(\starD{\twistD}+\modD)= -\scl[\ell(e-)]\left( z^{-\kappa_e}+\res_{q_e^-}(\starD{\twistD}+\modD)\right) \dfrac{dz}{z},\\
(\stdFormB[\calZ_{k}])^*&(\starD{\twistD}+\modD)= \scl[\ell(\calZ_{k})]z^{m_k} dz.
\end{split}
\]
\item The holomorphic maps $\stdFormB^+,\stdFormB^-, \stdFormB[\calZ_{k}]$ agree with the corresponding maps $\stdFormA^+,\stdFormA^-, \stdFormA[\calZ_{k}]$ on the most degenerate boundary stratum, i.e. on $\left(U\times (0,\dots,0)\right) \times A_{\delta/R,\delta}.$
\end{enumerate}
\end{theorem}

\begin{definition}\label{def:phicoord}
By a slight abuse of notation, we refer to  the family of coordinates given by $\stdFormA^+$ as $\stdFormA^+$-coordinates, and similarly  for $\stdFormA^-,\stdFormA[h],\stdFormB^+,\stdFormB^-,\stdFormB[\calZ_{k}]$.
If we do not want to specify whether we refer to a preimage of a node or a marked point, we simply write $\stdFormA[]$ or $\stdFormB[]$.
\end{definition}

The maps $\stdFormB^+,\stdFormB^-, \stdFormB[\calZ_{k}]$ are not determined uniquely. Following, \cite{BCGGMsm}, note that the maps can be specified uniquely by choosing base points near the marked zeroes and nodes.
We choose sections $\Nearby^+,\Nearby^-, \Nearby[\calZ_{k}]:B\rightarrow \universal$ such that the image is contained in a chart centered at $q_e^+,q_e^-$ and $Z_k$ respectively. Fix $p_0:=\delta/\sqrt{R}\in A_{\delta/R,\delta}$ as the base point of the annulus.
Then, by {\cite[Thm. 4.1]{BCGGMsm}} there exist unique $\stdFormB^+,\stdFormB^-, \stdFormB[h]$ such that
\[\begin{split}
\stdFormB^+&(b,p_0)=\Nearby[e+](b),\\
\stdFormB^-&(b,p_0)=\Nearby[e-](b),\\
 \stdFormB[\calZ_{k}]&(b,p_0)= \Nearby[\calZ_{k}](b),
\end{split}
\]
for any $b\in U\times (0,\dots,0)$.

\begin{convention}\label{setup:Nearby}
From now we fix once and for all a choice of nearby sections $\Nearby[]$.
\end{convention}

\subsection{The plumbing construction}
\label{section:PlumbConstruction}
For each node $e$ and each marked zero $\calZ_{k}$ we define conformal isomorphisms $\gluing^{\pm}:\stdAnn^{\pm}\to \imAnn^{\pm}$ and $\gluing[\calZ_{k}]:\stdAnn[\calZ_{k}]\to \imAnn[\calZ_{k}]$ by
\[
\begin{split}
\gluing^{+}(b,u,v)&:=\stdFormB^+(b,u),\\
\gluing^{-}(b,u,v)&:=\stdFormB^-(b,v),\\
\gluing[\calZ_{k}](b,z)&:=\stdFormB[\calZ_{k}](b,z),\\
\end{split}
\]
We define $\universal$ to be the family obtained by removing the disks $\imDisk^{\pm},\,\imDisk[\calZ_{k}]$ from $\eq$ and attaching $\stdFix$ and $\diskh$ by identifying the $\stdAnn[]$- and $\imAnn[]$-annuli via the $\gluing[]$-gluing maps. Since $\starD{\twistD}+\modD$ and $\Omega^{\pm}, \Omega_{\calZ_k}$ are identified via $\gluing[]$, the family $\universal$ inherits a relative one-form $\omega$.

We denote by $\imDisk[]$ the union of $\imDisk^+,\imDisk^-,\imDisk[\calZ_{k}]$ over all nodes and marked points and similarly by $\wideDisk[]$ the union of the families of disks $B_{\varepsilon}\times A_{\delta/R,\delta}\subseteq \universal$ over all marked points and nodes. The families of disks $\imDisk[]$ and $\wideDisk[]$ are exactly the families of disks from~\Cref{section:UniversalFamily}, $(3)$.

We have thus locally described the universal family $(\universal\to B,\omega)$. It will be needed to compare the periods of $\omega$ and the limit multi-scale differential $\twistD$ in~\Cref{thm:PerThm}. We will in particular need the particular form of $\stdFormB[]$-coordinates to analyze what happens in a neighborhood of the nodes.

\section{Level filtrations}\label{section:Filtr} In this section we introduce various notions of level for paths and homology classes. This will be necessary since the asymptotics of periods $\int_{\pa} \omega$ are governed by the level of $\pa$.

We now introduce the setup for the rest of this section. We let $\Sigma$ be a topological surface homeomorphic to surfaces in $\Stra$, which is obtained from a nodal Riemann surface using the plumbing construction from \Cref{section:ConstrUniversal}.
In \Cref{def:Van} we have defined the vanishing cycles on an annulus. Using the plumbing maps from \Cref{thm:Plumbing} we can pullback $\lambda_e$ from the annulus to $\Sigma$. By abuse of notation we denote the resulting curves again by $\lambda_e$.
We define $\Lambda:=\{\lambda_e,e\in E\}\subseteq \Sigma$ considered as a multicurve and then topologically a stable curve in $\bdComp$ is obtained by pinching $\Lambda$. As before we fix a stable curve $X\in \bdComp$.
% We usually write $\Lambda=\{\lambda_e, e\in E\}$ and call the elements $\lambda_e$ vanishing cycles.

\subsection{Thickenings of vanishing cycles}
For each vanishing cycle $\lambda_e$, let  $\lambda_e^{\circ}$ be a small open neighborhood of $\lambda_e$ that deformation retracts onto $\lambda_e$, and denote by $\Lambda^{\circ}\subseteq \Sigma$ the union of all such thickenings.

We decompose
\[
\Lambda=\Lambda^{ver}\sqcup \Lambda^{hor}
\]
into vanishing cycles corresponding to vertical and horizontal nodes, respectively, and further decompose
\[
\Lambda^{ver}=\sqcup_{i\in \lvlset} \Lambda^{ver}_{(i)},\quad \Lambda^{hor}= \sqcup_{i\in\lvlset} \Lambda_{(i)}^{hor}
\]
where \[
\begin{split}
\Lambda^{ver}_{(i)}&:=\{\lambda_e\,:\, e\in E^{ver},\, \ell(e+)=i,\, \ell(e-)<i\},\quad\\ \Lambda_{(i)}^{hor}&:=\{\lambda_e\,:\, e\in E^{hor},\, \ell(e+)=\ell(e-)=i\}.
\end{split}
\]
In words, $\Lambda^{ver}_{(i)}$ consists of  vertical vanishing cycles connecting $\Sigma_{(i)}$ to lower levels and $\Lambda^{hor}_{(i)}$ consists of horizontal vanishing cycles contained in $\Sigma_{(i)}$.
%We define the level $\ell(\lambda)$ of a vanishing cycles as $\ell(\lambda_e):=\ell(e+)$.

For each edge $e\in E$ the boundary $\partial \lambda_e^{\circ}$ consists of two boundary circles $\lambda_e^+ \sqcup \lambda_e^-$ with $\lambda_e^+\subset \Sigma_{(\ell(e+))}$ and $\lambda_e^-\subset\Sigma_{(\ell(e-))}$. At horizontal nodes we randomly choose which boundary component is denoted $\lambda_e^+$. We need to be careful about choosing orientations for $\lambda_e^{\pm}$. Our convention is that $\lambda_{e}^+$ has the same orientation as $\lambda_e$, while $\lambda_e^-$ has the opposite orientation.
We write
\[
\begin{split}
\Lambda^+&:=\{\lambda_e^+\,:\, e\in E^{ver}\},\\
\Lambda^-&:=\{\lambda_e^-\,:\,e\in E^{ver}\}\sqcup \{\lambda_e^{\pm}\,:\, e\in E^{hor}\}.
\end{split}
\]

We define analogues of $\Lambda_{(i)}$ and $\Lambda^{ver}$ for $\Lambda^{\circ}$ and $\Lambda^{\pm}$. For example $\Lambda^{+,ver}_{(i)}$ consists of $\lambda^{+}_e$ for all vertical nodes $e$ with $\ell(e+)=i, \ell(e-)<i$  and
$\Lambda^{\circ,ver}_{(> i)}$ consists of  $\lambda^{\circ}_e$ for all vertical nodes with $\ell(e+)> i, \ell(e-)\leq i$.

\Cref{fig:VanishingCycles} illustrates the definitions.

\begin{figure}[h]
\includegraphics[width=0.3\columnwidth,scale=0.3]{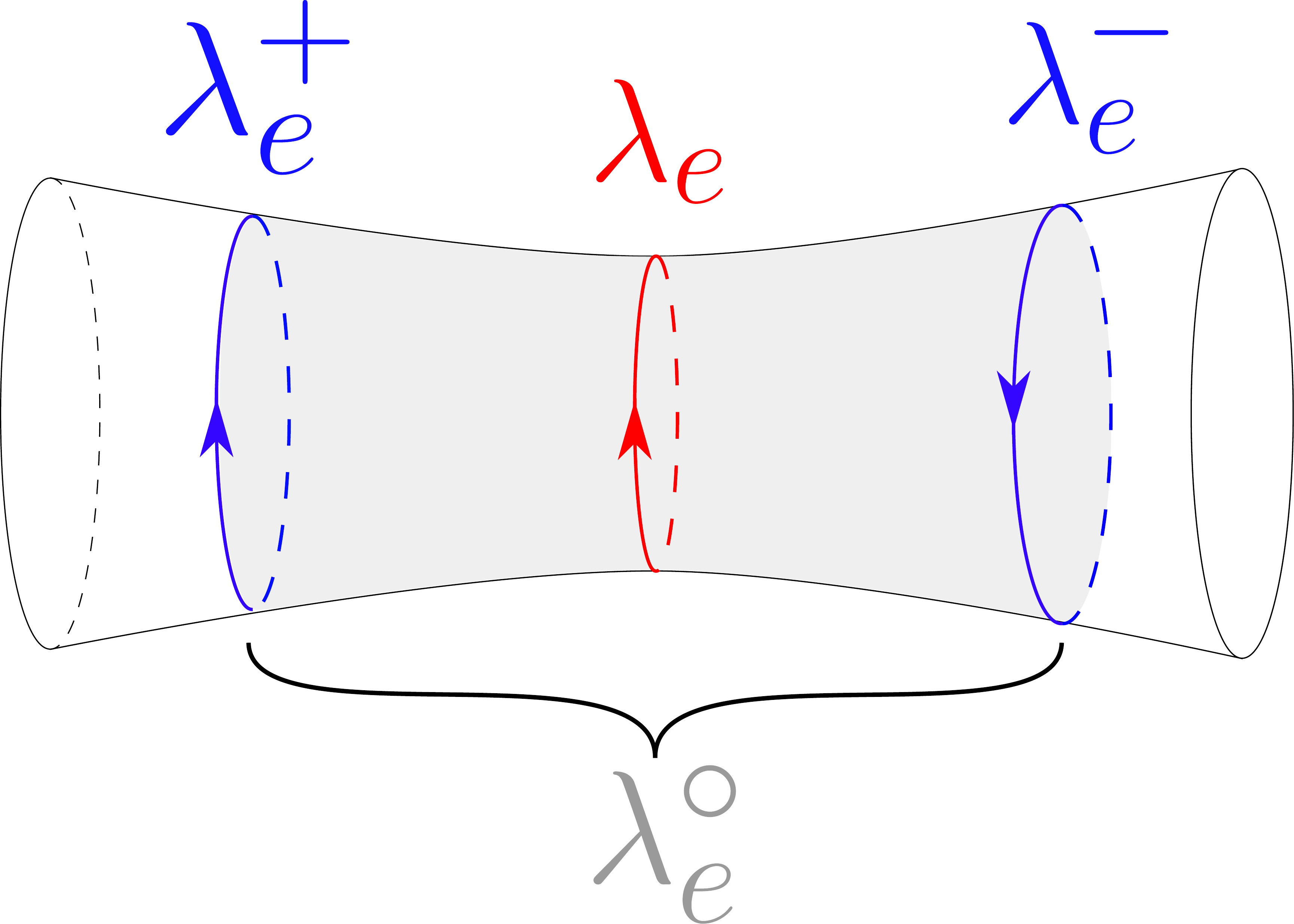}
\caption{Thickenings of vanishing cycles}
\label{fig:VanishingCycles}
\end{figure}

We now introduce several different ways of filtering the Riemann surface $\Sigma$ by levels.
If we remove all vertical vanishing cycles from $X$, then the remaining surface decomposes into a disjoint union of surfaces of a fixed level, i.e.
\[
\Sigma\setminus \Lambda^{\circ,ver} =\sqcup_{i\in\lvlset} \Sigma_{(i)}.
\]
Each of the resulting subcurves $\Sigma_{(i)}$ is a connected compact surface potentially with boundary.
We note that we remove $\Lambda^{\circ,ver}$ instead of just $\Lambda^{ver}$ since we want the result to be compact.
If instead of removing all vertical vanishing cycles, we only remove vertical vanishing cycles that cross the $i$-th level transition, i.e. exactly the vanishing cycles in $ \Lambda^{\circ,ver}_{(\geq i)}$, then we decompose $\Sigma$ into two surfaces: The part of $\Sigma$ that is at least of level $i$ and the part of $\Sigma$ below level $i$. We write 
\[
\Sigma \setminus \Lambda^{\circ,ver}_{(\geq i)}=: \Sigma_{(\geq i)} \sqcup \Sigma_{(< i)}.
\]

So far we have only removed the vertical vanishing cycles but later we will also need to remove the horizontal vanishing cycles.
We thus define
\begin{gather}
%\Sigma \setminus \Lambda^{\circ,ver}_{(\geq i)}=: \Sigma_{(\geq i)} \sqcup \Sigma_{(< i)}, \\
\Sigma_{(i)}^{cut}:= \Sigma_{(i)}\setminus \Lambda^{\circ,hor}
\end{gather}
to be the surface obtained from $\Sigma_{(i)}$ by cutting along all horizontal vanishing cycles.
Note that the resulting surface is usually a disconnected compact surface with boundary.
%
%We note that $\Sigma_{(i)}$ and $\Sigma_{(i)}^{cut}$ are compact Riemann surfaces with (potentially empty) boundary,
 See \Cref{fig:SigmaFilter} for an example illustrating the different ways of filtering $\Sigma$.

\begin{figure}[h]
\centering
\includegraphics[width=0.8\columnwidth]{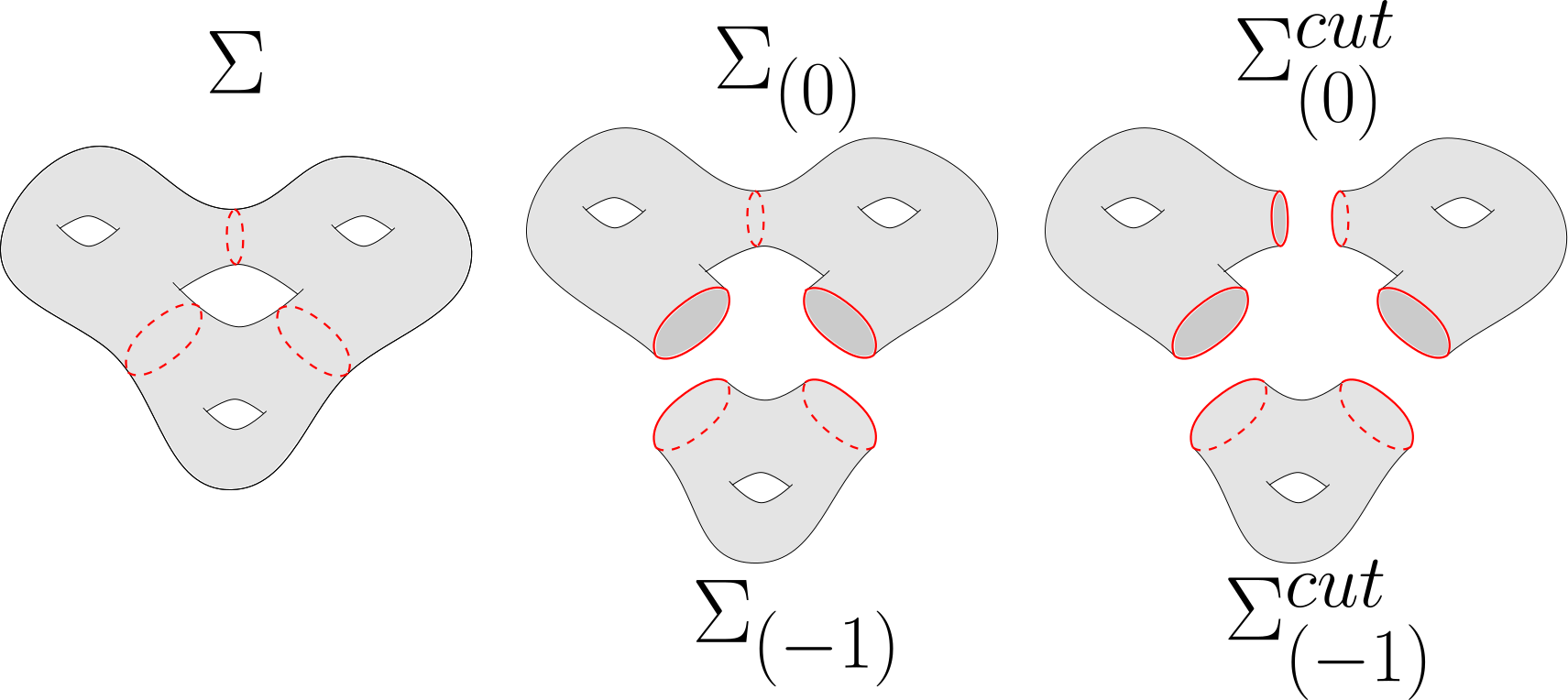}
\caption{The different ways of filtering $\Sigma$ by level}
\label{fig:SigmaFilter}
\end{figure}

\subsection{Global residue condition revisited}\label{sec:GRC}
We later want to compare period coordinates on $\Stra$ and generalized period coordinates on $\bdComp$. Rephrased in the terminology of this section, we have
\[
H^1_{(i)}(X;\CC)= H^1\!\left(\Sigma_{(i)}^{cut}\!\setminus\! P,Z\cup \Lambda_{(i)}^{+,ver};\CC\right);
\] 
where we recall $\Hi(X)$ from \Cref{eq:HiX}. Thus local coordinates on  $\bdComp$ and  $\Stra$  are given by  ~$\oplus_{i\in\lvlset} H^1(\Sigma_{(i)}^{cut}\!\setminus\! P,Z\cup \Lambda_{(i)}^{+,ver};\ZZ)^{\GRC}$ and ~$H^1(\Sigma\!\setminus\! Z,P;\CC)$, respectively.

Instead of working with cohomology we phrase everything in this section in terms of homology, where we think of a linear subspace of cohomology as being the annihilator of a subspace in homology. For the rest of the section we follow the convention that, unless stated otherwise, homology is taken with $\ZZ$-coefficients.
We first  need to setup various spaces modeling residue conditions on multi-scale differentials.
Let $Y$ be a connected component of $\resg{X}$ with $P\cap Y=\emptyset$, and denote $\{e_1,\ldots,e_b\}$ the set of all nodes where $Y$ intersects $\reseq{X}$.
Then we define
\[
\lambda_Y:= \sum_{k=1}^b \lambda_{e_k}\in \Hver.
\]
and we denote $\GRC^{ver}_{(i)}\subseteq \Hver$ the linear span
 \begin{equation}\label{eq:GRCver}
 \GRC^{ver}_{(i)}:=\langle \lambda_Y, Y \text{ a connected component of $\resg{X}$ with $P\cap Y=\emptyset$}\rangle_{\CC}.
 \end{equation}
 In words, $\GRC_{(i)}^{ver}$ is the span of all the equations defining the global residue conditions of level $i$.
We stress that this does not include the matching residue conditions at horizontal nodes at level $i$.
We analogously define $\GRC_{(i)}^{ver,cut}\subseteq \Hcut$ defined by the same cycles $\lambda_Y$, now considered as elements of
 $\Hcut$.

To include that matching residue condition at horizontal nodes, we let \[
 \MRH_{(i)}:=\langle \lambda^+-\lambda^-, \lambda\in \Lambda^{hor}_{(i)}\rangle_{\CC}\subseteq \Hcut.
 \]
 and finally we define
\[
\begin{split}
\GRC_{(i)}&:=\GRC_{(i)}^{ver,cut} + \MRH_{(i)}
\subseteq \Hcut,\\
\GRC&:= \bigoplus_{i\in\lvlset} \GRC_{(i)}.
\end{split}
\]

In particular $\GRC_{}$ consists exactly of all  global residue equations  including the matching residue condition at horizontal nodes.
We obtain the following description for $H^1\!\left(\Sigma_{(i)}^{cut}\setminus P,Z\cup \Lambda_{(i)}^{+,ver}\right)^{GRC}$, which is exactly the subspace in cohomology satisfying all global residue and matching residue conditions:
\begin{equation}\label{eq:GRC}
\begin{split}
H^1\!\left(\Sigma_{(i)}^{cut}\setminus P,Z\cup \Lambda_{(i)}^{+,ver}\right)^{GRC}
&=\ANN\left(\GRC_{(i)}\right)\\
&\simeq \left(H_1\!\left(\Sigma^{cut}_{(i)}\setminus P,Z\cup \Lambda_{(i)}^{+,ver}\right)/\GRC_{(i)}\right)^*.
\end{split}
\end{equation}

\subsection{Level and vertical filtration}\label{sec:Filtr}
In this section we define the concept of level for homology classes.
Additionally, we introduce two filtrations $\LVLF[\bullet]$ and $\GRCF[\bullet]$ of $H_1(\Sigma\setminus P,Z)$.
Roughly speaking, $\LVLF$ will consist of all cycles which can be represented by paths supported in the subsurface  $\resleq{\Sigma}$ of level at most $i$, while cycles $\GRCF\subseteq \LVLF$ additionally can be represented by paths disjoint from the horizontal vanishing cycles of level $i$.

The motivation for introducing $\GRCF$ is that it will come with a surjective linear map
\[
f_i:W_i\to \Hcut/\GRC_{(i)},
\]
the {\em specialization morphism},
with kernel $\LVLF[i-1]$, and we will thus have
\[
\bigoplus_{i\in\lvlset} \GRCF/ \LVLF[i-1]\simeq \bigoplus_{i\in\lvlset} \Hcut/\GRC_{(i)}.
\]
This will allow us to compare the local coordinates on~$\Stra$ and on~$\bdComp$.

We start by describing the specialization morphism in words. Any class $[\gamma]$ in $W_i$ can be represented by a collection of smooth curves in $\Sigma_{(\leq i)}$ which are all disjoint from the horizontal vanishing cycles. Suppose for simplicity that $[\gamma]$ can be represented by a single curve $\alpha$ in $\Sigma_{(\leq i)}$, disjoint from the vanishing cycles. We then restrict $\alpha$ to the subsurface $\Sigma_{(i)}$ of level $i$, i.e. we remove all the parts of $\alpha$ that go into lower levels. The result is a path $\alpha'$ in $\Sigma_{(i)}$ but since $\alpha$ is disjoint from all horizontal vanishing cycles, $\alpha'$ is actually contained in $\Sigma_{(i)}^{cut}$.
 We then define
 \[
f_i([\gamma])=\alpha'. 
 \]
The rest of the section is concerned with making the definition of $f_i$ precise and showing that it is well-defined in homology. The specialization morphism is illustrated in \Cref{fig:Specialization}.
\begin{figure}[h]
\centering
\includegraphics[width=0.6\columnwidth]{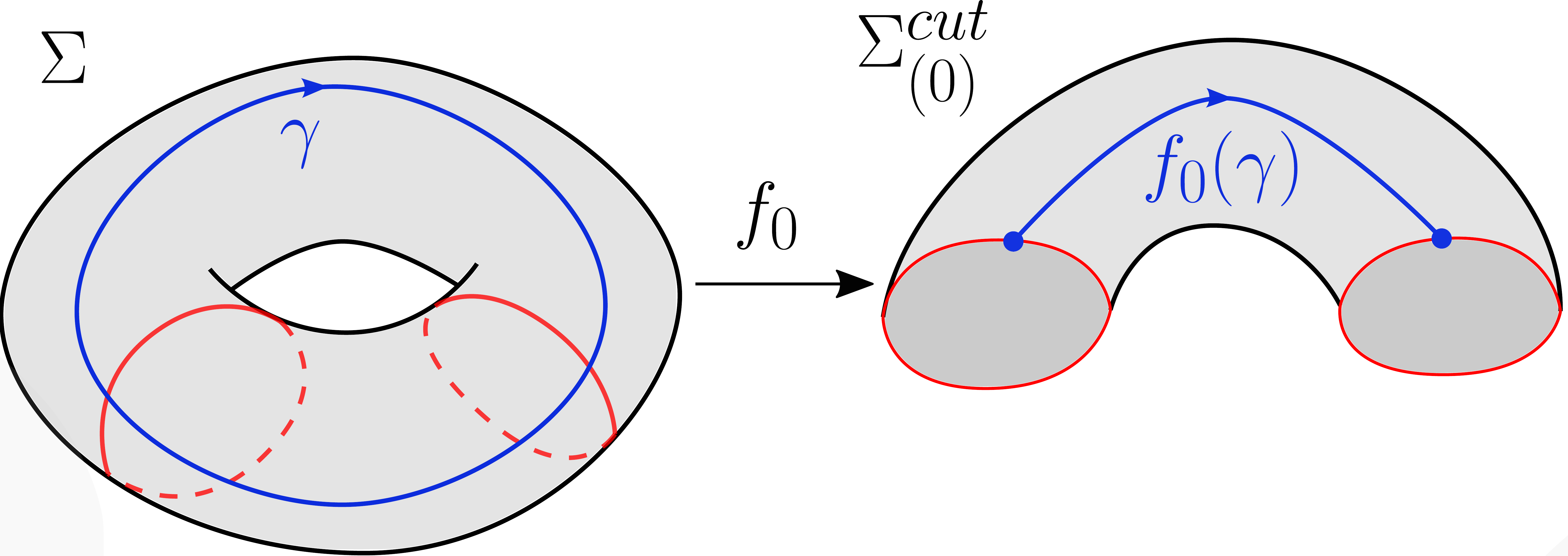}
\caption{The specialization morphism $f_0$}
\label{fig:Specialization}
\end{figure}

To simplify notation, for the rest of this section we adopt the convention that for $B,A\subseteq C$ the relative homology $H_1(B,A)$ is always to be understood as $H_1(B,A\cap B)$ (which is simply equal to $H_1(B,A)$ if $A\subseteq B$).
\begin{definition}
The inclusion $\Sigma_{(\leq i)}\subseteq \Sigma$ induces a map
\[
v_i:H_1(\Sigma_{(\leq i)}\!\!\setminus\! P,Z)\to H_1(\Sigma\!\setminus\! P,Z)\,,
\]
and we define the \textit{level filtration} $\LVLF[\bullet]$ by
\[
\LVLF:=\im(v_i)\subseteq H_1(\Sigma\!\setminus\! P,Z)\,.
\]

By naturality $v_{i-1}$ factors over the natural map
\[
H_1(\Sigma_{(\leq i-1)}\!\!\setminus\! P,Z)\to H_1(\Sigma_{(\leq i)}\!\!\setminus\! P,Z)\,,
\]
and thus $\LVLF[i-1]\subseteq\LVLF$. We say a cycle $\cycle\in H_1(\Sigma\!\setminus\! P,Z)$ is of \textit{top level} $i$ if $\cycle \in \LVLF\!\setminus\! \LVLF[i-1]$, and we then write $\topl(\cycle)=i$.

We let
\[(\Lambda_{(i)}^{hor})^{\perp}:=\{ \cycle \in H_1(\Sigma\setminus P,Z)\,|\, \langle \cycle, \lambda\rangle = 0\  \forall\ \lambda \in \Lambda_{(i)}^{hor}\}
\]
where $\langle \cdot,\cdot\rangle: H_1(\Sigma\setminus P,Z)\times H_1(\Sigma\setminus Z,P)\to \ZZ$ denotes the algebraic intersection pairing.
We then  define the \textit{vertical filtration} $\GRCF[\bullet]$ by
\[
\GRCF:= \LVLF \cap (\Lambda_{(i)}^{hor})^{\perp}\subseteq \LVLF.
\]
\end{definition}

By construction  every cycle in $\LVLF[i-1]$ can be represented by a collection of paths contained in $\resleq[i-1]{\Sigma}$. Since $\Sigma_{(\leq i-1)}$ is disjoint from all horizontal vanishing cycles of level $i$, we have in particular that
\[
\LVLF\supseteq \GRCF\supseteq \LVLF[i-1].
\]

\begin{example}
To demonstrate some of the features of the level and vertical filtration we consider the example from \Cref{fig:Specialization}.
Here $\lambda_1$ is a horizontal vanishing cycle and $\lambda_2$ and $\lambda_3$ are vertical vanishing cycles separating $\Sigma$ into $\Sigma_{(0)}$ and $\Sigma_{(-1)}$.
Since $\alpha$ is a path completely contained in $\Sigma_{(0)}$ and is not homologous to any path contained in $\Sigma_{(-1)}$ we have $[\alpha] \in L_{0}$. Furthermore $\alpha$ has zero intersection number with $\lambda_1$ and thus $[\alpha]$ is contained in the vertical filtration $W_{0}$.
On the other hand the path $\beta$ intersects $\Sigma_{(0)}$ but is homologous to a path completely contained in $\Sigma_{(-1)}$, thus $[\beta]\in L_{-1}$. Since there are no horizontal vanishing cycles in $\Sigma_{(-1)}$, the $-1$-th piece of the level filtration coincides with the $-1$-th piece of the vertical filtration, i.e. $L_{-1}=W_{-1}$.  Thus $[\beta]\in W_{-1}=L_{-1}$.
All three vanishing cycles $\lambda_1,\lambda_2$ and $\lambda_3$ are homologous (up to orientation), thus $\lambda_i\in L_{-1}$ for $i=1,2,3$. In \Cref{fig:Filtr}(B) we exhibit an explicit basis for each graded piece of the filtration $L_{-1}=W_{-1}\subseteq W_{0}\subseteq L_{0}$.
\begin{figure}[h]
\centering

\subfloat[Cautionary examples]{\label{fig:FiltrA}
\includegraphics[scale=0.4]{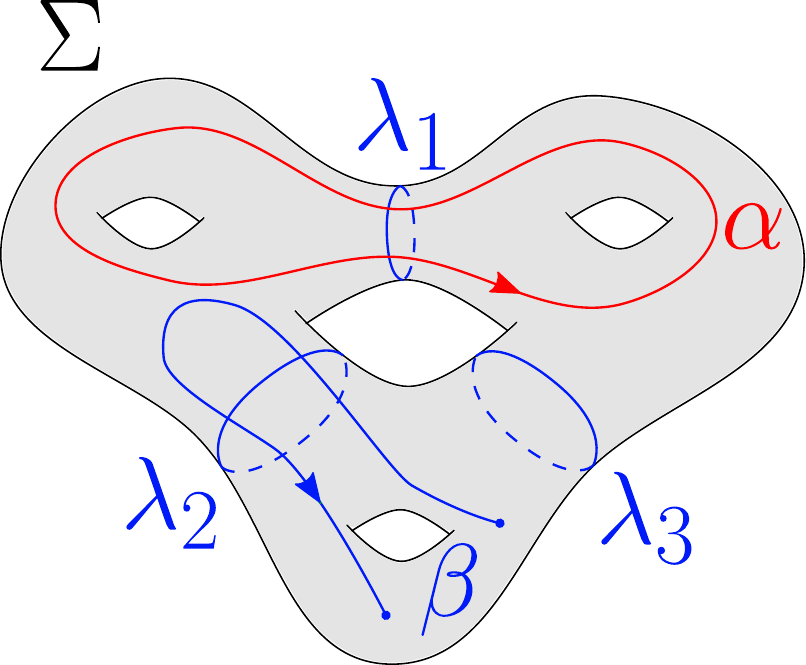}}
\qquad
\subfloat[A basis for $L_{\bullet}$ and $W_{\bullet}$]{\label{fig:FiltrB}
\includegraphics[scale=0.4]{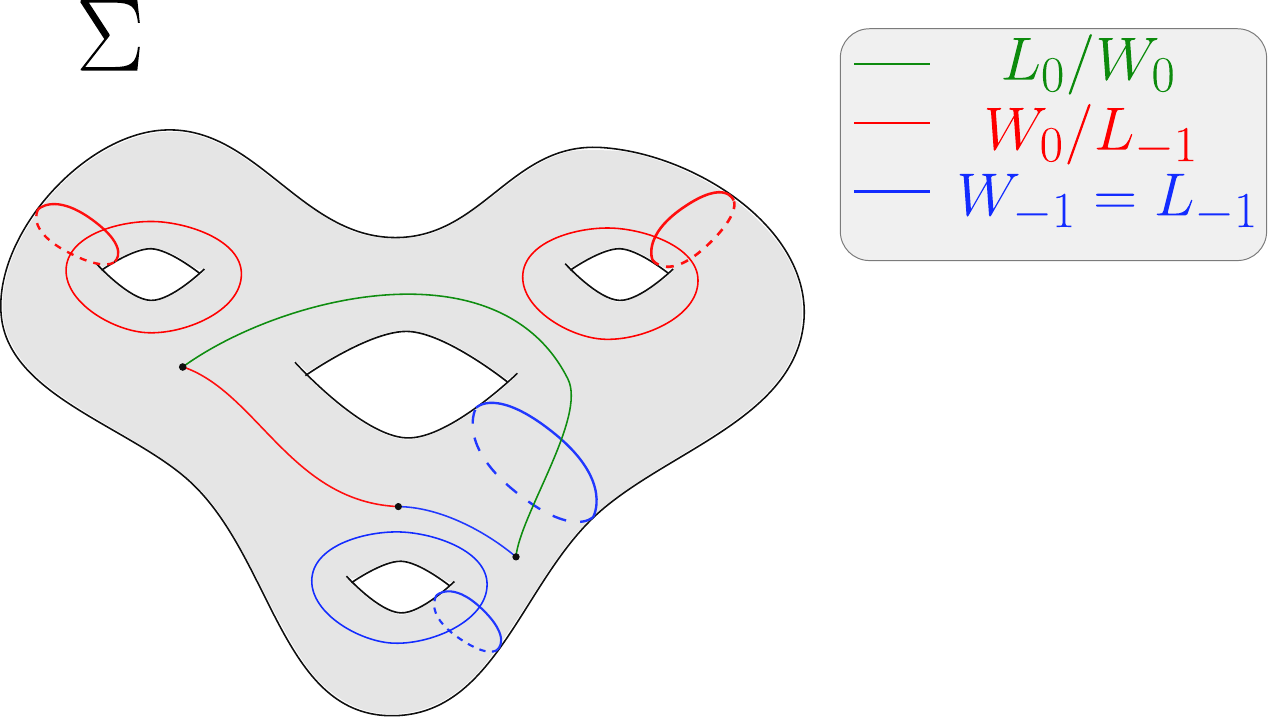}
}
\caption{The level filtration $L_{\bullet}$ and the vertical filtration $W_{\bullet}$}
\label{fig:Filtr}
\end{figure}
%\begin{figure}
%\centering
%\includegraphics[width=0.4\columnwidth]{figures/LevelFiltration.pdf}
%\caption{Cautionary examples for the level and vertical filtration}
%\label{fig:LevelFiltration}
%\end{figure}
\end{example}

\subsection*{The specialization morphism} Our goal is now to make the construction of the {\em specialization morphism}

\[
f_i:\GRCF\to \Hcut/\GRC_{(i)}
\]
precise. As a technical step we first construct an auxiliary map

\[
g_i:\LVLF\to \Hver
\]
which is defined on the whole level filtration $L_i$ and not just the vertical filtration $W_i$.
Both $f_i$ and $g_i$ are basically restriction maps: Given a path of level $i$, the map $g_i$ simply restricts the path to the $i$-th level, i.e. the result is a path in $\Sigma_{(i)}$. The map $f_i$ is similar, the difference is that if a path is disjoint from all horizontal vanishing cycles, then its restriction is actually contained in $\Sigma_{(i)}^{cut}=\Sigma_{(i)}\setminus \Lambda_{(i)}^{hor}$. There is an ambiguity in how a cycle in $W_i$ can be represented by a collection of curves disjoint from the vanishing cycles. It turns out that the ambiguity is an element of ${\GRC_{(i)}}$ and thus $f_i$ will give a well-defined map to  $\Hcut/\GRC_{(i)}$ and not to $\Hcut$.

We now start the description of $f_i$ and $g_i$.
We first describe ~$f_i$ and ~$g_i$ on the level of paths and then show that these give well-defined maps on homology afterwards.

Let $\cycle\in \LVLF$. By the definition of the level filtration we can write $\cycle=\sum_{k} a_k\pa_k$ where $\pa_k$ are simple smooth curves contained in $\resleq{\Sigma}$. For each $\pa_k$ we let $\pa_k':=(\pa_k)_{|\reseq{\Sigma}}$ be the restriction to $\Sigma_{(i)}$ considered as relative cycles with boundaries in $Z\cup \Lambda_{(i)}^{ver,+}$ and then define
\[
g_i(\cycle):= \sum_{k} a_k[\pa_k']\in \Hver.
\]

Now that we have constructed $g_i$ we need to define an auxiliary map $h_i$, and afterwards we will define $f_i$ as the composition of $h_i$ and $g_i$. Set \[
\widetilde{W}_i:=(\Lambda_{(i)}^{hor} )^{\perp}\subseteq \Hver
\]
Note that in particular $g_i(\GRCF)\subseteq \widetilde{W}_i$.
We are now going to define a map
\[
h_i:\widetilde{W}_i\to  \Hcut/\GRC_{(i)}
\]
as follows.

Let $\cycle\in \widetilde{W}_i$.
Write $\cycle=\sum_{k}c_k\alpha_k$ as a sum of smooth simple curves. Since the intersection number with any horizontal vanishing cycle is zero, we can make the collection of curves $\{\alpha_k\}$ disjoint from any horizontal vanishing cycle of level $i$ by a series of {\em band moves}, as depicted in \Cref{fig:Bandmove} or \cite{JohnsonSpin}. Thus we can write $\cycle=\sum_{k}c_k\alpha_k''$ where $\alpha_k''$ is a collection of smooth simple curves in $\Sigma_{(i)}^{cut}$.
We then define
\[
h_i(\cycle):=\sum_{k}c_k[\alpha_k'']\in \Hcut/\GRC_{(i)}.
\]
and finally set
\[
f_i:=h_i\circ {g_i}_{|\GRCF}.
\]

\begin{figure}[h]
\centering
\includegraphics[width=0.4\columnwidth,scale=0.3]{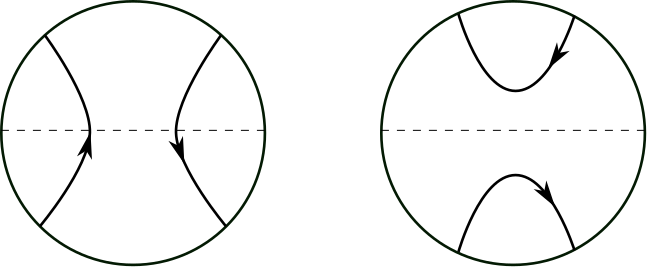}
\caption{A band move}
\label{fig:Bandmove}
\end{figure}

The maps $f_i$ are not well-defined as maps to $\Hcut$, since different choices of band moves can differ by multiples of the vanishing cycles, as seen in \Cref{fig:Ambiguity}.

\begin{figure}[H]
\includegraphics[width=0.7\columnwidth]{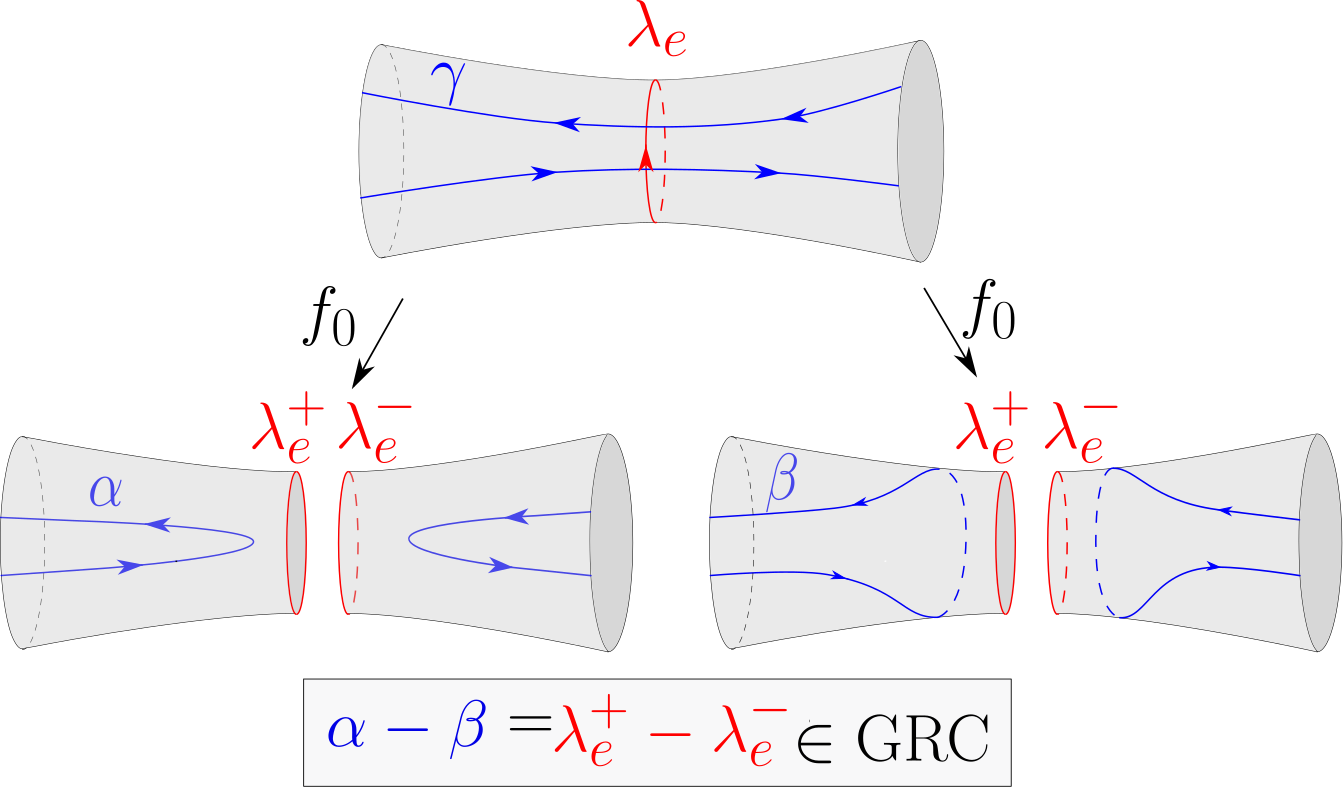}
\caption{The ambiguity of band moves}
\label{fig:Ambiguity}
\end{figure}

But we will see that $f_i$ is well-defined as a map to $\Hcut/\GRC_{(i)}$. See \Cref{fig:fiEx} for an illustration of the map $f_i$.

\begin{figure}[H]
\centering
\includegraphics[width=0.65\columnwidth,scale=0.38]{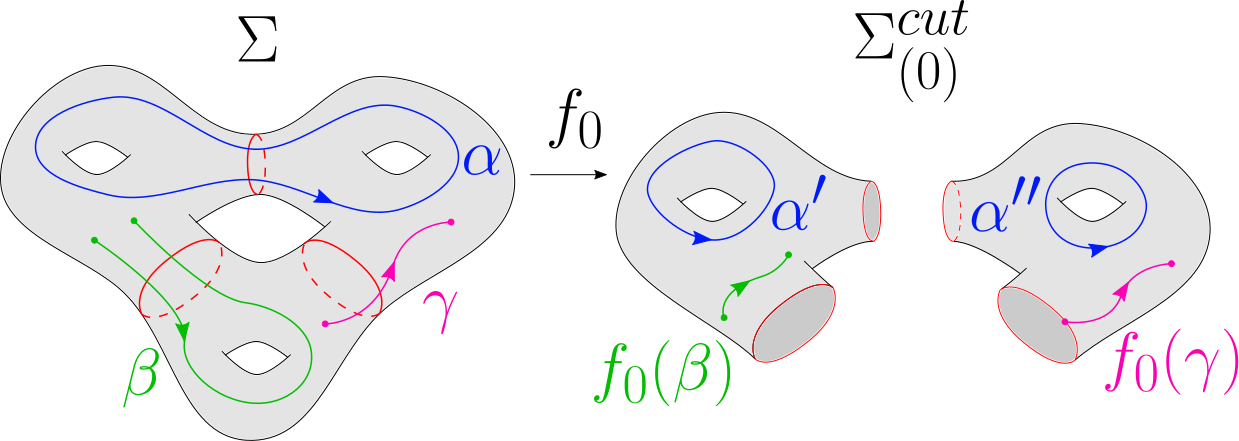}
\caption{The map $f_0$}
\label{fig:fiEx}
\end{figure}

\begin{proposition}
The linear maps 
\begin{gather*}
g_i:\LVLF\to\Hver,\\
f_i:\GRCF\to\Hcut/\GRC_{(i)}
\end{gather*} are well-defined and surjective.
Furthermore
\[
\ker f_i=\ker g_i=\LVLF[i-1].
\]
\end{proposition}

\begin{proof}
We start with the map $g_i$.
%Let $\cycle\in\LVLF$ and write $\cycle=\sum_{k} c_k\alpha_k=\sum_{l} d_l\beta_l$ as a sum of smooth simple curves in two different ways. We need to show that $\sum_{k} c_k\alpha_k'=\sum_{l} d_l\beta_l'$ where $\alpha'$ and $\beta'$ denote the restrictions of $\alpha$ and $\beta$ to $\reseq{\Sigma}$, respectively. 
From the long exact sequence of the triple
\[
Z \subseteq (\resleq[i-1]{\Sigma}\setminus P)\cup Z \subseteq (\resleq{\Sigma}\setminus P)\cup Z,
\]
 we obtain
an exact sequence
 \[
 \begin{tikzcd}
 H_1(\resleq[i-1]{\Sigma}\setminus P,Z)\ar[r,"\nu_i"] &  H_1(\resleq[i]{\Sigma}\setminus P,Z) \ar[r] &  H_1(\resleq[i]{\Sigma}\setminus P, \left(\resleq[i-1]{\Sigma}\setminus P\right) \cup Z)\ar[d,"\simeq"] \\
 & & H_1(\reseq[i]{\Sigma}\setminus P, \Lambda_{(i)}^{ver,+}\cup Z)
 \end{tikzcd}
 \]
 where the vertical isomorphism is induced by excising $\resleq[i-1]{\Sigma}\setminus \Lambda_{(i)}^{ver}$.

 Since the excision map is defined via barycentric subdivision, it follows that for a simple smooth curve $\alpha$ the composition
 \[
 H_1(\resleq[i]{\Sigma}\setminus P,Z) \to H_1(\reseq[i]{\Sigma}\setminus P, \Lambda_{(i)}^{ver,+}\cup Z)
 \]
 is given by the restriction  $\alpha'$ to $\reseq{\Sigma}$ and thus the map coincides with $g_i$. From the exact sequence we then obtain that $\ker g_i=\LVLF[i-1]$.
 To see that $g_i$ is surjective, take any smooth closed curve $\pa$ representing a class in $\Hver$. By \cite[Lemma 3.9]{BCGGMgrc} we can connect the boundary points of $\pa$ in $\Lambda_{(i)}^{ver,+}$ to marked zeros in $\resleq{\Sigma}$ by only passing through levels below $i$ and thus creating a $g_i$-preimage for $\pa$.
 This proves all claims about $g_i$ and it remains to prove the analogous statements for $f_i$.

To show that $f_i$ is well-defined, it is enough to show that $h_i$ is well-defined.
Let $\cycle\in \widetilde{W}_i$ and represent $\cycle=\sum_{k} c_k\alpha_k= \sum_{l} d_l\beta_l \in \Hver$
 in two different ways by collections of smooth simple curves contained in $\Sigma_{(i)}^{cut}$.
We let $\alpha= \sum_{k} c_k\alpha_k$ and $\beta =\sum_{l} d_l\beta_l$ be the associated relative homology classes in $\Hcut$.
  We want to show that $\alpha-\beta\in GRC_{(i)} $.

We will apply the relative version of Mayer-Vietoris.
We set \[
A:=\Sigma_{(i)}\setminus \left (\Lambda_{(i)}^{hor}\cup P\right)=\Sigma_{(i)}^{cut}\setminus P,\, B:= \Lambda_{(i)}^{hor,\circ}.
\] In particular we have $A\cap B=\Lambda_{(i)}^{hor,\pm}$ and $A\cup B=\reseq{\Sigma}\setminus P$. We need the following part of the Mayer-Vietoris sequence
\[
\begin{tikzcd}
H_1(\Lambda^{hor,\pm}_{(i)}) \ar[r,"{(\iota_*,\iota_*')}"] &\Hcut\oplus H_1(\Lambda^{hor,\circ}_{(i)})\ar[r,"k_*-l_*"] & \Hver
\end{tikzcd}
\]
where $\iota:A\cap B\to A,\iota':A\cap B\to B,k:A\to A\cup B,l:B\to A\cup B$ are the natural  maps induced from the inclusions.

By construction $(\alpha-\beta,0)\in \Hcut\oplus H_1(\Lambda^{hor,\circ}_{(i)})$ lies in the kernel of $k_*-l_*$ and thus in the image of $(\iota_*,\iota_*')$.
Note that $(\iota_*,\iota_*')(a\lambda^+ +b\lambda^-)=(a\lambda^+ +b\lambda^-, (a+b)\lambda)$.
We conclude that $\alpha-\beta= a(\lambda^+-\lambda^-)\in\GRC_{(i)}$ and thus $h_i$ and $f_i$ are well-defined.

Note that $h_i$ fits into a commutative diagram
%\[
%\begin{tikzcd}
%\widetilde{W}_i \ar[swap,dr,"h_i",end anchor={[xshift=-9ex]}, start anchor={[xshift=-0.5ex, yshift=-1ex]}]& &\ar[ll,swap,"p_i"] \Hcut \ar[dl,"q_i", end anchor={[xshift=8ex]}, start anchor={[yshift=-0.5ex, xshift=6ex]}]\\
%& \Hcut/\GRC_{(i)}
%\end{tikzcd}
%\]

\[
\begin{tikzcd}
\Hcut   \ar[d,swap,"q_i"]    \ar[rr,"p_i"] & &\widetilde{W}_i \ar[dll,"h_i", start anchor = {[xshift=1ex, yshift=-1.5ex]},end anchor={[xshift=7ex, yshift=-1ex]}]\\
 \Hcut/\GRC_{(i)} &&
\end{tikzcd}
\]
where $p_i$ is the natural map induced by the inclusion $\Sigma_{(i)}^{cut}\subseteq \Sigma_{(i)}$ and $q_i$ is the natural quotient map. Thus $h_i$ is surjective and $\ker h_i=p_i(\GRC_{(i)})$.

Since $g_i$ and $h_i$ are surjective it follows that $f_i=h_i\circ g_i$ is surjective. It remains to show that $\ker f_i=L_{i-1}$.

We define $G_{(i)}\subseteq \GRCF$ to be the subspace generated by $\lambda_Y$ as in the definition \cref{eq:GRCver}. Then $g_i(G_{(i)})=p_i(\GRC_{(i)})=
\ker h_i$ and thus $\ker f_i=G_{(i)} + \ker g_i= G_{(i)}+\LVLF[i-1]$.
We claim that $G_{(i)}\subseteq \LVLF[i-1]$. This can be seen as follows. Let $Y$ be a connected component of $\resg{X}$ with $P\cap Y=\emptyset$, and denote by $\{e_1,\ldots,e_b\}$ the set of all nodes where $Y$ intersects $\reseq{X}$ and additionally by $\{e_{b+1},\ldots,e_c\}$ the nodes where $Y$ intersects $\resl{X}$.
Then $\sum_{k=1}^{c}\lambda_{e_k}=0\in H_1(\Sigma\setminus P,Z)$ since this collection of vanishing cycles is separating and we thus have
\[
\lambda_Y= \sum_{k=1}^{b}\lambda_{e_k}= -\sum_{k=b+1}^{c}\lambda_{e_k}\in \LVLF[i-1].
\]
\end{proof}
\subsection{Top level}
So far we have defined the level of cycles $\cycle \in H_1(\Sigma\setminus P,Z)$ but it will be convenient to be able to talk about the level of paths. There has to be some care when comparing the level of a path and of its homology class.

\begin{definition}
For a collection of curves $\pa$ on $\Sigma$ we define its \textit{(top) level} to be the largest  $i$ such that ~$\pa\cap \Sigma_{(i)} \neq \emptyset$ and then write $\topl(\pa)=i$. Note that this is only well-defined as long as none of the curves are contained in $\Lambda^{\circ}$.
We  then define the {\em top level restriction} $\pa_{\topl}$ to be the intersection of $\pa$ with $\Sigma_{(\topl(\pa))}$.
By considering $\Sigma_{(\topl(\pa))}$ as a subsurface of $X$ we can also define the level of a collection of curves on the stable curve $X$.
In this case we define $\pa_{\topl}$ to be the restriction of $\pa$ to $X_{(\topl(\pa))}$.
\end{definition}

The following example shows that one has to be cautious when comparing the level of a path and a homology class.

\begin{example}[Tilted cherry]
\Cref{fig:TiltedCherry} depicts a smooth genus $3$ curve with vanishing cycles corresponding to a tilted cherry level graph.
\begin{figure}[H]
\centering
 \includegraphics[scale=0.22]{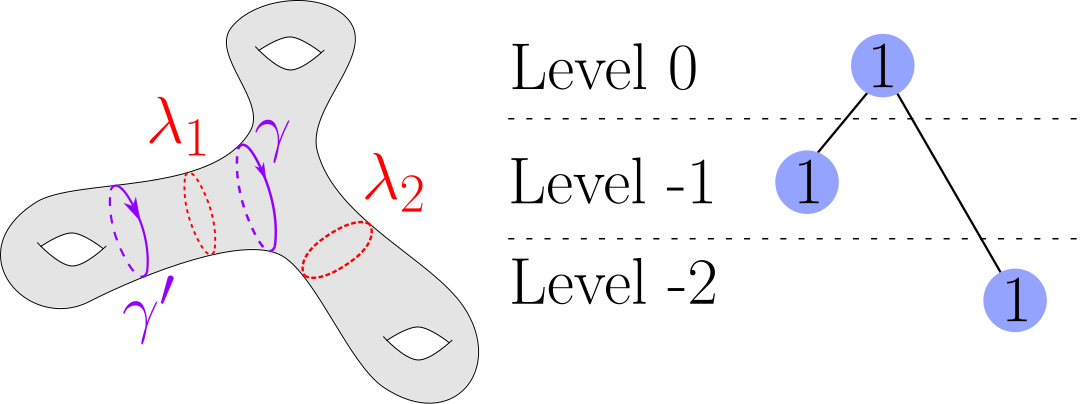}
 \caption{A smooth genus $3$ curve and the tilted cherry level graph}
 \label{fig:TiltedCherry}
\end{figure}
The two vanishing cycles $\lambda_1$ and $\lambda_2$ are homologous and thus $\topl([\lambda_1])=\topl([\lambda_2])=-2$.
On the other hand, both $\pa$ and $\pa'$ are simple closed curves representing $[\lambda_1]$, but ~$\topl(\pa)=0, \ \topl(\pa')=-1$.
\end{example}
The example shows that even if a cycle is represented by a simple curve, we cannot necessarily read off the level of a cycle from a path representing it. But since every cycle $\cycle\in \LVLF$ can be represented by a collection of paths supported on $\resleq{X}$ we have
\[
\topl(\cycle)=\inf\left\{ \topl(\pa)\, \, \text{ $\pa$ is a collection of simple smooth curves  representing of $\cycle$}\right\}.
\]

\subsection{An adapted homology basis}
\label{section:CohBasis}
In this section we construct a homology basis suited to analyzing linear equations. Roughly speaking, we only want to consider paths that cross different levels as little as possible. This will allow us later to compare the local coordinates on $\Stra$ and on $\bdComp$.
For the remainder of this section we let $\Sigma$ be a topological surface in $\Stra$ and $X$ a stable curve in $\bdComp$.

\begin{definition}\label{def:hcc}
We say a cycle $\cycle\in H_1(\Sigma\setminus P,Z)$ {\em crosses} a node $e\in E(\enhancG)$ if $\IP[\cycle,\lambda_e]\neq 0$.
A cycle $\cycle$ is called a {\em \hcc} if it crosses some horizontal node at level $\topl(\cycle)$
and {\em non-horizontal} if  is not a \hcc.
Note that non-horizontal cycles are allowed to cross horizontal nodes below top level.
Similarly, for $\cycle\in H_1(X\setminus P,Z)$ we say that $\cycle$ crosses $e$, is a \hcc or non-horizontal, if the same is true for some lift of $\cycle$ to  $H_1(\Sigma\setminus P,Z)$.
\end{definition}

If $\cycle\in H_1(\Sigma\setminus P,Z)$ has top level $i$ and is a horizontal crossing cycle, then $\cycle\in \LVLF\setminus \GRCF$. On the other hand if $\cycle$ is non-horizontal, then $\cycle\in \GRCF\setminus \LVLF[i-1]$.

\begin{example}
We consider the dual graph in \Cref{fig:CrossCurve} with two components of top level and three horizontal nodes. The diamond indicates a marked pole. The cycle $\pa_1$ crosses $e_1$, but $\pa_2$ is non-horizontal, since it can be deformed away from $e_2$.
\begin{figure}[H]
\centering
\includegraphics[scale=0.3]{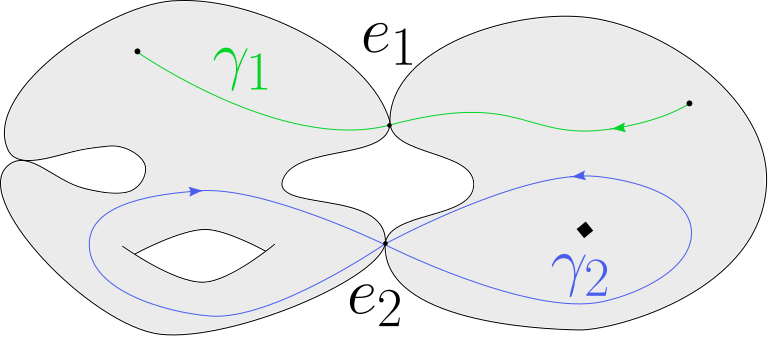}
\caption{Crossing and non-crossing curves}
\label{fig:CrossCurve}
\end{figure}
\end{example}

\begin{definition}
A basis $\{\pa_1,\ldots, \pa_n\}$ of $\Hsigma{}$ is called $\enhancG$-\textit{adapted} if
there exists a partition
\[
\{\pa_1,\ldots,\pa_n\}=\bigsqcup_{i\in\lvlset} \left(\left\{ \alpha_1^{(i)},\ldots,\alpha_{n(i)}^{(i)}\right\} \sqcup \left\{\delta_e^{(i)}, e\in E_{(i)}^{hor}\right\}\right)
\]
into horizontal-crossing cycles $\delta_e^{(i)}\in \LVLF\setminus \GRCF$ and non-horizontal cycles $\alpha_k^{(i)}\in \GRCF\setminus \LVLF[i-1]$
such that
\begin{gather*}
\LVLF= \langle \bigsqcup_{j \leq i}\, \{ \alpha_1^{(j)},\ldots,\alpha_{n(j)}^{(j)}\} \bigsqcup \{\delta_e^{(j)}, e\in E_{(j)}^{hor}\}\rangle_{\CC}; \\
\LVLF/\GRCF\simeq \langle \delta_e^{(i)}, e\in E_{(i)}^{hor}\rangle_{\CC};\\
\GRCF/\LVLF[i-1]\simeq \langle \alpha_1^{(i)},\ldots,\alpha_{n(i)}^{(i)}\rangle_{\CC};
%\\
%  \im\left( H_1(\Lambda_{(i)}^{hor})\to \Hcut\right )\simeq\langle \beta_1^{(i)},\ldots,\beta_{n_2(i)}^{(i)}\rangle_{\CC},
\end{gather*}
and additionally
\[
\IP[\delta_e^{(i)},\lambda_{e'}]=\begin{cases} 1 & \text{ if } e=e',\\
 0 & \text{otherwise,}
 \end{cases} \text{ for all } e\in E^{hor}_{(i)}, e'\in E^{hor}.
\]
\end{definition}

As a first remark, we note that the definition of $\enhancG$-adapted basis only depends on the level graph and not on the enhancement. The basic statement is the existence of $\enhancG$-adapted bases.
\begin{proposition}\label{prop:ExBasis}
For every enhanced level graph $\enhancG$ there exists a $\enhancG$-adapted homology basis.
\end{proposition}
\begin{proof}
We claim that the natural map
\[
\rho_i:\LVLF\to \CC^{|E_{(i)}^{hor}|}, \cycle\mapsto \left(\langle \cycle,\lambda_e\rangle\right)_{e\in E_{(i)}^{hor}}
\]
is surjective and thus
\[
\LVLF/\GRCF\simeq \CC^{|E_{(i)}^{hor}|}.
\]

Assuming this for now, the existence of a $\enhancG$-adapted basis can now be seen as follows.
We have the filtration
\[
\Hsigma\supseteq \LVLF[0]\supseteq \GRCF[0]\supseteq \ldots\supseteq  \LVLF[\ell(\enhancG)]\supseteq \GRCF[\ell(\enhancG)]
\]
with graded pieces
\[
\LVLF/\GRCF\simeq  \CC^{|E_{(i)}^{hor}|},\quad
\GRCF/\LVLF[i-1]\simeq\Hcut/\GRC_{(i)}.
\]
We are now going to construct a $\enhancG$-adapted basis inductively by lifting a basis from each graded piece.
We start by choosing a basis $\left \{ \tilde{\alpha}_1^{(i)},\ldots,\tilde{\alpha}_{n(i)}^{(i)}\right \}$ for $\Hcut/\GRC_{(i)}$ and then let $\alpha_k^{(i)}$ be preimages of $ \tilde{\alpha}_k^{(i)}$ under the specialization map $f_i$, respectively.
 Afterwards we let $\left\{\delta_e^{(i)}\right\}$ be $\rho_{i}$-preimages of the unit basis in $\CC^{|E^{hor}_{i}|}$.

It thus remains to prove the surjectivity of the map $\rho_i$. For this we construct explicit cycles $\delta_e^{(i)}$ with \[
\IP[\delta_e^{(i)},\lambda_{e'}]=\begin{cases} 1 & \text{ if } e=e'\\
 0 & \text{otherwise}
 \end{cases} \text{ for all } e'\in E^{hor}.
\]

We fix a horizontal node $e\in E_{(i)}^{hor}$.
If both $v(e+)$ and $v(e-)$ are local minima for the level order, then they contain a marked point which is not a pole or a preimage of any node, see \cite[Lemma 3.9]{BCGGMgrc}, and we can then connect these marked points by a path that goes through the node $e$ once and does not cross any other horizontal nodes.
On the other hand, if for example $v(e+)$ is not a local minimum then consider a path $\pa'$ in the dual graph connecting $v(e+)$ to a local minimum $v'$ by only passing through vertical edges connecting to levels below $i$. Then by the same argument as before $X_{v'}$ contains a marked point $P_+$ as above. We can run the same argument for $v(e-)$ and find a marked point $P_-$.
By embedding the dual graph into $X$ we can represent $\pa'$ by a path $\delta_e^{(i)}$ in $X$ connecting $P_+$ and $P_-$.
By construction $\delta_e^{(i)}$  intersects $\lambda_e$ once and is disjoint from all other horizontal vanishing cycles.
\end{proof}

The motivation for introducing $\enhancG$-adapted bases is that it allows to relate the coordinates on the boundary $\bdComp$ to the coordinates on the open stratum $\Stra$, which we recall are given by  $\bigoplus_{i\in\lvlset} H_{(i)}^{1}(X)^{\GRC}$ and $H^1(\Sigma\setminus P,Z)$, respectively.
We now fix a $\enhancG$-adapted basis $\{\gamma_1,\dots,\gamma_n\}$ once and for all.

\begin{example} We now illustrate this definition in some examples.~\Cref{fig:AdpBasisNon} depicts a level graph $\enhancG$ and two different homology bases for $H_1(X\setminus P,Z)$. The numbers inside the vertices denote the genus of the corresponding irreducible component of the stable curve. Note that the decorations $\kappa_e$ and the zero orders at the marked points are irrelevant for our discussion (since the notion of a $\enhancG$-adapted basis only depends on the level graph and not on the choice of an enhancement) and we thus omit them. The stable curves are degenerations of genus $3$ curves and the multi-scale differentials live in the codimension $4$ boundary stratum $\bdComp$.
The homology basis of~\Cref{fig:AdpBasisNon-b} provides an example of a $\enhancG$-adapted basis while the basis in~\Cref{fig:AdpBasisNon-c} violates the definition in two ways. Firstly,
every horizontal node is crossed by multiple basis elements and secondly, all paths have top level $0$, thus when restricting to the top level  they are linearly dependent. In particular, the top level restrictions together with the vanishing cycles generate $ H_1(\reseq[0]{X}\setminus \reseq[0]{P},\reseq[0]{Z})$ but fail to generate $H_1(\reseq[1]{X}\setminus \reseq[1]{P},\reseq[1]{Z})$.

\begin{figure}[h]
\centering
\subfloat{
         \centering
         \includegraphics[scale=0.18]{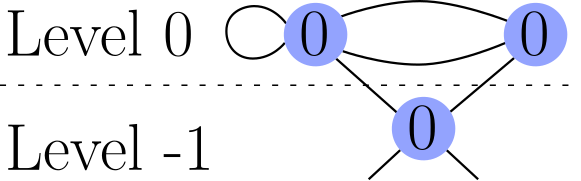}
     }
    \\
\subfloat[A $\enhancG$-adapted basis]{\label{fig:AdpBasisNon-b}
\includegraphics[scale=0.2]{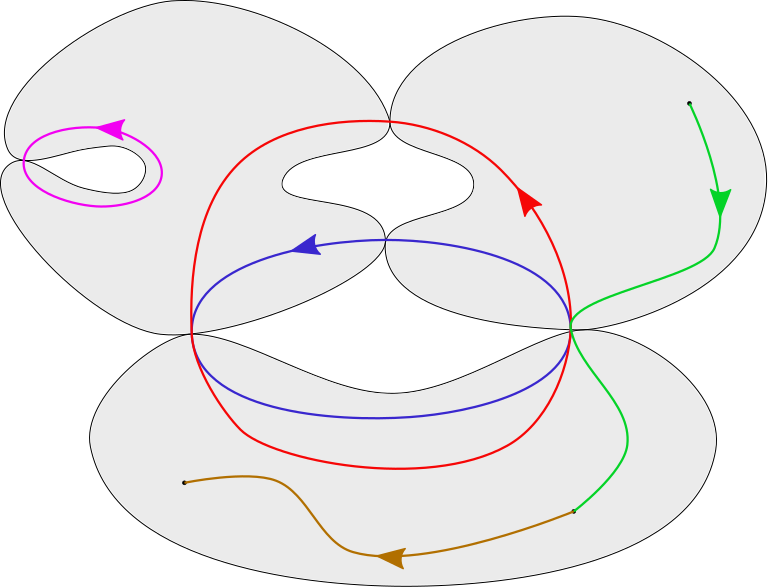}}
\qquad
\subfloat[Not a $\enhancG$-adapted basis]{\label{fig:AdpBasisNon-c}
\includegraphics[scale=0.2]{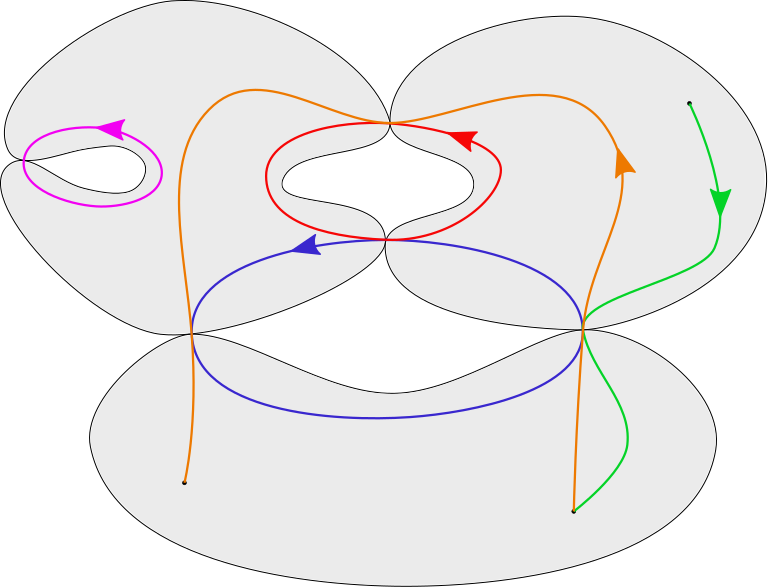}
}
\caption{A level graph $\enhancG$ and an example and a non-example of $\enhancG$-adapted bases}
\label{fig:AdpBasisNon}
\end{figure}

The second example, as depicted in~\Cref{fig:AdpBasisEx}, is a degeneration of a genus $7$ curve in a codimension $4$ boundary component of $\MSDS$. Bullets and diamonds represent marked zeroes and marked poles, respectively. We omit the orientation of paths. In this slightly more complicated example one can see all the features of an $\enhancG$-adapted basis.
First we start by choosing a basis for the vertical filtration $W_{-1}=\langle \alpha_1^{(-1)},\alpha_2^{(-1)}\rangle$  and extend this to a basis of $L_{-1}=\langle \delta_1^{(-1)},\delta_2^{(-1)}\rangle\oplus W_{-1}$ by adding paths that cross the horizontal nodes in level $-1$.
On the top level we have
\[
W_0=L_{-1}\oplus\langle \alpha^{(0)}_1,\ldots, \alpha^{(0)}_8\rangle.
\]
Here we started by choosing a symplectic basis $\{\alpha^{(0)}_1,\ldots, \alpha^{(0)}_6\}$ on each irreducible component and then extended it by paths encircling marked nodes, in this case only $\alpha_{7}^{(0)}$, as well as paths connecting marked zeros, in this case $\alpha^{(0)}_{8}$.
Finally for the vertical filtration $L_0$ one has to add paths crossing horizontal nodes in level $0$. In this case we have
\[
L_0=W_{0}\oplus \langle \delta_1^{(0)},\delta_2^{(0)},\delta_3^{(0)}\rangle.
\]
% First one starts by choosing a symplectic basis on each connected component of the normalization, in this case $\{\alpha_1,\beta_1,\ldots,\alpha_4,\beta_4\}$.
%Then one adds cycles encircling marked poles once, in this case only $\gamma_2$. Afterwards one starts to add paths connecting marked zeroes. First one chooses paths not passing through any horizontal nodes, only passing through higher levels if necessary, here given by $\{\delta_3,\delta_4,\gamma_1\}$. And finally one adds the cross cycles $\{\delta_1,\delta_2,\alpha_2\}$.

\begin{figure}[h]
\centering
\subfloat{
\includegraphics[scale=0.18]{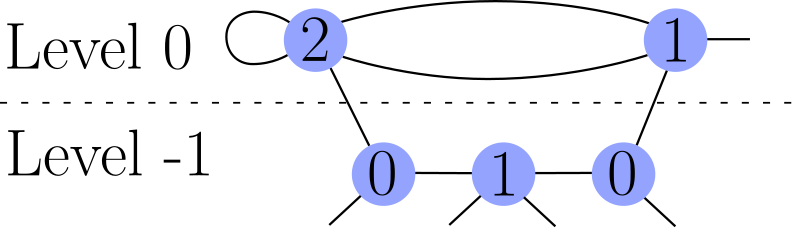}
}
\\
\subfloat{

\includegraphics[scale=0.35]{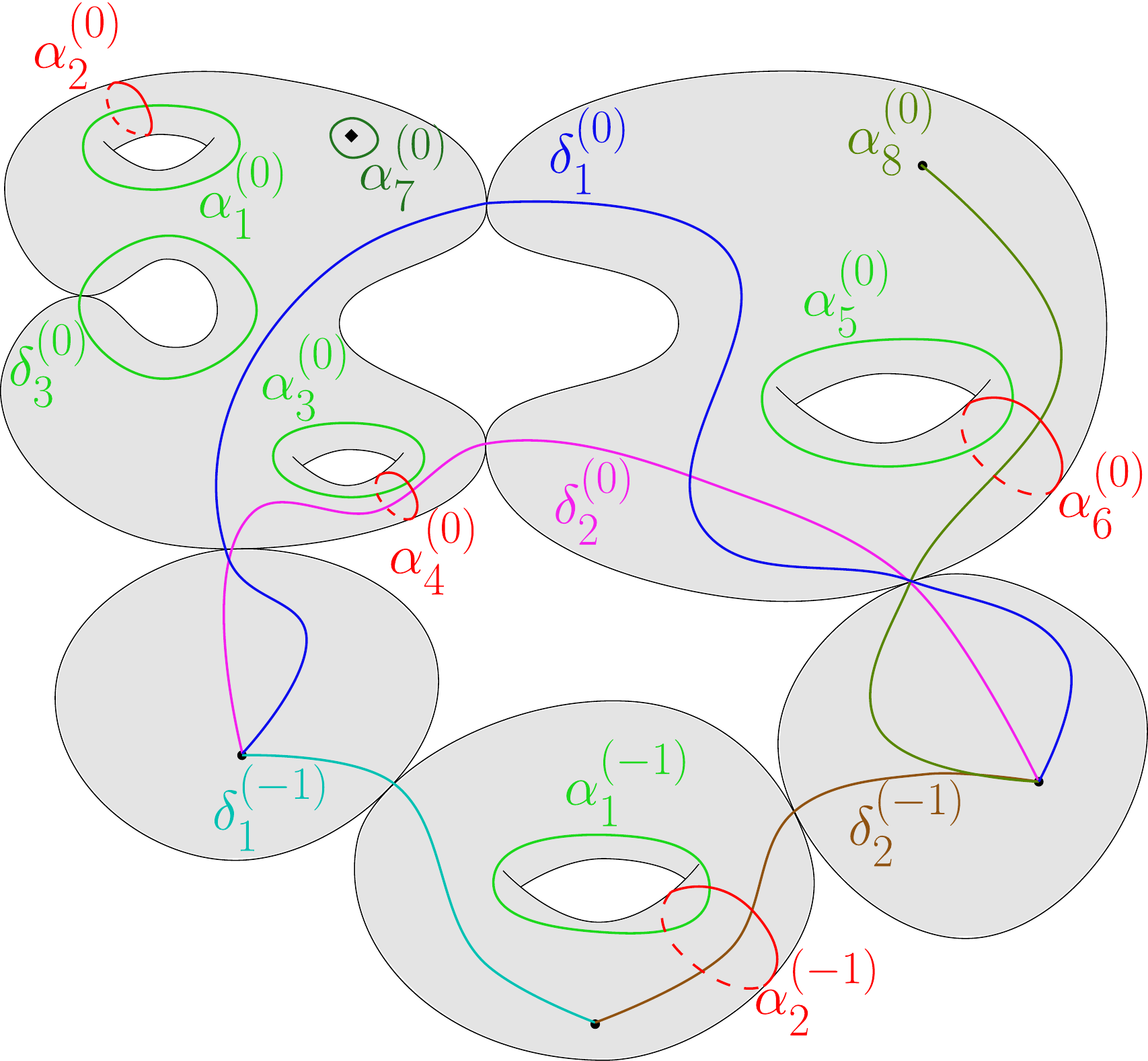}
}
\caption{An example of a $\enhancG$-adapted basis}
\label{fig:AdpBasisEx}
\end{figure}
\end{example}

\section{Log periods}
In this section we only work with the universal family of multi-scale differentials $\universalFam$. Our goal is to study the asymptotics of relative periods as we approach the boundary $D$ of the local model domain $B$.
When integrating differential forms over cycles, we will usually not distinguish between a cycle $\cycle$ and a representative $\pa$.
\label{section:LogPeriods}

\subsection{The definition of log periods}
We fix a cycle $\cycle \in H_1(X\!\setminus\! P, Z)$ represented by a path $\pa$. We want to investigate the behavior of the periods of $\omega(b)$ as $b$ approaches the boundary. Thus we need a way of deforming the cycle $\cycle$ from $X=X_{b_0}=Y_{b_0}$ to nearby fibers of $\universal\to B$. In~\Cref{section:CycleExtension} below we will give an explicit construction of a continuous family of cycles $[\pa(b)]\in H_1(Y_b\!\setminus\!P_b,Z_b)$ deforming $\cycle$.
The family of cycles $[\pa(b)]$ is only well-defined locally. By a process analogous to analytic continuation, it can be considered as a multivalued family with multiple branches; the values of different branches differ by integral multiples of the vanishing cycles $\lambda_e$.

Afterwards, in~\Cref{section:invcycle}, we perform a second construction, making the family of cycles $[\pa(b)]$ invariant under the monodromy, thus obtaining a family $\invcycle$ of relative cycles well-defined on $B\setminus D$.
By repeating the process for all elements of a basis $\{\pa_1,\ldots,\pa_n\}$ of homology we construct a family of bases for $H_1(X_{b}\setminus P_{b},Z_{b},\CC)$ for all $b\in B\setminus D$ that varies continuously over $B\setminus D$. In algebro-geometric terms, we construct an explicit frame for the dual of the Deligne extension of the local system of relative cohomology, see for example \cite[Section 3]{SchnellCanonical}. Postponing the explicit constructions of $[\pa(b)]$ and $[\hat{\pa}(b)]$ for now, we are able to define log periods.
We define the \textit{vanishing cycle period}
\[
r_e(b):=\frac{1}{2\pi i}\int_{\lambda_e}\omega(b).
\]
\begin{definition}\label{def:LogPeriod}
We define the \textit{log period} $\LogP:B\setminus D\rightarrow \CC$ of $\omega$ along $\gamma$ by
\[\begin{split}
\LogP(b):=& \dfrac{1}{\scl[\topl(\gamma)]}\int_{\invpa}\omega(b)\\
=&\dfrac{1}{\scl[\topl(\gamma)]}\left[\int_{\gamma(\stdpar)}\omega(b)-\sum_{e\in E}\IP[\gamma,\lambda_e] r_e(b)\ln(s_e)\right],
\end{split}
\]
where $[\pa(b)]$ and $\invcycle$ are the families of cycles that will be constructed in~\Cref{section:CycleExtension} and~\Cref{section:invcycle}, respectively, and $\IP$ denotes the intersection pairing. Recall that the plumbing parameters $s_e$ were defined in \cref{eq:PlumbingParameters}.
\end{definition}

Several comments are in order. Heuristically, the scaling factor $ \dfrac{1}{\scl[i]}$ comes from the fact that on $\eq_{(i)}$, at least away from the nodes, the differential $\omega$ behaves like $(\starD{\twistD})_{(i)}=\scl[\topl(\gamma)]\twistD$ since the contribution from the modification differential $\modD$ is small. The logarithm $\ln(s_e)$ is to be understood as follows. One starts by choosing a branch of the logarithm for each coordinate $t_i$ and $h_e$ at some base point $x_0\in B\setminus\bdComp$. Afterwards we extend the branches via analytic continuation. By requiring that $\ln(s_e)=\sum_{i=\ell(e-)}^{\ell(e+)-1}a_i\ln(t_i)$ we then define branches for all parameters $s_e$. This of course only defines a multivalued function but later on we will see that $\LogP$ is single-valued, where we recall that the deformation $\gamma(b)$ is also multivalued, and this multivaluedness will cancel out the multivaluedness of $\ln(s_e)$. The idea is that $\int_{\pa(b)} \omega$ and $\ln(s_e)$ behave similarly under analytic continuation and thus their difference is single-valued. We stress that there is not a unique function $\LogP$ but there is a countable collection of such depending on  the choices of representatives $\pa$ on the fixed based surface and branches of the logarithms, but once those initial choices are made, $\LogP$ is a well-defined and single-valued function on $B\setminus D$. From now on we always fix such an initial choice.

The following is the main result of this section.
\begin{theorem}\label{thm:PerThm} For any homology class $\cycle\in H_1(X\setminus P,Z)$, the log period $\LogP$
is single-valued and extends to an analytic function on $B$.
Furthermore, the limit  of $\LogP$ at  $b=(\eta,0,\ldots, 0)\in \bdComp$ is
\[
\LogP(b)=
\int_{\pa(b)_{\topl}}\Hol(\twistD)-
 \sum_{e\in E^{hor}}\IP[\pa_{\topl},\lambda_e] \res_{q_e^-}(\twistD)c_e.
\]
where $c_e$ is a constant
and the {\em holomorphic part}
%$\int_{\pa(b)}\Hol(\omega(b)),\,$
$\int_{\pa(b)}\Hol(\twistD) $ will be defined in~\cref{eq:HolPart}. The constants $c_e$ only depend on choices of the normal coordinates and branches of the logarithm.
\end{theorem}

While we postpone the definition of $\Hol(\omega(b))$ in general to~\Cref{section:HolPart}, we mention here a special case. If $b\in \bdComp$ and $\pa$ is non-horizontal, then our definition will yield
\[
\int_{\pa}\Hol(\omega(b))=\int_{\pa} \twistD.
\]
Recall that non-horizontal was defined in \Cref{def:hcc}.
We thus obtain the following corollary

\begin{corollary}\label{cor:NonCross}
If $\pa$ is {\em non}-horizontal, then
\[
\LogP(\eta,0,0)=\int_{\gamma_{\topl}}\eta.
\]

\end{corollary}

\begin{remark}
We will later see in~\Cref{section:LogPeriodsHodge} that
the \Cref{thm:PerThm} can be seen as a version of Schmid's nilpotent orbit theorem \cite{SchmidVHS}
for flat surfaces with the following difference in the setup. Instead of a whole basis for stable differentials we only have a single multi-scale differential and instead of absolute homology we integrate over relative homology.
\end{remark}

\subsection*{Comparing log periods and perturbed periods} In \cite{BCGGMsm} the authors introduce a coordinate system on $B$ given by so-called {\em perturbed period coordinates}. Perturbed periods come in two different types, depending on whether $\pa$ crosses any horizontal nodes or not.
If $\pa$ only crosses vertical nodes, one  truncates $\pa_k$ at the nearby section $\Nearby[]$ of such a node and the perturbed period along $\pa$ is obtained by integrating $\omega(b)$ over the truncation of $\pa$. In particular, the perturbed period forgets about the period inside the plumbing cylinder.
We use log periods in this paper because it is easier to compare them, rather than perturbed periods, quantitatively to the actual periods $\int_{\pa(b)}\omega(b)$.
The downside of using log periods is that the collection of all log periods over a relative homology basis are not local coordinates on $B$, since one cannot recover the plumbing parameters $h_e$ at horizontal nodes.

The rest of the section is devoted to the setup for and the proof of~\Cref{thm:PerThm}.

\subsection{Deforming cycles to the universal family}\label{section:CycleExtension}
We now describe the construction of the family of cycles $[\pa(b)]$. We first explain the construction at the level of paths.%

The construction proceeds in two steps. In the first step we deform $\pa$ from $X$ to nearby curves $X_b$ of the {\em model family}. Afterwards, in the second step, we parse through the explicit construction of the universal family in~\Cref{section:PlumbConstruction} to deform the cycles to $\universal$.

We now start with the first step. First, lift $\pa$ to a path on the normalization $\tilde{X}$. Since the family of normalizations $\normFam$ is a family of (possibly disconnected) smooth Riemann surfaces, we can, after possibly shrinking $B$, find a $C^{\infty}$-trivialization of $\normFam$, by Ehresmanns lemma. Furthermore we can choose the trivialization such that it identifies the marked points and nodes. Via the trivialization we construct a family of paths on $\norm$ deforming $\pa$, which we still denote by $\pa$. Since we chose a trivialization that preserves the nodes, the family of paths $\pa$ descends to $\eq$. By abuse of notation we denote this new family of paths on $\eq$ also by $\pa$.

Suppose we now start with the homology cycle  $\cycle\in H_1(X\setminus P,Z)$ represented by the original path $\pa$. By deforming $\pa$ as above and then taking the associated homology class, we get a family of cycles in $H_1(X_b\setminus P_b,Z_b)$ for all $b\in B$ deforming $\cycle$. Note that the cycles still live in the appropriate relative homology since we  chose the trivializations to preserve the marked points.

\subsection{Thin and thick part of \texorpdfstring{$\pa$}{}}
We now prepare for the second step of the construction. Again we work with actual paths first.
For every node or marked point that $\pa$ goes through, we modify $\pa$ through a homotopy such that, locally in a $\stdFormA[]$-coordinate neighborhood of $b_0$, where we recall $\phi$-coordinates from \Cref{def:phicoord}, the path $\pa$ coincides with the straight line from $p_0$ to the origin. By choosing the trivialization from the first step appropriately, we can achieve this for the whole family of paths $\pa$ over $B$. Afterwards, we define the \textit{thick part} $\gamma^{thick}$ of $\pa$ to be the path contained in $\eq\setminus \imDisk[]$, obtained by truncating~$\pa$ at the nearby sections $\Nearby[](b_0)$. The remaining part of $\pa$, given by the straight lines from $p_0=\Nearby[](b_0)$ to the origin in $\stdFormA[]$-coordinates, is denoted by $\gamma^{thin}$ and is called the \textit{thin part} of $\pa$. By construction, $\pa$ is the collection of disjoint paths $\pa^{thick}$ and $\pa^{thin}$.
See \Cref{fig:gammathick} for an example.

\begin{figure}[h]
\centering
\includegraphics[scale=0.4]{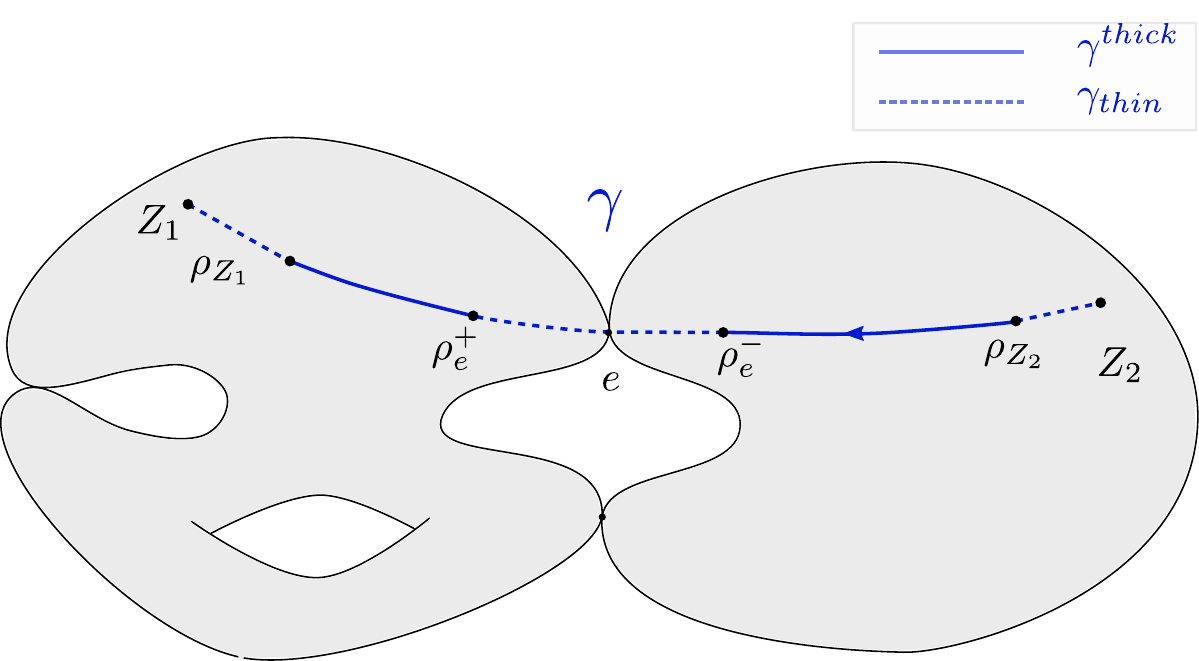}
\caption{The separation of $\gamma$ into $\gamma^{thick}$ and $\gamma^{thin}$}
\label{fig:gammathick}
\end{figure}

\subsection{Second construction step}\label{section:SecondStep}
We now proceed with the construction of $\gamma(b)$. Recall that so far we constructed a family $\pa$ of paths on $\eq$.
We will define the thick part $\gamma(b)^{thick}$ and the thin part $\gamma(b)^{thin}$ separately and finally let $\gamma(b)$ be the composition of $\gamma(b)^{thick}$ and $\gamma(b)^{thin}$.
The thick part $\gamma(b)^{thick}$ is simply
\[
\gamma(b)^{thick}:=\Psi^{-1}(\gamma^{thick}),
\]
where $\Psi$ was defined in \cref{def:psi}.
The construction now differs near nodes and near marked points.
We focus on a marked point $\calZ_k$ first.
In $\stdFormB[\calZ_k]$-coordinates the endpoint of $\pa(b)^{thick}$ is $(\stdFormB[h,b])^{-1}(\Nearby[k](b_0))$. We denote by $\gamma(b)_k^{thin}$ the straight line from $(\stdFormB[k,b])^{-1}(\Nearby[h](b_0))$ to the base point $(\stdFormB[k,b_0])^{-1}(\Nearby[h](b_0))=p_0$, followed by the straight line  from $p_0$ to the origin.

At nodes the construction is more involved.
As before we can connect $(\stdFormB[e,b]^+)^{-1}(\Nearby^+(b_0))$ to $p_0$  via a straight line in $\stdFormB^+$-coordinates and similarly  $(\stdFormB[e,b]^-)^{-1}(\Nearby^-(b_0))$ to $p_0$  via a straight line in $\stdFormB^-$-coordinates.
We denote $p_0^+$ and $p_0^-$ the images of $p_0$ in $\stdAnn^+$ and $\stdAnn^-$ respectively. To finish the construction it remains to connect $p_0^+$ and $p_0^-$ on the plumbing fixture $\stdFix$.

At $b\in B$, we identify $(\stdFix)_b:=\{(u,v)\in\Delta_{\delta}\,:\, uv=s_e\}$ with the annulus ~$A:=\{ u\,:\, \delta/|s_e|\leq |u|\leq\delta\}$ in $u$-coordinates. Under this identification we have
\[
p_0^+=\delta/\!\sqrt{R},\, p_0^-=s_e\!\sqrt{R}/\delta.
\]
We divide the annulus $A$ into finitely many sectors $\Sec_l$ each with a chosen base point $x_l$.
Suppose $s_e\!\sqrt{R}/\delta\in \Sec_l$. We then choose a path from $\delta/\!\sqrt{R}$ to $x_l$ and connect $x_l$ to $s_e\!\sqrt{R}/\delta$ via a straight line, as depicted in \Cref{fig:ucoord}.

\begin{figure}[h]
\centering
\includegraphics[scale=0.25]{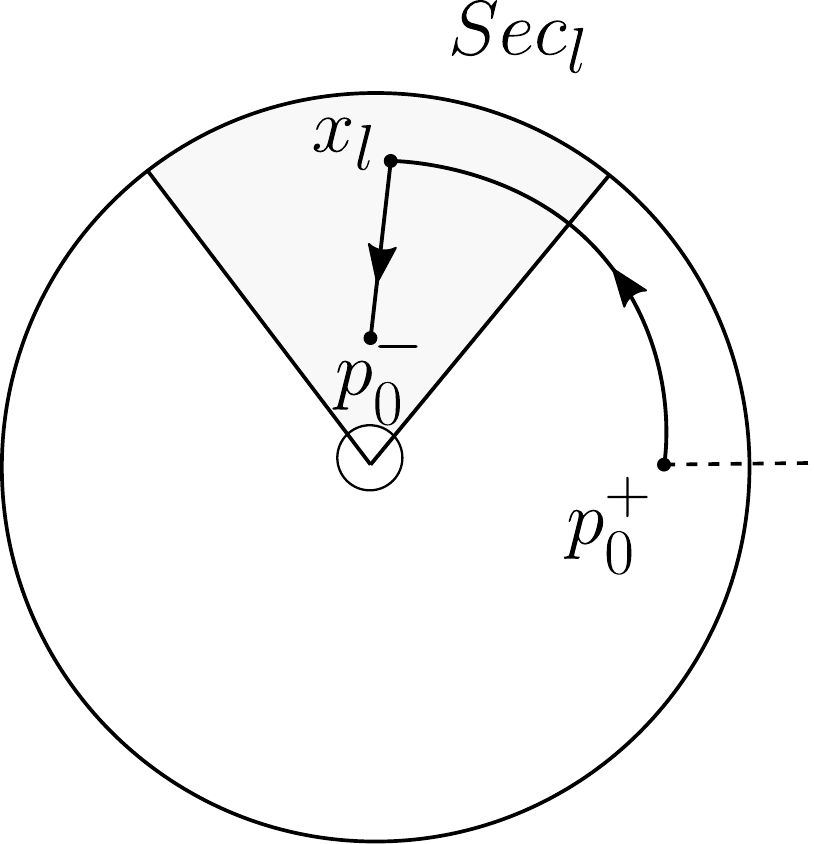}
\caption{The path $\gamma(b)_e^{thin}$ in $u$-coordinates}
\label{fig:ucoord}
\end{figure}
The resulting construction is continuous on $\Sec_l$ but depends on the choice of a path from $\delta/\!\sqrt{R}$ to $x_l$; different choices differ by a multiple of the vanishing cycle $\lambda_e$. We can make all choices in such a way that the construction is continuous on the intersection of two sectors but has monodromy if we try to extend it to all of $A$.
We thus constructed a (multivalued) path $\pa(b)_e^{thin}$ from $p_0^+$ to $p_0^-$.
We let $\pa^{thin}$ be the  composition of $\pa(b)^{thin}_e$ and $\pa(b)_k^{thin}$ for all nodes and marked zeroes $\calS_k$ crossed through by $\pa$.
We finally let $\gamma(b)$ be the composition of $\pa(b)^{thick}$ and $\pa(b)^{thin}$.

We stress that while $\gamma^{thick}$ and $\gamma^{thin}$ are paths on the model family $\eq$, the paths $\gamma(b)^{thin}$, $\gamma(b)^{thick}$, and thus also $\pa(b)$, are on the universal family $\universal$.

The family $\pa(b)$ is multivalued, with different branches differing by integral multiples of the vanishing cycles. Furthermore if $\pa$ and $\pa'$ are homologous on $X$ but differ by multiples of the vanishing cycles, then $\pa(b)$ and $\pa'(b)$ yield different branches of the same multivalued function. Thus we can define a multivalued family of cycles $[\pa(b)]$, but which branch is picked out depends on the choice of a representative $\pa$ for $\cycle$ on $X$.

\subsection{Monodromy invariant cycles}\label{section:invcycle}
Due to monodromy the family ~$[\pa(b)]\in H_1(X_b\setminus P_b,Z_b,\ZZ)$ is not well-defined on all of $B\setminus D$. By subtracting suitable logarithmic terms we are going to construct a new family of cycles which will be monodromy invariant.
We now choose, once and for all, branches of logarithms for $t_i$ and $h_e$ locally near a base point $x_0\in B\setminus D$,
and then define branches for $s_e$ at vertical nodes, locally near $x_0$, via
\[
\ln(s_e):=\sum_{i=\ell(e-)}^{\ell(e+)-1} a_i\ln(t_i)
\]
where $a_i$ was defined in \cref{eq:ak}.

To make the family $[\pa(b)]$ monodromy-invariant we set
\begin{equation}\label{eq:invcycle}
[\invpa]:= [\pa(b)]-\dfrac{1}{2\pi i}\sum_{e\in E}\IP[ \gamma,\lambda_e] \ln(s_e)[\lambda_e]\in H_1(X_b\setminus P_b,Z_b,\CC).
\end{equation}
We call $[\invpa]$ the \textit{invariant cycle associated to} $\pa$ since $[\invpa]$ is invariant under analytic continuation along any path in $\pi_1(B\setminus D,x_0)$.
The invariant cycle is well-defined on $B\setminus D$ but not unique, since both $[\pa(b)]$ as well as the branches of $\ln(s_e)$ involve certain choices. From now on we fix one set of those choices.

\subsection{Holomorphic part of a period}
\label{section:HolPart}
We can now define the holomorphic part of the period  which appeared in the statement of~\Cref{thm:PerThm}. We recall that $b=(\twistD,t,h)$.
%When analyzing the asymptotics of $\LogP$ near the boundary we will need the following definition:
We set
\begin{equation}\label{eq:HolPart}
\int_{\gamma}\Hol(\omega(b)):=\int_{\gamma^{thick}}(\starD{\twistD}+\modD) + \int_{\gamma^{thin}(b)} (\starD{\twistD}+\modD)^{hol}
\end{equation}
where $(\starD{\twistD}+\modD)^{hol}$ is the holomorphic part of the Laurent expansion of $\starD{\twistD}+\modD$ of $\eta$ in $\stdFormA^{\pm}$-coordinates
near the nodes.
We stress that we are not defining a differential $\Hol(\omega(b))$ but only the expression $\int_{\gamma}\Hol(\omega(b))$. We define $\int_{\pa}\Hol(\twistD)$ in the same way where $\int_{\gamma^{thin}(b)} (\starD{\twistD}+\modD)^{hol}$ is replaced by $ \int_{\gamma^{thin}}\twistD$.

Note that in particular $\int_{\gamma^{thin}}\twistD^{hol}=\int_{\gamma^{thin}}\twistD$ if $\gamma$ is non-horizontal and we thus obtain \Cref{cor:NonCross}

We have now defined all the objects appearing in $\Cref{thm:PerThm}$ and can thus proceed with the proof.

\begin{proof}[Proof of~\Cref{thm:PerThm}]
We first show that the log period is indeed single-valued.
Recall that both $\int_{\gamma(b)}\omega(b)$ and $r_e(b)\ln(s_e)$ are multivalued; analytic continuation along a path encircling the origin $k_e$ times counter-clockwise in $s_e$-coordinates  changes $\gamma(b)\mapsto \gamma(b)+k_e\IP[\pa(b),\lambda_e]\lambda_e$, where $\lambda_e$ is the vanishing cycle of the node $e$. Thus both $\int_{\gamma(b)}\omega(b)$ and $r_e(b)\IP[\pa(b),\lambda_e]\ln(s_e)$ change under such an analytic continuation by the addition of
\[
k_e\int_{\lambda_e}\omega(b)= k_er_e(b)\IP[\pa(b),\lambda_e],
\]
and in particular their difference is single-valued.
Our goal is to compare the periods of $\int_{\gamma} \eta$ and $\int_{\gamma(\stdpar)} \omega(\stdpar)$. For this we need to use the plumbing construction of $\omega$ reviewed in ~\Cref{section:PlumbConstruction}. We split the period over $\pa$ into the thick and thin part, i.e. $\int_{\gamma(\stdpar)}\omega(\stdpar)=\int_{\gamma(b)^{thick}}\omega(\stdpar)+\int_{\gamma(b)^{thin}}\omega(\stdpar)$.

Over the thick part we have
\[
\int_{\gamma(b)^{thick}}\omega(\stdpar)=\int_{\gamma^{thick}}(\starD{\twistD}+\modD(b))
\]
and thus
\[
\lim_{t,h\to 0} \dfrac{1}{\scl[\topl(\pa)]}\int_{\gamma(b)^{thick}}\omega(b)=\int_{\pa^{thick}_{\topl}}\eta,
\]
since on $\reseq{\eq}$ the modification differential $\modD$ is divisible by $\scl[i-1]$.
It remains to compute the integral over $\gamma(b)^{thin}$.

We discuss the situation at vertical nodes, horizontal nodes and marked zeroes separately.
We start with the case of vertical nodes.
We recall from the construction that, in this case, $\gamma(b)^{thin}_e$ consists of two parts,
the straight line from $\stdFormB[b]^{-1}(\Nearby[](b))$ to $p_0^{\pm}$ and then a chosen path from $p_0^+=\delta/\!\sqrt{R}$ to $p_0^-=s_e\!\sqrt{R}/\delta$. We analyze both parts separately.

 Near a vertical node $e$, in $u$-coordinates, we have
\[
\begin{split}
\int^{\delta/\!\sqrt{R}}_{s_e\sqrt{R}/\delta}(\starD{\twistD}+\modD)=& \int^{\delta/\!\sqrt{R}}_{s_e\!\sqrt{R}/\delta} \left(\scl[\ell(e+)]u^{\kappa_e}-r_e(b)\right)\dfrac{du}{u}\\
=& \scl[\ell(e+)] \cdot\dfrac{(\delta/\!\sqrt{R})^{\kappa_e}-(s_e\!\sqrt{R}/\delta)^{\kappa_e}}{\kappa_e}\\
&-r_e(b)\left[\ln(\delta/\!\sqrt{R})-\ln(s_e\!\sqrt{R}/\delta)\right].
\end{split}
\]

Note that there exist integers $\alpha_e$ such that
\[
\ln(s_e\!\sqrt{R}/\delta)= \ln(s_e)-\ln(\delta/\!\sqrt{R}) +2\pi i\alpha_e.
\]
We thus define $c_e:=2\pi i\alpha_e-2\ln(\delta/\!\sqrt{R})$.
Additionally, we compute
\[
\int_{\gamma^{thin}_{e^+}}\twistD=\int_{\gamma^{thin}_{e^+}}\eta^{hol}=\int_{0}^{\delta/\!\sqrt{R}}u^{\kappa_e-1}du=\dfrac{(\delta/\!\sqrt{R})^{\kappa_e}}{\kappa_e}.
\]

Finally, we need to estimate the period along the straight line segment from $\stdFormB[b]^{-1}(\Nearby[](b))$ to $p_0$.
Recall from~\Cref{section:StdFormCoord} that $\stdFormB[b_0]^{-1}(\Nearby[](b_0))=p_0$ for all multi-scale differentials $\twistD$ and thus
\[
\lim_{t,h\to 0} \stdFormB[(\twistD,t,h)]^{-1}(\Nearby[](\twistD,t,h))= p_0.
\]
We conclude that
 \[
\int^{\stdFormB[(\twistD,t,h)]^{-1}(\Nearby[](\twistD,t,h))}_{p_0}(\starD{\twistD}+\modD)=O(\scl[\ell(e+)](\parLevel+\parHor)).
\]
%since $\stdFormB^{\pm}(u,p_0)=\Nearby[e]^{\pm}(b_0)=\stdFormB^{\pm}(b_0,p_0)$ for any $u\in U\times (0,\dots,0)$.

The notation $O(\scl[\stdlvl](\parLevel+\parHor))$ here means that the left-hand side is analytic in $b=(t,h,\twistD)$ and every monomial in the power series expansion is divisible by $\scl[\stdlvl]t_i$ or $\scl[\stdlvl]h_e$ for some $i$ or $e$.

Putting everything together, we conclude that at vertical nodes
\[
\int_{\pa(b)^{thin}_e} \omega(b)= \scl[\ell(e+)]\int_{\gamma^{thin}_{e^+}}\eta^{hol} + r_e(b)\ln(s_e)+ O(\scl[\ell(e+)](\parLevel+\parHor))
\]
where we used that $r_e(b)=O(\scl[\ell(e+)](\parLevel+\parHor))$.

We now turn to the case of a marked zero $\calZ_k$. In this case $\pa(b)_k^{thin}$ consists of the straight line from
$(\stdFormB[k,b])^{-1}(\Nearby[h](b_0))$ to $(\stdFormB[k,b_0])^{-1}(\Nearby[h](b_0))=p_0$ combined with the straight line from $p_0$ to the origin.

As before we have
 \[
\int^{\stdFormB[(\twistD,t,h)]^{-1}(\Nearby[](\twistD,t,h))}_{p_0}(\starD{\twistD}+\modD)=O(\scl[\ell(e+)](\parLevel+\parHor)).
\]
And furthermore
\[
\begin{split}
\int^{p_0}_{0}(\starD{\twistD}+\modD)&= \int^{p_0}_{0} \left(\scl[\ell(e+)]u^{\kappa_e-1}\right)du\\
&= \scl[\ell(e+)] \int_{\pa_k^{thin}}\eta=\scl[\ell(e+)] \int_{\pa_k^{thin}}\eta^{hol}.\end{split}
\]

Finally we address the case of horizontal nodes. In that case we have $\twistD^{hol}=0$ in $\stdFormA$-coordinates and
\[
\begin{split}
\int^{\delta/\!\sqrt{R}}_{s_e\!\sqrt{R}/\delta}(\starD{\twistD}+\modD)&= \int^{\delta/\!\sqrt{R}}_{s_e\!\sqrt{R}/\delta} -r_e(b)\dfrac{du}{u}\\
&= r_e(b)\ln(s_e)+ O(\scl[\ell(e+)](\parLevel+\parHor)).
\end{split}
\]
\end{proof}

\section{Setup for one-parameter families}
\label{section:OneParameterRestriction}
Our method for studying the boundary of a linear subvariety is based on the observation that every point in the boundary can be approached along a holomorphic one-parameter family. This will enable us to do computations in one-parameter families, which is more useful for our purposes since it allows to control the relative growth rates of the parameters $t_i$ and $h_e$. We first collect some simple facts about one-parameter families.

\subsection{Short arcs}\label{section:ShortArcs}

See \cite{KollarArcs} for an introduction to this circle of ideas.
\begin{definition}
A \textit{(complex) analytic arc} on a complex analytic space $X$ is a holomorphic map $\sharc[X]$.
Given a subset $Z\subseteq X$, a \textit{short arc} on $(X,Z)$ is an analytic arc with $\arc^{-1}(Z)=\{0\}$.
We say an analytic arc $\arc$ \textit{connects two points} $x$ and $y$ if both are contained in the image of $\arc$. Similarly we say $\arc$ passes through $x$ if $x$ is contained in the image.
Furthermore we say a short arc $\arc$ is \textit{smooth} if $\arc(\Delta^*)\subseteq X_{reg}$, where $X_{reg}$ denotes the smooth locus of $X$.
\end{definition}
Unless stated otherwise, we denote the coordinate on $\Delta$ by $z$. Recall that we do not specify the radius of $\Delta$ in order to lighten the notation.
The following is a simple consequence of the ideas developed by Winkelmann in \cite{Winkelmann}.
\begin{lemma}\label{lemma:Arc}
Let $X$ be an irreducible complex analytic space and let $Z\subseteq X$ be a complex analytic subspace. Then for any pair of points $x\in X\setminus Z,\, z\in Z$ there exists a short arc $
\sharc[X]$ on $(X,Z)$ connecting $x$ and $z$. Furthermore if $x \in X_{reg}\!\setminus Z$, then there exists a smooth such arc~$\arc$.
\end{lemma}

\begin{proof}
By \cite[Thm. 5]{Winkelmann} there exists a holomorphic map $\arc:\Delta\to X$ passing through $x$ and $z$. Since $\arc^{-1}(Z)$ is a proper subspace,  after  possibly shrinking $\Delta$, we can assume that $\arc^{-1}(Z)$ is a finite set. We can choose a Jordan curve in $\Delta$ such that $z$ and $x$ are in its interior component while all other points of $\arc^{-1}(Z)$ lie in the exterior component. By the Riemann mapping theorem the interior component is biholomorphic to $\Delta$.
For the second claim we proceed similarly. Again, after shrinking, we can assume that the preimage $\arc^{-1}((X\setminus Z)_{sing})$ of the singular locus is finite. Again  we choose a Jordan curve containing $z$ and $x$ in its interior and all singular points in its exterior.
\end{proof}

\subsection{Log periods along one-parameter families}

Along a one-parameter family $\arc$ we immediately get the following slight improvement of~\Cref{thm:PerThm}. We define $\sigma_e$ to be the order of vanishing of $s_e\circ f$ at $z=0$.

\begin{corollary}\label{cor:LogPerOPF}
Let $\sharc[B]$ be a one-parameter family of differentials with multi-scale limit $\arc(0)=b_0=(X,\twistD)$ and let $\gamma\in H_1(X\setminus P,Z)$.
The log period $\LogP^{\arc}$ along $\arc$ defined by
\[
\LogP^{\arc}(z):= \frac{1}{\scl[\topl(\gamma)]}\left[\int_{\gamma(f(z))}\omega(f(z))-\left(\sum_{e\in E}\IP[\pa,\lambda_e] r_e(f(z))\sigma_e\right)\ln(z)\right]
\] is single-valued, analytic in a neighborhood of the origin,
and satisfies
\[
\LogP^{\arc}(0)=
 \int_{\gamma_{\topl}} \Hol(\twistD) + \sum_{e\in E}\IP[\pa_{\topl},\lambda_e] \res_{q_e^-}(\twistD)\sigma_ec_e'.
\]
Here $c_e'$ are constants and $\pa_{\topl}$ is the restriction of $\pa$ to its top level.
\end{corollary}

\begin{proof}
Along $\arc$ we can write $s_e(f(z))=z^{\sigma_e}e^{g_e(z)}$ for some analytic function $g_e$.  Then for each $e$ there exists an integer $k_e$ such that
\[
\ln(s_e)=\sigma_e\ln(z)+ g(z)+2\pi ik_e.
\]
We thus compute
\[
\LogP(f(z))=\LogP^{\arc}(z)-\left(\sum_{e\in E}\IP[\pa,\lambda_e] \dfrac{r_e(f(z))}{\scl[\topl(\pa)]}(g(z)+2\pi ik_e)\right)
\]
and then the result follows directly from~\Cref{thm:PerThm} by setting $c_e':=c_e+g(0)+2\pi ik_e$.
\end{proof}
We note that since $b_0$ is contained in the most degenerate stratum $\bdComp\subseteq D$, all integers $\sigma_e$ are strictly positive.

The usefulness of log periods $\LogP^{f}$ along $f$
stems from the fact that the logarithmic divergence now only depends  on one variable $z$. Thus in order to get sufficient control over the divergence of $\LogP^{\arc}$ on the punctured disk $\Delta^*\!$, we only need to control one expression $\sum_{e\in E(\enhancG)}\langle \pa_{\topl},\lambda_e\rangle r_e(f(z))\sigma_e$.

\section{Monodromy of complex linear varieties}\label{section:MonodromyLinearVariety}

\subsection{The Gauss-Manin connection}\label{section:MonodromyGM}
In this subsection we let $(\pi:\calT\rightarrow A,\omega)$ be an arbitrary family of flat surfaces over  an arbitrary smooth base $A$.

We let $\LL$ be the local system, or equivalently the  vector bundle with flat connection,  of relative cohomology over $A$, with fiber $\LL_{a}\simeq  \Hi[](T_a\setminus \calP_a,\calZ_a,\ZZ)$.
More explicitly,
 \[\mathbb{L}:= \Ri(\pi|_{\calT\setminus\calP})_*j_{!}\underline{\ZZ}\]
where $j: \calZ \rightarrow \calT\setminus\calP$ is the inclusion of the reduced zero-divisor of $\omega$.
For any given pair of points $a,a'\in A$ choose a path $\gamma$ connecting them.
The Gauss-Manin connection associated to $\mathbb{L}$ allows to identify different fibers $\LL_a\simeq \LL_{a'}$  via parallel transport along $\gamma$.
For convenience of the reader, we recall the details.
On any contractible neighborhood $W\subseteq A$\ we can trivialize $\mathbb{L}$ and thus identify $\LL_a\simeq \LL_{a'}$ for all $a,a'\in W$.
Let $\gamma:[0,1]\to A$ be a path and let $V\subseteq \mathbb{L}_{\gamma(0)}$ be a subspace. By covering $\gamma([0,1])$ with finitely many contractible neighborhoods, we get an induced isomorphism $\phi_{\gamma}:\LL_{\gamma(0)}\simeq \LL_{\gamma(1)}$ and we define
\[
\GM_{\gamma}(V):=\phi_{\gamma}(V)\subseteq\LL_{\gamma(1)},
\]
which only depends on the homotopy class of $\gamma$ and not on how it is covered by contractible neighborhoods.

\subsection{Hodge theoretic description of log periods}\label{section:LogPeriodsHodge}
We now describe the monodromy action on the relative cohomology near the boundary of $\MSDS$.
As a byproduct, we get a more conceptual definition of log periods.
In this subsection we work with the local universal family $\universalFam$ of multi-scale differentials.
For the remainder of this subsection only we relabel the local coordinates on $B$. We set
\[
(z_1,\dots,z_{\dimU},z_{\dimU+1},\dots,z_{\dimPar+\dimU}):=\stdpar=(\twistD,t,h)
\] where we recall from~\cref{eq:mn} that $\dimPar=\ell(\enhancG)-1+|E^{hor}|$ and $\dimU=\dim U$.

The boundary $D$ of $B$ in these coordinates is then  $D=\left(\prod_{k=\dimU+1}^{\dimU+\dimPar} z_k=0\right )$.
The universal family $(\universal,\omega)$ over $B\setminus D=\Delta^{\dimU}\times (\Delta^*)^{\dimPar}$ is a family of  flat surfaces contained in $\Stra$. We now restrict the local system $\LL$ from~\Cref{section:MonodromyGM} to $B\setminus D$ with associated monodromy action
\[
\ZZ^{N}\simeq \pi_1(B\setminus D,x_0)\rightarrow \GL(\Hi(X_{x_0}\setminus\calP_{x_0},\calZ_{x_0},\ZZ)
)\] for some base point $x_0$.

\begin{convention}
From now on $x_0\in B\setminus D$ always denotes a base point in $\Stra$ with corresponding fiber $(X_{x_0},\omega_{x_0})$.
\end{convention}

 Let $T_k$ be the monodromy operator, i.e. the image under the monodromy action, of the standard generator of $\pi_1(B\setminus D,x_0)$ encircling the origin once in the coordinate $z_k$ and constant otherwise. We sometimes write $T_i$ or $T_e$ instead of $T_k$ if $z_k=t_i$ or $z_k=s_e$. The monodromy action can be computed explicitly from the construction of $\MSDS$ in~\Cref{section:ConstrUniversal}.
We have
\begin{equation}\label{eq:MonHor}
T_e(\cycle)= \cycle+ \IP[\pa,\lambda_e] [\lambda_e],
\end{equation}
 i.e. $T_e$ acts as a Dehn twist along $\lambda_e$.
Similarly,
\begin{equation}\label{eq:MonVer}
T_i(\cycle)=\cycle+ \sum_{e\in E^{ver},\ \ell(e-)\leq i<\ell(e+)}m_{e,i}\,\IP[\pa,\lambda_e] [\lambda_e],
\end{equation}
i.e. $T_i$ acts as a multitwist along all curves $\lambda_e$ with $\ell(e-)\leq i<\ell(e+)$ where the multiplicities $m_{e,i}$ were defined in \Cref{eq:ak}. Note that in particular $(I-T_k)^2=0$ for all $k$.

We set
\begin{equation}\label{eq:MonLog}
N_k:=-\log T_k=I-T_k.
\end{equation}

We choose a basis $\{\gamma_1(b),\dots,\gamma_n(b)\}$ of $H_1(X_{x_0}\!\setminus  P_{x_0},Z_{x_0})$ where $d:=\dim H_1(X_{x_0}\!\setminus P_{x_0},Z_{x_0})=\dim\Stra$. Due to the multivaluedness of $\gamma_k(b)$, there are two ways of defining a relative period map, which we now explain. We choose one of the branches $\gamma_k(b)$ near $x_0$, as explained in~\Cref{section:SecondStep}.

\begin{definition}
Locally in a period chart $W$ around $x_0\in B\setminus D$ we can define $\varphi:W\rightarrow \CC^d$ by
\[\label{eq:LocalPer}
\varphi(\stdpar):= \left(\int_{\gamma_k(b)} \omega\right)_{k=1,\dots,d}.
\] Note that $\varphi$ does depend on the choice of branches for $\pa_k(b)$.
% We recall that $n=\dim H_1(X_{b_0}\setminus P_{x_0}, Z_{x_0})$ and
We cannot extend $\varphi$ to all of $B\setminus D$ due to the monodromy action but we still have the following analogue.
The fundamental group $\ZZ^{N}\simeq \pi_1(B\setminus D,x_0)$ acts on $\CC^{d}\simeq H^1(X_{x_0}\setminus P_{x_0}, Z_{x_0})$ by
\[
(m_1,\dots, m_N)\cdot v= T_1^{m_1}\circ\dots\circ T_N^{m_N}(v)
\]
and we denote $\pi:\CC^{d}\to  \CC^{d}/\ZZ^{\dimPar}$ the quotient by the monodromy action.
On $B\setminus D$ we define the relative log period map $\phi:B\setminus D\rightarrow \CC^{d}/\ZZ^{\dimPar}$ by setting $\phi:=\pi\circ\varphi$.
Note that $\phi$ does not depend on the choice of branches for $\pa_k(b)$, since different branches of log periods differ exactly by the monodromy action for some path $\pa\in \pi_{1}(B\setminus D,x_0)$.

\end{definition}
Via the universal cover
\[\begin{split}
\tilde{\pi}:\Delta^{\dimU}\times\mathbb{H}^{\dimPar}&\rightarrow \Delta^{\dimU}\times (\Delta^{*})^{\dimPar}\\
(w_1,\dots,w_{\dimPar+\dimU})&\mapsto (w_{1},\dots,w_{\dimU},e^{2\pi iw_{\dimU+1}},\dots,e^{2\pi iw_{\dimU+\dimPar}})
\end{split}
\] we obtain a lifting $\tilde{\phi}$ of $\phi$ that fits in the following commutative diagram:
\[
\begin{tikzcd}
\Delta^{\dimU}\times\mathbb{H}^{\dimPar}\arrow[r,"\tilde{\phi}"]\arrow[d,swap,"\tilde{\pi}
"]& \CC^d\arrow[d,"\pi"]\\
\Delta^{\dimU}\times(\Delta^{*})^{\dimPar}\arrow[r,"\phi"]&   \CC^d/{\ZZ^{\dimPar}}
\end{tikzcd}.
\]
The map
\[
\begin{split}
\widetilde{\psi}: \Delta^{\dimU}\times \mathbb{H}^{\dimPar}&\rightarrow \CC^d,\\
(w_1,\dots,w_{\dimPar+\dimU})&\mapsto e^{-\sum_{k=M+1}^{M+N} w_kN_{k-M}}\tilde{\phi}(w)
\end{split}
\]
 is $\ZZ^{\dimPar}$-invariant and thus descends to a map $\psi:\Delta^{\dimU}\times (\Delta^{*})^{\dimPar}\rightarrow \CC^d$.

\begin{proposition}\label{prop:LogPerHodge}
The map $\psi:\Delta^{\dimU}\times (\Delta^{*})^{\dimPar}\rightarrow \CC^d$  is, up to rescaling of each component by the scaling parameters of the top level of the corresponding curve, given by log periods. More precisely,
\[
\psi(w)= \left(\scl[\topl(\gamma_k)]\LogP[\gamma_k](w)\right)_{k=1}^d
\]
for all $w\in \Delta^{\dimU}\times (\Delta^{*})^{\dimPar}$.
\end{proposition}
\begin{proof}
In ~\cref{eq:MonHor} and~\cref{eq:MonVer} we described the action of $T_i$ and $T_e$ on homology.
Thus, if we locally write $w_e=\frac{1}{2\pi i}\ln(s_e), w_i=\frac{1}{2\pi i}\ln(t_i)$, we can compute
\[
e^{-\sum_k w_kN_k}\widetilde{\psi}(w)=\left(\int_{\gamma_k(w)}\omega-\dfrac{1}{2\pi i}\sum_{e\in E}\IP[\pa,\lambda_e]\ln(s_e) \int_{\lambda_e}\omega\right)_{k=1}^d.
\]
Now compare this with  \Cref{def:LogPeriod}, where we defined log periods.
\end{proof}

\subsection{Setup for complex linear varieties}
\label{section:SetupComplexLinear}
Let $M\subseteq \Stra$ be a linear subvariety.
Since $M$  is algebraic, the Euclidean closure  $\overline{M}\subseteq \MSDS$ is an algebraic variety. This uses the algebraicity of $\MSDS$ which is proved in \cite[Thm 1.3]{BCGGMsm}). We stress that this is the {\em only} time where we use algebraicity of $M$. From now on we assume that our chosen base point $b_0$ is contained in $\partial M\cap\bdComp$.  Locally near $b_0$, the variety $\overline{M}$ has finitely many irreducible components. Note that this only uses the fact $\overline{M}$ is an analytic variety, i.e. we do not have to use algebraicity a second time.

\begin{assumption}
For now we will assume that $\overline{M}$ is irreducible near $b_0$ and will work under this assumption. In~\Cref{section:MultComp} we explain how to extend the results to the general case.
\end{assumption}
Note that a linear subvariety is only near a smooth point defined by a single linear subspace. Near singular points it looks like a union of multiple linear subspaces. For example, affine invariant submanifolds are manifolds immersed in a stratum and the points of self-intersections correspond exactly to the singular locus.

We choose $x_0\in M_{reg}\cap B$ in the smooth locus $M_{reg}$ of $M$.
In a  local period chart near $x_0$, the variety $M$ coincides with a linear subspace $V\subseteq H^1(X_{x_0}\!\setminus P_{x_0},Z_{x_0})$.
By abuse of notation we don't distinguish the subspace $V$ from the analytic subvariety it defines in a period chart.
We let $\{\pa_1',\ldots,\pa'_{d'}\}$ be a basis of $H_1(X\setminus P,Z)$ where $d'=\dim H_1(X\setminus P,Z)$ and
choose a $\enhancG$-adapted basis $\{\pa_1(b),\ldots,\pa_d(b)\}$ of $H_1(X_{x_0}\setminus P_{x_0},Z_{x_0})$ such that each cycle is either a deformation of $\pa_k'$ for some $k$ as described in \Cref{section:CycleExtension}, a vanishing cycle or a horizontal-crossing cycle. In coordinates given by the $\enhancG$-adapted basis we can write
\[
V= \left(A\cdot \Per(\stdpar)=0\right)
\]
where $A=(A_{kl})_{\substack{k=1,\dots,\codim\\ l=1,\dots,d}}$ and $\Per(\stdpar):=\left(\int_{\gamma_l(b)}\omega(b)\right)_{l=1,\dots,d}$.
To make our computations easier, we will always assume that the matrix $A$ is in {\em reduced row echelon} form. This will be useful in two ways: this determines the matrix $A$ uniquely, and allows us to read off the rank of $A$ easily, for computations in~\Cref{section:CuttingOut}.
We consider linear equations on  $H^1(X_{x_0}\setminus P_{x_0},Z_{x_0})$ as elements of the dual and thus as homology classes.

We let \[
\lvl[(l)]:=\topl(\cycle[l]),\quad
 \lvl[(k)]:= \max_{l=1,\dots,d}\{ \lvl[(l)]\,|\, A_{kl}\neq 0\}.
\]
%We call $\ell{k]$ the \textit{level} of the $k$-th equation.
Let $A^{(i)}=(A^{(i)}_{kl})_{kl}$ denote the matrix obtained from $A$ by
defining
\[
A^{(i)}_{kl}:=\begin{cases} A_{kl} & \text{ if } \ell(k)=\ell(l)=i\\
0 & \text{ otherwise } \end{cases}.
\]
and deleting all zero rows.
In words, $A^{(i)}$ collects all the linear equations of top level $i$ and restricts them to the subsurface $\reseq{X}$, i.e. forgets about all terms in the linear equations corresponding to cycles of levels below $i$.
Furthermore, let $A^{(i),ver}$ denote the submatrix of $A^{(i)}$ only containing rows corresponding to non-horizontal equations.
% In words,  $A^{(i),ver}$ furthermore contains only the equations of $A^{(i)}$ not involving any cross cycles.
We refer to $A^{(i)}$ as \textit{$i$-th level equations} and to $A^{(i),ver}$ as \textit{vertical $i$-th level} equations.

From now on, $x_0$ denotes a point in $M$ and $b_0$ a point in $\partial M\cap\bdComp$.
If not stated otherwise, we denote by $\arc$ a short arc on $(B,B\cap \bdComp)$ connecting $x_0$ and $b_0$. We recall that these notions were defined in~\Cref{section:ShortArcs}. We also denote by $z_0\in\Delta^*$ a $f$-preimage of $x_0$, i.e. $f(z_0)=x_0$.

\begin{definition} A short arc $\arc$ on $(B,B\cap \bdComp)$ is called an {\em M-disk} if $\arc(\Delta^*)\subseteq M_{reg}$.
\end{definition}

\subsection{Monodromy along arcs}\label{sec:MonArcs}
Let $\arc$ be as above.
Then the monodromy of the local system $f^*(\LL|_{B\setminus D})$ can be described directly as follows. For every level $i$ and every horizontal node $e$ we define  $\sigma_i$ and $\sigma_e$ to be the orders of vanishing of $t_i$ and $h_e$, respectively, as functions of $z$. At vertical nodes we set
\[
\sigma_e:= \sum_{i=\ell(e-)}^{\ell(e+)-1} m_{e,i}\sigma_i.
\]
Thus $\sigma_e$ is defined for all nodes $e\in E$ as the vanishing order of $s_e\circ f$ at $z=0$.
We call the tuple  \begin{equation}
\label{eq:sigmaf}
\sigma_{\arc}:=\left((\sigma_i)_{i\in\lvlsetb},(\sigma_e)_{e\in E^{hor}}\right)\in \ZZ^\dimPar
\end{equation}
 the {\em monodromy type} of $\arc$.
We let $T_{\arc}$ be the monodromy of the standard generator on $\Delta^*$ and denote $N_f=I-T_f$ its {\em monodromy logarithm}. We have the explicit equation
\begin{equation}\label{eq:MonArc}\begin{split}
%T_{\arc}:= \prod_{i\in\lvlsetb} T_{i}^{n_i}\cdot \prod_{e\in E^{hor}} T_e^{m_e},\\
N_{\arc}= \sum_{i\in\lvlsetb} \sigma_iN_{i}+ \sum_{e\in E^{hor}} \sigma_eN_e.
\end{split}
\end{equation}
where we recall the monodromy logarithms $N_k$ from \Cref{section:LogPeriodsHodge}.
In particular the monodromy action on the homology $H_1(X_{x_0}\setminus P_{x_0}, Z_{x_0})$  is completely determined by the monodromy type.

We now study one-parameter families of differentials contained in a linear subvariety.
For an arc $\arc:\Delta\to B$
the monodromy of $f^*(\LL)$ is controlled completely by the monodromy type $\sigma_f$. Along a one-parameter family contained in a linear subvariety, the monodromy acts trivially on the defining subspaces and this forces the linear equations for $V$ to be of a special type. That is precisely the content of the next proposition.

\begin{proposition}\label{prop:Vdisk}
Let $\arc$ be an $M$-disk. Then
\[N_f(V)\subseteq V,\]
i.e. the linear subspace defining $M$ is invariant under the monodromy logarithm.
\end{proposition}

\begin{proof}Let $z_0\in\Delta^*$ be a preimage of $x_0$ under $f$ and $y\in\Delta^*$ an arbitrary point.
Furthermore we choose a path  $\gamma:[0,1]\to \Delta^*$ starting at $z_0$ and ending at $y$.
We let \[
S:=\{ t\in \left[0,1\right]\,:\, \exists \text{ open period chart $W\ni\arc(\pa(t))$ such that} \GM_{\gamma(t)}(V)\cap W=M\cap W \}
.
\]
Note that  by abuse of notation we do not distinguish between the vector space $\GM_{\gamma(t)}(V)$ and the analytic variety it defines in a small period chart around $\pa(t)$.
A period chart here is any  open contractible subset $W\subseteq\Stra$ such that periods are injective. For any point $y\in S$, let $W$ be an open period chart as in the definition of $S$, then $W\subseteq S$. Thus $S$ is open and it is also non-empty since $z_0\in S$. 
Let $(t_l)_{l}$ be a sequence in $S$ converging to $t$. After passing to a subsequence we can assume that the whole segment $\arc(\gamma([t_1,t]))$ lies in a contractible period chart $W$ around $\arc(\gamma(t))$. Furthermore, we can choose $W$ such that the analytic variety $\GM_{\gamma(t)}(V)\cap W$ is irreducible.
By assumption there exists a contractible period chart $W_1\subseteq W$ containing $\arc(\pa(t_1))$ such that \[
\GM_{\gamma(t)}(V)\cap W_1= \GM_{\gamma(t_1)}(V)\cap W_1= M\cap W_1
\]
where the first equality follows since $\GM_{\gamma(t)}(V)$ and $\GM_{\gamma(t_1)}$ are obtained from each other via parallel transport along $\pa|_{[t_1,t]}$ and thus both vector spaces define the same analytic variety.
 Since both $\GM_{\gamma(t)}(V)\cap W$ and $M\cap W$ are irreducible it follows that we have equality $\GM_{\gamma(t)}(V)\cap W= M\cap W$.

Let $\gamma'$ be another path connecting $z_0$ and $y$. We then have
\[
\GM_{\gamma}(V)\cap W=M\cap W=\GM_{\gamma'}(V)\cap W
\]
and thus $\GM_{\gamma}(V)=\GM_{\gamma'}(V)$.

The second statement follows by choosing a loop $\pa$ starting at $x_0$. Since $\GM_{\pa}(V)=V$, the monodromy operator $T_f=I-N_f$ sends $V$ to itself and thus $N_f(V)\subseteq V$.
\end{proof}

\begin{remark}\label{rem:MonInv}
\Cref{prop:Vdisk} should be seen as a type of Cylinder deformation theorem, see \cite[Thm. 5.1]{WrightCylinder} in the sense that it constrains the possible linear equations of complex linear varieties. In period coordinates the equations for $M$ are
\begin{equation} \label{eq:LinEq}
\sum_{l=1}^{d} A_{kl}\int_{\gamma_l(b)}\omega(b)=0 \text{ for $k=1,\dots,\codim$},
\end{equation}
and the condition $N_f(V)\subseteq V$ can be written as
\begin{equation}  \label{eq:ResEq}
\sum_{l=1}^d A_{kl}\intSum[\pa_l]  \sigma_er_e(b)=0 \text{ for $k=1,\dots,\codim$}.
\end{equation}
Thus every linear equation for $M$ forces an additional relation between the vanishing cycle periods.
Furthermore, note  that the coefficients of~\cref{eq:ResEq} involve the monodromy type of $\arc$. In particular the monodromy type of $M$-disks is not arbitrary. This is the main motivation for the complicated construction of the log period space $\LPS$ in~\Cref{section:CuttingOut}.
\end{remark}

\begin{remark}
In \Cref{sec:Example} we give an example of how one linear equation forces an additional one. In a subsequent work \cite{BDGCylinder} we study this phenomenon in more detail and obtain several more restrictions among the linear equations. As a consequence we are able to determine the explicit analytic equations defining $\overline{M}$ in a neighborhood of a boundary point, instead of only the defining equations of $\partial M$ as we do in \Cref{thm:Main}. In (loc.cit) we heavily use the results from this paper. If one could compute the analytic equations by other means this would potentially give a much quicker proof of \Cref{thm:Main}, avoiding the technical difficulties of the log period space in \Cref{section:CuttingOut}. In the special case of linear subvarieties defined over the real numbers, in \cite[Theorem 1.9]{BDGCylinder} we use the restrictions on the linear equations to reprove Wright's Cylinder deformation theorem \cite{WrightCylinder}.
\end{remark}

Along $M$-disks we can rewrite the linear equations cutting out $V$ in period coordinates as linear equations in log periods and, this will allow us to take the limit of the linear equations as $z$ goes to zero and to obtain necessary linear equations that are satisfied on the boundary $\partial M$.
\begin{corollary}\label{cor:LimEq}
Let $\arc$ be an $M$-disk. Then the boundary point $b_0=\arc(0)$ lies in the linear subvariety of $U\subseteq\bdComp$ locally defined by the equations
\begin{equation}\label{eq:LimEq}
A^{(i)}\cdot \psi^{\arc}(0) =0 \text{ for every } i\in\lvlset,
\end{equation}
where
$\psi^{\arc}(0):=\left(\psi^{\arc}_{\gamma_k}(0)\right)_{k=1}^{n}$ is the vector of log periods.
\end{corollary}

\begin{proof}
Locally near $x_0$ we know that
\[\begin{split}
0&=\sum_{l=1}^d A_{kl}\int_{\pa_l(z)}\omega\\
&= \sum_{l=1}^d A_{kl}\left(\int_{\pa_l(z)}\omega-\sum_{e\in E}\IP[\pa_l,\lambda_e] \sigma_er_e(b)\ln(z)\right)\\
&=\sum_{l=1}^d A_{kl}\cdot\scl[\ell(l)]\LogP[\pa_l]^f(z),
\end{split}
\]
where the second equality follows from~\cref{eq:ResEq}.
After rescaling each equation by $\frac{1}{\scl[\ell(k)]}$, it follows that the function
\[
\sum_{l=1}^d A_{kl}\frac{\scl[\ell(l)]}{\scl[\ell(k)]}\LogP[\pa_l]^f(z)
\]
is identically zero on $\Delta^*$.
We then take the limit as $z\mapsto 0$.
Since
\[
\lim_{z\to 0} \frac{\scl[\ell(l)]}{\scl[\ell(k)]}=\begin{cases} 1 & \text{ if } \ell(l)=\ell(k), \\
0 & \text{ if } \ell(l)<\ell(k) \end{cases}
\]
we conclude that
\[
\lim_{z\to 0}
\sum_{l=1}^d A_{kl}\frac{\scl[\ell(l)]}{\scl[\ell(k)]}\LogP[\pa_l]^f(z) = \sum_{l=1} A^{(\ell(k))}_{kl}\LogP[\pa_l]^f(0).
\]
\end{proof}

\begin{remark}\label{rem:ExplEq}
\Cref{eq:LimEq} depends not only on the limit point $b_0$, but also on the short arc $\arc$.
On the other hand, if we restrict ourselves to the vertical equations, we can write
\begin{equation}\label{eq:LimEqVer}
A^{(i),ver}\cdot \LogP[]^{\arc}(0)
=\left(\sum_{l:\ell(l)=\ell(k)=i} A^{(i)}_{kl} \int_{(\pa_l)_{\topl}} \eta\right)_{k}=0
\end{equation}
where the index $k$ runs only over non-horizontal equations. Note that ~\cref{eq:LimEqVer} is independent of $\arc$ and only depends on the limit point $b_0$. The goal of the next section
% later (see~\Cref{section:LPS})
is to show that given any boundary point  $b_0 \in \bdComp$ satisfying~\Cref{eq:LimEqVer}, we can choose a short arc $\arc$ such that~\cref{eq:LimEq} are satisfied along $\arc$. I.e. showing that such $b_0$ lies in $\partial M$ and thus proving sufficiency of the linear equations which were shown above to be necessary in~\Cref{cor:LimEq}.
\end{remark}

\begin{remark}[Avoiding the cautionary example]\label{rem:AvoidEx}
In \cite[Section 4]{ChenWrightWYSIWYG} the authors give an example of a continuous family $f:[0,\varepsilon_0)\to \MSDS$ that satisfies certain linear equations for $t\in(0,\varepsilon_0)$ such that the limit at $t=0$ does not satisfy the limit of the equations, which is in stark contrast to \Cref{cor:LimEq}.
The limit $(X_0,\omega_0)$ is a multi-scale differential which contains two horizontal nodes and such that their plumbing parameters  along $f$ behave like $e^{-1/t^2}$ and thus are not real-analytic at $t=0$.
The proof of \Cref{cor:LimEq} breaks down since one cannot fine suitable rescaling parameters $\scl[i]$.  On the other hand, for families that extend real-analytically to the boundary an analogue of \Cref{cor:LimEq} holds, since all periods and plumbing parameters asymptotically grow like a power of the base parameter and are thus comparable to each other.
\end{remark}

\section{The defining equations on the boundary}\label{section:CuttingOut}
This section contains the proof of~\Cref{thm:Main}. The setup of this section is the same as in~\Cref{section:SetupComplexLinear}. We recall what we need to prove.
Given any boundary point $b_0\in \partial M\cap \bdComp$, we need to show that in a small neighborhood $U\subseteq\bdComp$ of $b_0$ the subvariety $\partial M\cap U$ of $U$ is defined by linear equations in generalized period coordinates, as introduced in \Cref{sec:GeneralizedPeriodCoordinates}. In \cref{eq:LimEqVer} we have found a collection of necessary equations satisfied by $\partial M\cap \bdComp\subset\bdComp$ in a neighborhood of $b_0$, and our goal is now to show that these equations define $\partial M\cap \bdComp$, i.e. that any point in $\bdComp$ near $b_0$ satisfying them indeed is contained in $\partial M$.

\begin{definition}
We define $V^{\lim}$ to be the subvariety of $U\subseteq \bdComp$ defined by~\cref{eq:LimEqVer}, that is we define
\[
V^{\lim}:=\left(A^{(i),ver}\cdot \Per^{ver}(\twistD)=0,\, i\in \lvlset\right)
\]
where $\Per^{ver}(\twistD):=\left(\int_{(\gamma_l)_{\topl}}\eta\right)_l$.
\end{definition}
Recall that geometrically this means we take the equations defining $M$, restrict them to each level subsurface of the stable curve, and forget about all horizontal-crossing equations.

The following proposition says precisely that $\partial M\cap\bdComp$ is defined by the linear equations  $V^{\lim}$, i.e. satisfying the linear equations  defining $V^{\lim}$ are both necessary and sufficient conditions for a point of~$\bdComp$ to be contained in $\partial M\cap \bdComp$.
%The following proposition will complete the proof of the main theorem~\Cref{thm:Main}.
\begin{proposition}\label{prop:LinContainment}After possibly shrinking $U$, we have
\[
\partial M\cap U=V^{\lim}.
\]
\end{proposition}

For now we only prove the inclusion $\partial M\cap U\subseteq V^{\lim}$, which follows readily from  ~\Cref{cor:LimEq}. The proof of the remaining inclusion $ V^{\lim}\subseteq \partial M\cap U$ is the core argument, which we will give in~\Cref{section:LPS}.
\begin{proof}[Proof of the containment $\partial M\cap U\subseteq V^{\lim}$]
Let $b_0\in \partial M\cap U$. By~\Cref{lemma:Arc} there exists an $M$-disk connecting $b_0$ and $x_0$. By~\Cref{cor:LimEq}
the limit $\arc(0)=b_0$ satisfies the equations
\[
A^{(i)} \cdot\LogP^{\arc}(0)=0
\]
which in particular implies
\[
A^{(i),ver}\cdot \Per^{ver}(\twistD)=0
\]
as explained in~\Cref{rem:ExplEq}
\end{proof}

\subsection{Proof of the main theorem}
Assuming the proof of \Cref{prop:LinContainment} for now, we show how to finish the proof of our main theorem.

\begin{proof}[Proof of~\Cref{thm:Main}]
We stress that at the moment we still work under the additional assumption that $\partial M$ is locally irreducible near $b_0$. The general case will be handled in \Cref{section:MultComp}.

We recall our setup for convenience. Let $b_0\in\partial M\cap\bdComp$ and $U\subseteq \bdComp$ a period chart containing $b_0$.
To finish the proof we need to exhibit linear equations defining $\partial M$ in a neighborhood of $b_0$.
The content of \Cref{prop:LinContainment} is exactly that $\partial M\cap U$ is defined by the linear equations defining $V^{\lim}$.
\end{proof}

\subsection{The log period space}\label{section:LPS}
Our goal is now to show the remaining inclusion $\partial M\cap U\supseteq V^{\lim}$, after possibly further shrinking $U$. For this we need a new concept, the log period space, which we now motivate. We have already seen in~\Cref{prop:Vdisk} and~\Cref{rem:MonInv} that along one-parameter families the monodromy type of a short arc is restricted by the linear equations for $V$. Instead of working on $\MSDS$, where the monodromy around the boundary is unrestricted, we will thus work on a suitable cover $\LPS$, the \textit{log period space}.
On $U$  the linear equations defining $M$ are only well-defined in a small period chart, and they do not extend to a whole neighborhood of the boundary due to monodromy. But $\LPS$ will be defined in such a way  that the linear equations extend to a whole neighborhood of the boundary and thus define a subvariety $\Vext\subseteq \LPS$. By studying the limiting behavior of the equations for $\Vext$ explicitly, we will be able to prove the inclusion above. We remark that there is not just one log period space, but rather a collection $(\Vext\subseteq \LPS),\, \sigma\in\Sigma$ indexed by the set $\Sigma$ of possible vanishing orders of coordinates $t_i$ and $h_e$ along one-parameter families.
In contrast to $\MSDS$, on $\LPS$ the vanishing of the plumbing parameters $s_e$ is controlled by a single parameter $z$, and the discrete data~$\sigma$ controls how fast each plumbing parameters tends to zero. Thus $\LPS$ has monodromy properties similar to a holomorphic arc.
Before giving the (technical) definition of $\LPS$, we state its properties that we need, and then demonstrate how $\LPS$ can be used to finish the proof of~\Cref{prop:LinContainment}.

\begin{proposition}\label{prop:LPS}
There exists a collection of varieties $(\Vext\subseteq \LPS)_{\sigma\in\Sigma}$ with maps $\LPSmap:\LPS\rightarrow B$ such that
\begin{enumerate}
\item every $M$-disk $\arc$ can be lifted to a short arc $\tilde{\arc}:\Delta\to\Vext$ on $(\Vext, \Vext\cap\LPSmap^{-1}(\bdComp))$ for some $\sigma\in\Sigma$;
\item for every short arc $\tilde{\arc}$ on  $(\Vext,  \Vext\cap\LPSmap^{-1}(\bdComp))$ passing through some preimage of $x_0$ under $\LPSmap$, the composition $\LPSmap\circ\tilde{\arc}$ is an $M$-disk;
\item $\Vext$ is smooth at any point of the preimage $\LPSmap^{-1}(\bdComp)$;
\item the restriction $(\LPSmap)|_{\Vext\cap \LPSmap^{-1}(\bdComp)}$ is open, and $\LPSmap(\Vext)\cap U\subseteq V^{\lim}$.
\end{enumerate}
\end{proposition}
Assuming the above proposition for now, we can prove the other containment in~\Cref{prop:LinContainment}, finishing its proof, and thus also the proof of our main theorem.

\begin{proof}[Proof of~\Cref{prop:LinContainment}]
Proof of the containment $\partial M\cap U\supseteq V^{\lim}$: Choose an $M$-disk $\arc_0$ connecting $b_0$ and $x_0$. By  $(1)$ there exists a lift $\tilde{\arc_0}$ to a short arc on $\Vext$ for some $\sigma$. Let ${\tilde{b}}_0,  \tilde{x}_0$ be some $\LPSmap$-preimages of $b_0,x_0$ contained in $\tilde{\arc}_0(\Delta)$, respectively. Let $Z$ be the irreducible component of $\Vext$ containing $\tilde\arc_0(\Delta^*)$. Since $\Vext$ is smooth at $\tilde{b}_0$ by (3), only one irreducible component of $\Vext$ passes through $\tilde b_0$ and thus there exists an open neighborhood $W\subseteq \Vext$ of $\tilde{b}_0$ contained in $Z$.
We define $U_{b_0}:= \LPSmap(W\cap \LPSmap^{-1}(U))=\LPSmap(W)\cap U$, and note that $U_{b_0}$ is an open neighborhood of $b_0$ by $(4)$.
It remains to show that
\[
\left(\partial M\cap U_{b_0}\right)\supseteq \left(V^{\lim}\cap U_{b_0}\right).
\]
By definition of $U_{b_0}$, for any point $z\in V^{\lim}\cap U_{b_0}$ there exists a $\LPSmap$-preimage $\tilde{z}\in W\subseteq Z$ of $z$. Since $Z$ is irreducible, there exists a short arc on $(Z,Z\cap \LPSmap^{-1}(\bdComp))$ connecting $\tilde{z}$ and $\tilde{x}_0$. Composing with $\LPSmap$ yields an $M$-disk connecting $z$ and $x_0$\, by $(2)$. By the definition of $M$-disks, this shows $z\in \partial M$.
\end{proof}

\subsection{The construction of \texorpdfstring{$\LPS$}{}}
We now start constructing the log period spaces $\LPS$.
In this section we write an element $\sigma\in\ZZ^{\dimPar}$ as \[\sigma=((\sigma_i)_{i\in\lvlsetb}, (\sigma_e)_{e\in E^{hor}}).\]

We consider the positive cone
\[
\calC:=\left\{\sigma\in\ZZ^{\dimPar}\,|\, \sigma_i> 0,\, \sigma_e> 0\right\}\subseteq \ZZ^{\dimPar}.
\]

In analogy to the monodromy logarithm $N_{\arc}$ of a short arc, see~\cref{eq:MonArc}, for any $\sigma\in\calC$ we define the associated {\em monodromy logarithm}
\[
N_{\sigma}:= \sum_{i\in\lvlsetb} \sigma_iN_i + \sum_{e\in E^{hor}} \sigma_eN_e\,.
\]
Additionally, we define the \textit{$V$-preserving cone} to be the set of those monodromy logarithms that preserve~$V$:
\[
\Sigma:=\{ \sigma\in \calC\,|\, N_{\sigma}(V)\subseteq V\}.
\]
This $\Sigma$ will be the index set for $\LPS$ stipulated in the Proposition above.
By ~\Cref{prop:Vdisk} we have $\sigma_{\arc}\in\calC_{V}$ for any $M$-disk $\arc$,  where we recall $\sigma_{\arc}$ from \Cref{sec:MonArcs}.

Our construction of $\LPS$ proceeds in two steps. First we define a covering space $\LPSmap:\LPSpre\to B\setminus D$, and then construct $\LPS$ by adding suitable limit points to $\LPSpre$  such that the map extends to a holomorphic map $\LPSmap:\LPS\to B$.

We start by describing $\LPSpre$.
For any $\sigma\in\Sigma$, we let $\LPSmap:\LPSpre\to B\setminus D$ be the covering of $B\setminus D$ corresponding to the cyclic subgroup $\langle\sigma\rangle\subseteq\ZZ^{\dimPar}=\pi_1(B\setminus D)$. Denote coordinates on $\Delta^*\times \CC^{\dimPar-1}\times \Delta^{\dimU}$ by
\[
\tilde{b}=(z,\nu=((\nu_i)_{i\in\lvlsetb},\chi=(\chi_e)_{e\in E^{hor}}),\twistD).
\]
If $\enhancG$ has at least two levels, we set $\nu_{\ell(\enhancG)}=0$.
On the other hand, if $\ell(\enhancG)=0$, we choose one horizontal node $e_0\in E^{hor}$ and set $\chi_{e_0}=0$. This notation will simplify the following formulas.

Explicitly, we can describe $\LPSpre\subseteq \Delta^*\times \CC^{\dimPar-1}\times \Delta^{\dimU}$ as the domain (that is, open connected subset) given by
\begin{gather*}
\LPSpre:=\left\{ (z,\nu,\chi,\twistD),|\, \im \nu_i > \tfrac{\sigma_i}{2\pi}\log |z|,\,\im \chi_e > \tfrac{\sigma_e}{2\pi}\log |z|\right\}.
\end{gather*}

Note that $\LPSpre$ is diffeomorphic to $\Delta^*\times \HH^{N-1}\times \Delta^{M}$, since the conditions on the imaginary parts define a family of smoothly varying horizontal half-planes over the punctured disk,
and thus in particular $\pi_1(\LPSpre)\simeq \ZZ$.

The covering map
\[
\LPSmap:\LPSpre\to \Delta^*\times (\Delta^*)^{\dimPar-1}\times\Delta^{\dimU}=B\setminus D
\]
is explicitly given by
\[
z=z,\,t_i=z^{\sigma_i}e^{2\pi i\nu_i},\,h_e=z^{\sigma_e}e^{2\pi i\chi_e},\twistD=\twistD.
\]
Additionally, the universal cover $\HH\times\HH^{\dimPar-1}\times\Delta^{\dimU}\to \LPSpre$ is given by
\[
z=e^{2\pi i\tau},\, t_i=\alpha_i-\sigma_i\tau,\, h_e=\beta_e-\sigma_e\tau,\, \twistD=\twistD
\]
where $(\tau,(\alpha_i),(\beta_e),\twistD)$ are the coordinates on $\HH\times\HH^{\dimPar-1}\times\Delta^{\dimU}$.
At horizontal nodes we set $s_e(\tilde{b}):=h_e(\tilde{b})=z^{\sigma_e}e^{2\pi i\chi_e}$ and we are now going to also define functions $s_e:\LPSpre\to \CC$ at vertical nodes. For any vertical node $e$ we define
\begin{gather}\label{eq:ThetaDef}
\sigma_{e}:= \sum_{i=\ell(e-)}^{\ell(e+)-1} m_{e,i}\sigma_{i}, \quad
\chi_{e}:= \sum_{i=\ell(e-)}^{\ell(e+)-1} m_{e,i}\nu_i,\\
s_e(\tilde{b}):= z^{\sigma_e}e^{2\pi i\chi_e}.
\end{gather}
Here $\sigma_e$ and $\chi_e$ are defined such that the relation
\[
s_e=\prod_{\ell(e-)}^{\ell(e+)-1} t_i^{m_{e,i}}
\]
is satisfied, where $m_{e,i}$ was defined by \cref{eq:ak}.
The fact that $N_{\sigma}$ preserves $V$ is then  equivalent to
\[
\sum_{l=1}^d A_{kl} \intSum[\pa_l] r_e(b)\sigma_{e}=0 \text{ for all } k=1,\ldots,\codim.
\]
Note that this follows from \cref{eq:ResEq} together with \cref{eq:ThetaDef}.

We let finally
\[
\LPS:=\LPS^{\circ}\sqcup \left(\{0\}\times \CC^{\dimPar-1}\times \Delta^{\dimU}\right)\subseteq \Delta\times\CC^{\dimPar-1}\times \Delta^{\dimU}.
\]
Observe that $\LPS=\interior(\overline{\LPS^{\circ}})\subseteq \Delta\times\CC^{\dimPar-1}\times\Delta^{\dimU}$  and thus $\LPS$ is open.

\begin{remark} The space $\LPSpre$ can be seen as a family of products of horizontal half-planes $\{\im z>c(b)\}$ parametrized over the punctured disk with $\lim_{b\to 0}c(b)=-\infty$. Each half-plane becomes a copy of $\CC$ in the limit $b\mapsto 0$ and taking the interior closure of $\LPSpre$ fills in the limiting copies of $\CC$.
\end{remark}
Furthermore, since $\LPSmap:\LPS^{\circ}\to B\setminus D$ is the restriction of a holomorphic map $\CC^{\dimPar+\dimU}\to\CC^{\dimPar+\dimU}$, it extends to a holomorphic map of the closures $\LPS\to B$, which we still denote $\LPSmap$.
The \textit{boundary} $\tilde{D}$ of $\LPS$ is
\[
\tilde{D}:=\{z=0\}=\LPSmap^{-1}(B\cap \bdComp)=\LPS\setminus\LPSpre\subseteq\LPS.
\]

\subsection*{Arc log periods}
Now that we have explicitly described the log period space $\LPS$, we describe a variant of log periods which is suitably adapted to $\LPS$.

\begin{definition}\label{def:ArcLogPeriod}
We define the {\em arc log period} $\LogP^{\Delta}:\LPS\to \CC$ by
\[
\LogP^{\Delta}(\tilde{b}):= \dfrac{1}{\scl[\topl(\pa)]}\left[\int_{\gamma(\tilde{b})}\omega(\tilde{b})-\left(\sum_{e\in E}\IP[ \pa,\lambda_e] r_e(\tilde{b})\right)\sigma_{e}\ln(z))\right]
\]
where $\sigma_{e}$ is defined by~\cref{eq:ThetaDef}.
\end{definition}
As in the case of one-parameter families \Cref{cor:LogPerOPF}, we can use  the asymptotics of log periods from ~\Cref{thm:PerThm} to obtain the limit of arc log periods at the boundary $\widetilde{D}$.

\begin{proposition}\label{prop:ArcLogPer}
The arc log period $\LogP^{\Delta}:\LPS\rightarrow \CC$ is single-valued and analytic. Furthermore,
\[
\LogP^{\Delta}(0,\nu,\chi,\twistD)=
 \left[\int_{\pa_{\topl}}\Hol(\twistD) + \sum_{e\in E}\IP[ \gamma_{\topl},\lambda_e] \res_{q_e^+}(\twistD)(\chi_e+\tilde{c}_e)\right].
\]
where
 $\tilde{c}_e$ are certain constants, depending only on the choice of normal form coordinates and branches of logarithms.
\end{proposition}
\begin{proof}
We write $b=\LPSmap(\tilde{b})$ for the rest of the proof.
For all nodes $e$, there exist integers $k_e'$ such that \[
\ln(s_e(\tilde{b}))=\sigma_{e}\ln(z)+\chi_e+2\pi ik_e'
\] by \cref{eq:ThetaDef}.
We thus have
\[
\begin{split}
\LogP(b) &= \dfrac{1}{\scl[\topl(\pa)]}\left [ \int_{\pa}\omega(b)- \sum_{e\in E}\IP[ \pa,\lambda_e] r_e(b)\ln(s_e)\right]\\
&= \LogP^{\Delta}(\tilde{b})-\sum_{e\in E} \IP[\pa,\lambda_e]\dfrac{r_e(b)}{\scl[\topl(\pa)]} \left(\chi_e+2\pi ik_e'\right)
\end{split}
\]
Thus the result follows from~\Cref{thm:PerThm} with $\tilde{c}_e:=c_e+2\pi ik_e'$.
\end{proof}
\subsection*{The subvariety \texorpdfstring{$\Vext$}{}}
We now come to the definition of $\Vext\subseteq\LPS$. On the stratum we can only define the linear equations defining $M$ in a small period chart. Due to monodromy, periods do not extend as holomorphic functions to the boundary $\partial\MSDS$. On the other hand, we have seen that log periods do extend to $\MSDS$. Thus na\"\i vely one would try to convert the linear equations defining $M$ into equations involving log periods. The na\"\i ve idea does not work since the logarithmic divergences do not cancel out.
The space $\LPS$ is constructed in such a way that the logarithmic divergences cancel out, and thus we will be able to rewrite linear equations in period coordinates as equations in arc log periods. We let
 $A'=(A'_{kl})_{1\leq k\leq \codim, 1\leq l\leq d}$ be the matrix with
\[
A'_{kl}:= \dfrac{\scl[\ell(l)]}{\scl[\ell(k)]}A_{kl}
\]
being the equations for $V$, suitably rescaled, and define
\begin{equation}\label{eq:Vtil}
\Vext:=\left\{\tilde{b}\in\LPS\,:\,A'\cdot\LogP[]^{\Delta}(\tilde{b})=0\right\}\subseteq \LPS
\end{equation}
where $\LogP[]^{\Delta}:=(\LogP[\gamma_k]^{\Delta}(\tilde{b}))_{k}$.
The rescaling factors in the definition of $A_{kl}'$  are  motivated by the proof of ~\Cref{cor:LimEq}.

The next result says that, over a period chart, $\Vext$ is just the $\LPSmap$-preimage of $V$.

\begin{proposition}\label{prop:LPSEq}For any sufficiently small period chart $W\subseteq B\setminus D$ containing $x_0$ we have
\begin{align}
\LPSmap(\Vext)\cap W&=V\cap W,\label{eq:LPSEq1}\\
\LPSmap^{-1}(V\cap W)&= \Vext \cap \LPSmap^{-1}(W). \label{eq:LPSEq2}
\end{align}
\end{proposition}

\begin{proof} For any $\tilde{b}\in \LPSmap^{-1}(W)$ we compute
\[
\sum_{l=1}^d A_{kl}\scl[\ell(l)]\LogP[\gamma_l]^{\Delta}(\tilde{b})=\sum_{l=1}^d A_{kl}\left[\int_{\gamma_l}\omega-\intSum[\pa_l] \sigma_er_e(b)\ln(z)\right].
\]
Since the matrix $A$ are the defining linear equations for the linear subvariety $M$ near $x_0$, it follows from~\Cref{prop:Vdisk} or equivalently~\cref{eq:ResEq}
that \[\sum_{l=1}^d A_{kl}\sum_{e\in E}\IP[\pa_l,\lambda_e] \sigma_er_e(b)=0 \text{ for } k=1,\ldots,\codim.\]
Thus $\LPSmap^{-1}(V\cap W)=\Vext\cap \LPSmap^{-1}(W)$ and the first claim follows since $W$ is contained in the image of $\LPSmap$.
\end{proof}

We now study the limiting behavior of the equations defining $\Vext$ on the boundary $\widetilde{D}$ of $\LPS$, for arbitrary $\sigma$.

Consider one of the defining equations
$\sum_{l=1}^{d} A'_{kl}\LogP[\gamma_l]^{\Delta}=0$  of  $\Vext$ and restrict it to $\widetilde{D}=\{z=0\}$. In the limit $z\mapsto 0$ only the arc log periods $\LogP[\gamma_l]^{\Delta}$ with $\ell(l)=\ell(k)$ contribute. Thus, using \Cref{prop:ArcLogPer} the equations for $\Vext\cap\widetilde{D}$ can be written as
\begin{equation}\label{eq:Vlimeq}
\sum_{\{l: \lvl[(l)]=\lvl[(k)]\}}A_{kl} \left[\int_{(\gamma_l)_{\topl}} \Hol(\twistD)+\intSum[(\pa_l)_{\topl}]\res_e(\twistD)(\chi_e+\tilde{c}_e)\right]=0.
\end{equation}
Thus on the boundary $\widetilde{D}$, the equations for $\Vext$ and for $\LPSmap^{-1}(V^{\lim})$ coincide except for the equations involving horizontal nodes. As a corollary of this discussion we obtain

\begin{corollary}\label{cor:Vimage}
 The image $\LPSmap(\Vext\cap \widetilde{D})$ is contained in $V^{\lim}$.
\end{corollary}

The following step is crucial in the proof of property $(3)$ of \Cref{prop:LPS}.
\begin{proposition}\label{prop:subm} For any $\tilde{b}_0\in \Vext\cap\widetilde{D}$,
the subvarieties $\Vext$ and $\Vext\cap\widetilde D$ are smooth at $\tilde{b}_0$ and furthermore the restriction \[
(\LPSmap)|_{\Vext\cap\widetilde{D}}:\Vext\cap\widetilde D\rightarrow V^{\lim}
\]
 is a submersion at $\tilde{b}_0$.
\end{proposition}

\begin{remark}
This Proposition is the key technical component of the proof of \Cref{thm:Main}. The proof uses both the asymptotic analysis for log periods as well as the notion of $\enhancG$-adapted basis in a crucial way.
\end{remark}

\begin{proof}
For the rest of the proof we write $\boldsymbol\ell=\ell(\enhancG)$.
We start with the smoothness of $\Vext$.
We choose a $\enhancG$-adapted basis
\begin{gather*}
\{\pa_1,\ldots,\pa_d\}=\\
\left\{\delta_{1}^{(0)},\ldots,\delta_{c(0)}^{(0)}, \alpha^{(0)}_{1},\ldots,\alpha^{(0)}_{d(0)},\ldots, \delta_{1}^{(\boldsymbol\ell)},\ldots,\delta_{c(L)}^{(\boldsymbol\ell)},\alpha^{(\boldsymbol\ell)}_{1},\ldots,\alpha^{(\boldsymbol\ell)}_{d(\boldsymbol\ell)}\right \},
\end{gather*}
where we recall that $\alpha^{(i)}_{1},\ldots,\alpha^{(i)}_{d(i)}$ are non-horizontal cycles of level $i$ and $\delta_{1}^{(i)},\ldots,\delta_{c(i)}^{(i)}$ are \hcc of level $i$.
We will write $\chi_{l}^{(i)}$ instead of $\chi_e$ where $e$ is the unique horizontal edge crossed by $\delta_{l}^{(i)}$.
 For each level $i$ we order the $\enhancG$-adapted basis in such a way that $\int_{(\alpha^i_{d(i)})_{\topl}} \twistD\neq 0$.
Additionally, if $\enhancG$ has only one level, we also arrange that $\delta_{d(0)}^{(0)}$ crosses only the horizontal node  $e_0$, where $\chi_{e_0}$ is the omitted coordinate on $\LPS$.

Let $F_1,\ldots, F_{\codim}$ be the defining equations for $\Vext$ in $\LPS$, considered as functions on $\LPS$. Our goal is to show that the Jacobian matrix of the $F_1,\dots, F_{\codim}$ has full rank with respect to a suitable coordinate system on $\LPS$.
For this we recall that $\LPS$ has different coordinates, depending on whether $\enhancG$ has only one or multiple levels.

In the case of only one level we can describe a coordinate system as follows. We choose a horizontal edge $e_0$ such that the coordinate $\chi_{e_0}$ is omitted.
Then $z, \chi_e$ for $e\in E^{hor}\setminus \{e_0\}$ and $\int_{(\pa_l^{(0)})_{\topl}}\twistD$ for  $l=1,\ldots,d(0)$ are coordinates on $\LPS$.

On the other hand if $\enhancG$ has multiple levels, coordinates on~$\LPS$ are given by
$z,\nu_i$ for $i=1,\ldots \boldsymbol\ell-1$, $\chi_e$ for $e\in E^{hor}$ and
%$ \int_{(\alpha_l^{(0)})_{\topl}}$ for $l=1,\ldots,d(0)$ ,
 $\int_{(\alpha_l^{(i)})_{\topl}}\twistD$ for $i=0,\ldots,\boldsymbol\ell$ and $l=1,\ldots,d(i)-1$ as well as $\int_{(\alpha_{d(0)}^{(0)})_{\topl}}\eta$.

We are now going to compute the Jacobian with respect to the coordinate systems just described.
% equations.
According to \cref{eq:Vtil} we can write
\[
F_k=\sum_{l=1}^d A_{kl}\frac{\scl[\ell(l)]}{\scl[\ell(k)]}\LogP[\gamma_l]^{\Delta}\,,
\]
where $A_{kl}$ are the coefficients of the linear equations defining $V$.
We assume, as always, that the matrix $A=(A_{kl})$ is in reduced row echelon form and we denote $A_{kp(k)}$ the pivot of the $k$-th row.
Each pivot $A_{kp(k)}$ corresponds to some element $\gamma_{p(k)}$ of the $\enhancG$-adapted basis.
For the rest of the proof we write $u:=\codim$ and we let $F_1,\ldots,F_{u'}$ be the linear equations such that $\gamma_{p(k)}$ is a horizontal-crossing cycle  and $F_{u'+1},\ldots,F_u$ the remaining equations.
In the former case, we let $e(k)$ be the unique horizontal edge crossed by $\gamma_{p(k)}$.

We now distinguish the two cases described above. First we assume that $\enhancG$ has more than one level.

For every $1\leq k\leq u'$ we can write, using \Cref{prop:ArcLogPer},
\[
F_k=r_{e(k)}(\chi_{e(k)}+\widetilde{c}_{e(k)})+ h_k(\tilde{b})+ zg_k(\tilde{b})
\]
where $g_k, h_k$ are analytic. Furthermore by inspecting \Cref{prop:ArcLogPer} closely we see that
\[
h_k(\tilde{b})=h_k\left( \chi^{(\ell(k))}_{p(k)+1},\ldots,\chi^{(\ell(k))}_{c(k)}, \int_{\left(\alpha^{(\ell(k)}_{1}\right)_{\topl}} \eta, \ldots, \int_{\left(\alpha^{(\ell(k))}_{d(k)}\right)_{\topl}}\twistD\right).
\]
Similarly, the remaining equations $F_{u'+1},\dots, F_u$ can be written near $\tilde{b}_0$ as
\[
F_k= \int_{\left(\alpha^{(\ell(k)}_{p(k)}\right)_{\topl}} \eta + h_k\left( \int_{\left(\alpha^{(\ell(k)}_{p(k)+1}\right)_{\topl}} \eta, \ldots, \int_{\left(\alpha^{(\ell(k)}_{d(k)}\right)_{\topl}} \eta\right) +zg_k(\widetilde{b})
\]
where again $h_k$ and $g_k$ are analytic.

We note that in this case $p(k)\neq d(i)$ since otherwise $\int_{\left(\alpha^{(i)}_{d(i)}\right)_{\topl} } \eta =0$, which can be seen by taking the limit of the equations at $b_0$. Thus for any equation $F_k, \, k> u'$ the pivot $\int_{\left(\alpha^{(\ell(k)}_{p(k)}\right)_{\topl}} \eta$ is a coordinate on $\LPS$.
Thus the submatrix of the Jacobian corresponding to $\chi_{e(k)}$ for $k=1,\ldots, u'$ and $\int_{\left(\alpha^{(\ell(k)}_{p(k)}\right)_{\topl}} \eta$ for $k=u'+1,\ldots, u$ has full rank at $\tilde{b}_0$ and therefore
 $\Vext$ is smooth at $\widetilde{b}_0$.
Smoothness of $\Vext\cap\widetilde{D}$ follows similarly, by noting that additionally $z$ is one of the coordinates on $\LPS$.

The argument is very similar in the second case where $\enhancG$ has only one level, with some care needed to make sure everything works out well for the omitted coordinate $\chi_{e_0}=\chi_{c(0)}^{(0)}$. We recall that  we ordered the $\enhancG$-adapted basis in such a way that $\chi_{e(i)}$ is the omitted coordinate on $\LPS$. We claim that $\delta_{c(0)}^{(0)}$ does not correspond to any of the pivots $p(k)$, since otherwise we would have $\res_{e_0}(\twistD)=0$ by~\Cref{rem:MonInv}, which is impossible.
Thus, as before, each pivot corresponds to a coordinate on $\LPS$, and thus the Jacobian has full rank.

We now come to the final claim that $(\LPSmap)|_{\Vext\cap\widetilde{D}}:\Vext\cap\widetilde D\rightarrow V^{\lim}
$
 is a submersion at $\tilde{b}_0$.
Let $\Omega\subseteq \{1,\dots,d\}$ be the set of all non-pivotal rows and let $\Omega'\subseteq \Omega$ be the non-pivots corresponding to  cross cycles $\delta_e^{(i)}$.
In particular we can then use the periods $\lbrace \int_{(\pa_l)_{\topl}} \twistD\rbrace$ for $l\in \Omega\setminus\Omega'$
together with $\chi_e$ for $e\in \Omega'$ and $\nu_i$ for $i\in\lvlsetb$ as local coordinates on $\Vext\cap \widetilde{D}$.

Similarly, we can use $\int_{(\pa_l)_{\topl}} \twistD$ for all $l\in \Omega\setminus\Omega'$, as coordinates on $V^{\lim}$, and thus $\LPSmap$ is a submersion near $\tilde{b}_0$.
\end{proof}

\subsection*{The proof of \texorpdfstring{\Cref{prop:LPS}}{}}
We now have all the necessary ingredients for the proof of \Cref{prop:LPS}, it is a matter of summarizing what we have proved so far.
\begin{proof}[Proof of \Cref{prop:LPS}]
We have seen in \Cref{prop:subm} that $\Vext$ is smooth at any point of $\LPSmap^{-1}(\bdComp)$, thus proving~(3).  Furthermore,  \Cref{prop:subm} also shows that $(\LPSmap)|_{\Vext\cap \LPSmap^{-1}(\bdComp)}$ is open and maps into $V^{lim}$ by \Cref{cor:Vimage}, and we have thus proved~(4).

We now address the lifting properties of short arcs~(1) and~(2).
Let  $\sharc$ be a short arc with $\sigma_{\arc}=\sigma$.
Since $\arc_{*}(\pi_1(\Delta^*,x_0))=\langle \sigma_{\arc}\rangle$, there exists a lift $\tilde{\arc}^{\circ}:\Delta^*\to\LPS$.
More explicitly, we can write
\begin{equation}\label{eq:Lift}
t_i=z^{\sigma_i}e^{2\pi i\nu_i(z)},s_e=z^{\sigma_e}e^{2\pi i\chi_e(z)},\twistD=\twistD(z),
\end{equation}
where $\nu_i$ and $\chi_e$ are holomorphic. And in particular it follows that $\nu_i$ and $\chi_e$ are holomorphic at $z=0$.
Recall that on $\LPS$ there exists either a level $i$ with  $\nu_i=0$ or a horizontal node $e$ with $\chi_e=0$. We assume that $\nu_i=0$ for some $i$, the other case can be treated analogously. After a change of coordinates $z\mapsto ze^{2\pi i\tfrac{\nu_i(z)}{\sigma_i}}$ we can arrange that $\nu_i(z)=0$
and then define
\[
\tilde{\arc}^{\circ}(z):=(z,(\nu_i(z))_i,(\chi_e(z))_e,\twistD).
\]
Since $f$ is holomorphic at the origin, $\tilde{\arc}^{\circ}$ extends to a short arc $\tilde{\arc}:\Delta\to B$.
Now suppose $\arc$ is an $M$-disk and $W$ a period chart containing $x_0$, then $\arc(\Delta)\cap W\subseteq V\cap W$ and by \cref{eq:LPSEq2} the lift $\tilde{\arc}$ maps into $\Vext$, thus showing~(1).

Similarly, if $g:\Delta\to \Vext$ is a short arc on $(\Vext,\Vext\cap \widetilde{D})$ passing through a preimage of $x_0$ on, then by~\cref{eq:LPSEq2} the composition $\LPSmap\circ g$ is an $M$-disk. This proves~(2).
\end{proof}

\subsection{Multiple components}
\label{section:MultComp}
So far we assumed that $\overline{M}$ is locally irreducible near $b_0$.
In general  we can write $\overline{M}=\cup_{\alpha} M_{\alpha}$ locally near $b_0$ where $M_\alpha$ are the finitely many, local irreducible components of $\overline{M}$. For every $\alpha$ we choose a base point $x_\alpha$ and a subspace $V_\alpha$ such that $M_\alpha$ coincides with $V_\alpha$ near $x_\alpha$.
We can then apply~\Cref{prop:LinContainment} to each irreducible component $M_\alpha$ and thus obtain
\[
\partial M\cap U= \bigcup_{\alpha} V_\alpha^{\lim}
\]
for a suitable neighborhood $U$. In particular $\partial M$ is defined by a finite union of linear subspaces at any boundary point $b_0\in\partial M\cap \bdComp$ and this finishes the proof of \Cref{thm:Main} for the case of multiple components.

\section{An example}\label{sec:Example}
We now demonstrate how to obtain the linear equations on the boundary from the linear equations on a nearby smooth surface in an example.
We stress that we do not claim that there exists an actual linear subvariety which is locally defined by those linear equations; the example is only hypothetical.
\begin{figure}
\centering
\includegraphics[scale=0.22]{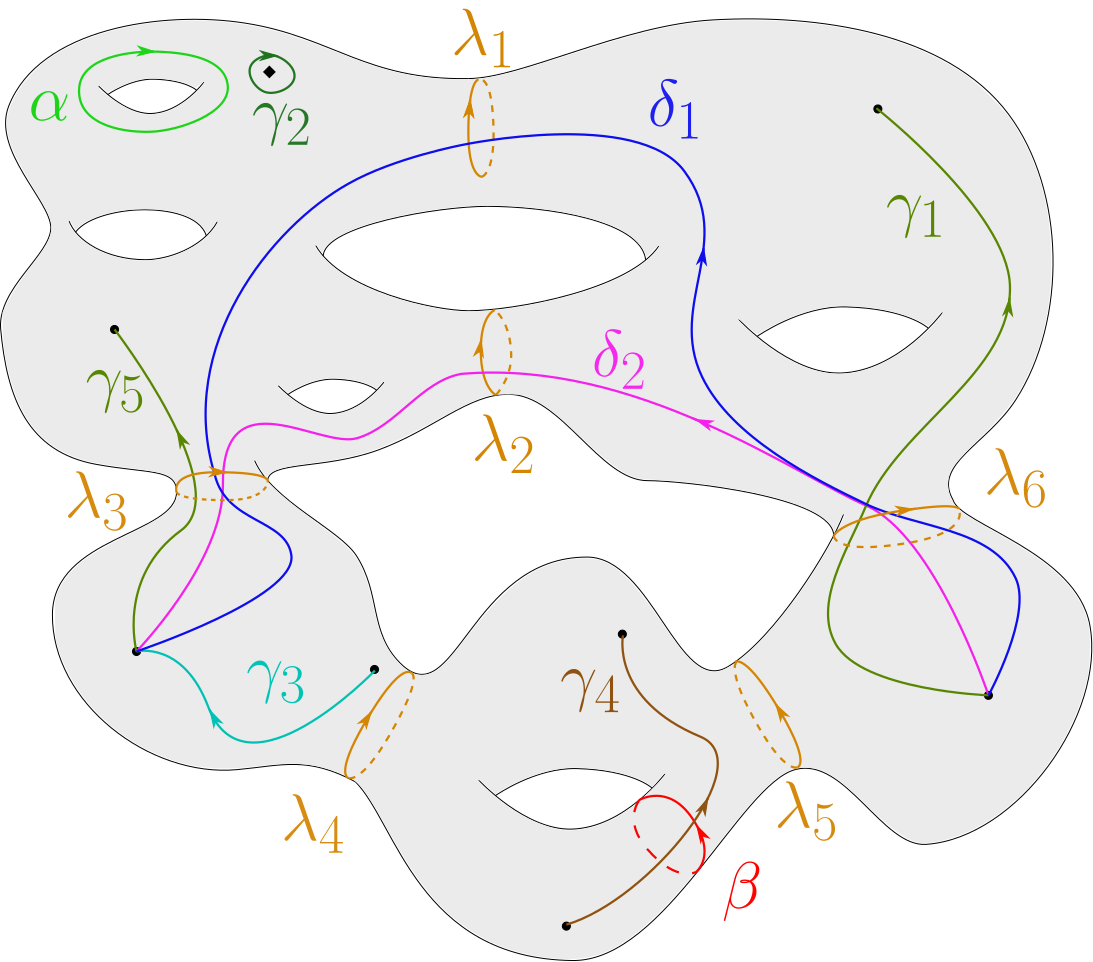}
\caption{A smooth genus $7$ curve}
\label{fig:EqExample}
\end{figure}
In~\Cref{fig:EqExample} we see a smooth genus $7$ curve $\Sigma$, just chosen sufficiently complicated to illustrate all possible phenomena. We consider the degeneration $X$ obtained by simultaneously pinching the cycles $\lambda_i,\, i=1,\ldots, 6,$ with normalization $\widetilde{X}\to X$.
By abuse of notation we denote homology cycles on $X$ and $\Sigma$ with the same name.
The level structure on $\enhancG$ can be seen in \Cref{fig:lG3}.
\begin{figure}[H]
\centering
\includegraphics[scale=0.38]{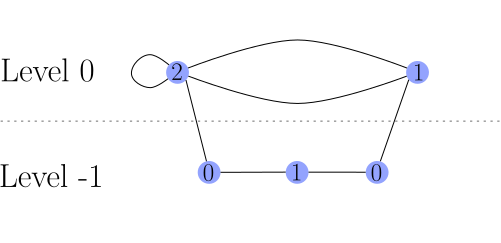}
\caption{The level graph $\enhancG$}
\label{fig:lG3}
\end{figure}

We note that the image of the vanishing cycles in
$H_1(\Sigma\setminus P,Z)$ is generated by $\langle \lambda_1,\lambda_2\rangle$. Furthermore, the images of $\{\alpha,\gamma_1,\gamma_2,\delta_1,\delta_2,\gamma_3,\gamma_4,\beta\}$ on the stable curve $X$ can be extended to a $\enhancG$-adapted basis. The advantage of using a $\enhancG$-adapted basis is that we can read off the equations directly. We assume that at all vertical nodes the number of prongs is one, i.e. $\kappa_e=1$.

Suppose $M\subseteq \Stra$ were a linear subvariety which locally near $\Sigma$ is given by the linear equations

\begin{gather}
\int_{\alpha}\omega+\int_{\gamma_2}\omega+3\int_{\beta}\omega=0,\label{eq:Lin1}\\
\int_{\gamma_1}\omega+\int_{\gamma_5}\omega=0,\label{eq:Lin2}\\
3\int_{\delta_1}\omega-5\int_{\delta_2}\omega=0,\label{eq:Lin3}\\
3\int_{\lambda_1}\omega-10\int_{\lambda_2}\omega=0.\label{eq:Lin4}\\
\int_{\gamma_3}\omega=\int_{\gamma_4}\omega.\label{eq:Lin5}
\end{gather}

Before describing the linear equations of $\partial M\cap\bdComp$ we describe the implications of  \Cref{prop:Vdisk} in this case.
Since \cref{eq:Lin3} crosses the horizontal vanishing cycles $\lambda_1$ and $\lambda_2$  as well as the vertical vanishing cycles $\lambda_3$ and $\lambda_6$
there has to be an additional equation of the form \begin{equation}\label{eq:forced}
\begin{split}
3\left(m_1\int_{\lambda_1}\omega+ m_3\int_{\lambda_3}\omega+m_3\int_{\lambda_6}\omega\right)-5\left(m_2\int_{\lambda_2}\omega+m_3\int_{\lambda_3}\omega+m_3\int_{\lambda_6}\omega\right).\\
%3m_1\int_{\lambda_1}\omega-5m_2\left(-\int_{\lambda_1}\omega-\int_{\lambda_3}\omega\right)=0
\end{split}
\end{equation}
for some positive integers $m_1,m_2$ and $m_3$.
Here we used that the number of prongs at $\lambda_3$ and $\lambda_6$ is both one, thus the coefficient $m_3$ is the same for both of them.
Note that $\lambda_3+\lambda_6=0$ since the sum of the two vanishing cycles is separating.
We decide for the sake of an example, that $m_1=1,m_2=2$ and thus \cref{eq:forced} reduces to \cref{eq:Lin4}.
%And for this example we assume that $m_1=m_2=1$.
Similarly, the equation $\int_{\gamma_1}\omega+\int_{\gamma_5}\omega=0$ forces a linear equation
\[
m_3\int_{\lambda_3} \omega+ m_3\int_{\lambda_6} \omega=0.
\]
Since $\lambda_6=-\lambda_3$ this equation is vacuously true and thus does not impose an additional constraint.

\begin{figure}[H]
\subfloat{
\includegraphics[scale=0.26]{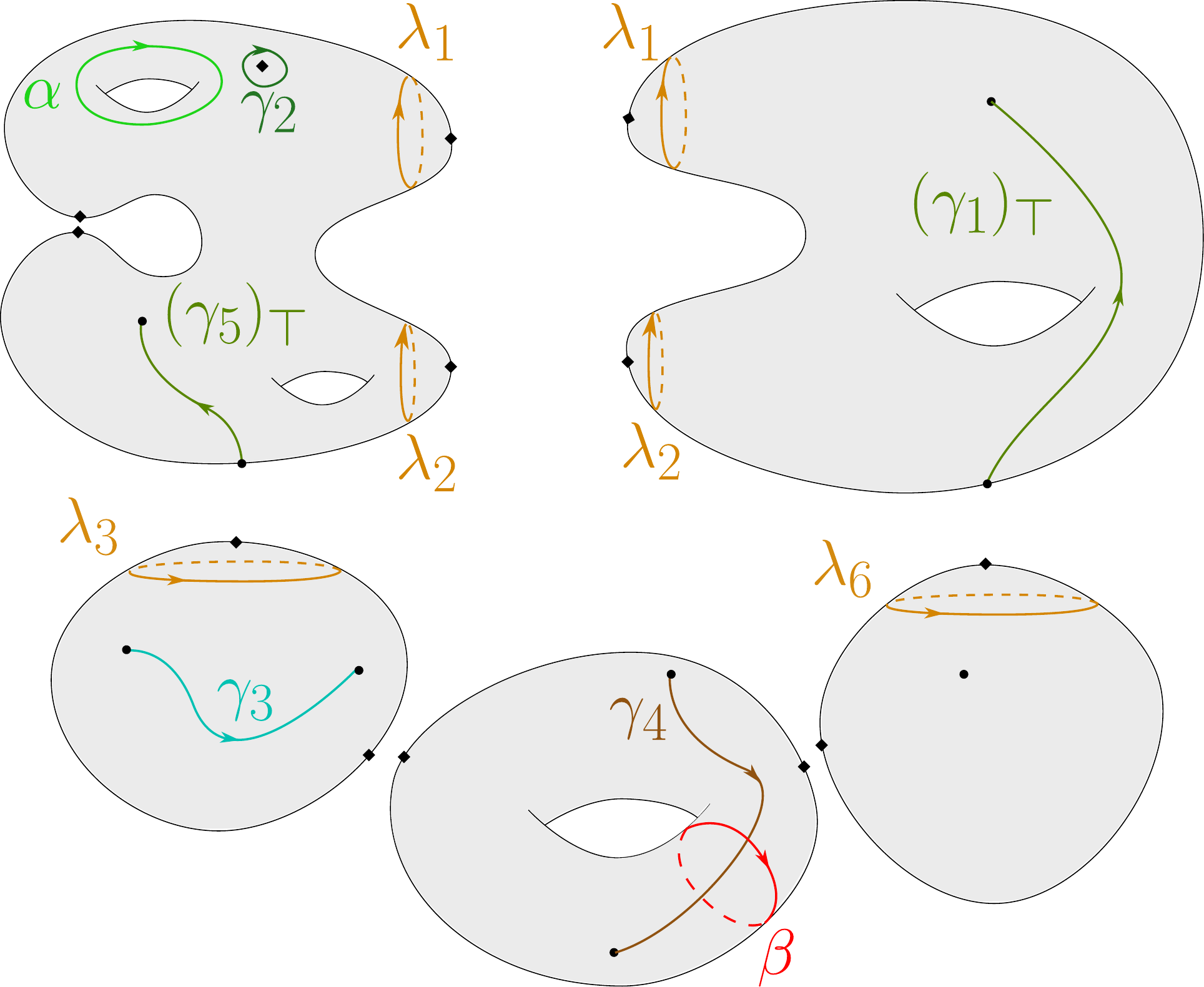}
}
\caption{The normalization $\widetilde{X}$}
\label{fig:EqSingular}
\end{figure}

We now describe the linear equations defining $\partial M\cap\bdComp$ near $X$.
For each of the defining equations $F\in H_1(\Sigma\setminus P,Z)$ for $M$ we repeat the following steps.
\begin{enumerate}
\item Determine the top level $\topl(F)$ and write $F=\sum_{k}d_{k}\pa_k$ in a $\enhancG$-adapted basis
\item If the equation $F$ is a \hcc, delete it.
\item Otherwise, restrict $F$ to its top level $\topl(F)$, i.e.  we consider $f_{\topl(F)}(F)$ in the language of \Cref{section:Filtr}. The resulting cycle then defines an equation for the boundary $\partial M\cap \bdComp$ at level $\topl(F)$.
\end{enumerate}

 First, since~\cref{eq:Lin3} crosses the horizontal vanishing cycles $\delta_1$ and $\delta_2$, we omit it for the equations of $\partial M\cap\bdComp$.
 We then restrict all remaining equations to their respective level. The equations \cref{eq:Lin1,eq:Lin2,eq:Lin4} are of level $0$, while \cref{eq:Lin5} is of level $-1$.
When restricting~\cref{eq:Lin1} we loose $\int_{\beta}\omega$ since $\ell(\beta)=-1$.
We thus arrive at the following equations

\begin{gather}
\int_{\alpha}\twistD+\int_{\gamma_2}\twistD=0,\label{eq:Res1}\\
\int_{(\gamma_1)_{\topl}}\twistD-\int_{(\gamma_5)_{\topl}}\twistD=0,\label{eq:Res2}\\
3\int_{\lambda_1}\twistD-10\int_{\lambda_2}\twistD=0.\label{eq:Res4}\\
\int_{\gamma_3}\twistD=\int_{\gamma_4}\twistD.\label{eq:Res5}
\end{gather}
We remark that the global residue condition for $\enhancG$ imposes the additional constraint ~$\int_{\lambda_3} \twistD=\int_{\lambda_6}\twistD$, that does not come from the linear equations for $M$ but simply from the fact that $\lambda_3$ and $\lambda_6$ are homologous.

\section{Comparing the linear structures}
\label{section:LinearStructures}
We now give a coordinate-free way of interpreting the linear equations for $\partial M$ on the boundary, which will give the final and most precise formulation of \Cref{thm:Main}. Again we suppose that $\overline{M}$ is locally irreducible near $b_0$ and for the general case we just apply the results on each local irreducible component separately.
We still use the same setup as in~\Cref{section:SetupComplexLinear}.
In the course of the proof of ~\Cref{thm:Main} we described how to obtain the subspace defining $\partial M$ in terms of defining equations.
% If a linear equation has top level $i$ and crosses a horizontal vanishing cycle of level $i$, we forget this equation. If not then we simply restrict the equation to the $i$-th level $X_{(i)}$.

%\subsection{The specialization map}
%\label{section:RelatingLinear}
Suppose $T_{x_0}M=V\subseteq H^1(X_{x_0}\setminus P_{x_0},Z_{x_0};\CC)$ is the linear subspace defining $M$ near $x_0$.  We consider an equation in $H^1(\Sigma\setminus P,Z;\CC)$ as an element of the dual $H_1(\Sigma\!\setminus\! P,Z;\CC)$.
 Suppose
\[
F=\sum_{l} A_{l}[\pa_l] \in H_1(X_{x_0}\setminus P_{x_0},Z_{x_0};\CC)
\]
is an equation vanishing on $V$ of level $i$. If $F$ is a \hcc  we want to ignore $F$ and otherwise we restrict it to its top level to obtain an equation vanishing on $\partial M$.
In order to makes this precise, we need to recall some of the setup from \Cref{sec:Filtr}. The only difference is that now we exclusively work with homology and cohomology with $\CC$-coefficients; to lighten notation all spaces here are simply obtained by tensoring their analogue from \Cref{sec:Filtr} by $\CC$.
For example, the vertical filtration \[W_i\subseteq  H_1(X_{x_0}\setminus P_{x_0},Z_{x_0};\CC)\] consists of paths of level at most $i$ that do not intersect any horizontal vanishing cycles of level $i$.
We denote $a_i:W_i\to H_1(X_{x_0}\setminus P_{x_0},Z_{x_0};\CC)$ the inclusion.
The specialization map
\[
f_i:W_i\to  \Hcut\!/\!\GRC_{(i)}
\]
is obtained by restricting a path in $W_i$ to the surface $\Sigma_{(i)}^{cut}$. We recall that $\Sigma_{(i)}^{cut}$ consists only of the $i$-th level piece of $\Sigma$ and is furthermore cut along all horizontal vanishing cycles of level $i$.
% Using the vertical filtration $\GRCF$  and the maps $f_i, \alpha_i$ from \Cref{sec:Filtr} we can now phrase this as follows.

The following proposition is now a matter of rephrasing what we have proved so far in the language of \Cref{sec:Filtr}. We denote by $f_i^*$ and $a_i^*$ the dual maps of $f_i$ and $a_i$, respectively.

\begin{proposition}\label{prop:CoordFree}
Let $M$ be a linear subvariety, $b_0\in \partial M\cap\bdComp$ and $x_0 \in M_{reg}\cap \left(B\setminus D\right)$ with $T_{x_0}M=V$.
 Then $\partial M\cap \bdComp$ is near $b_0$ defined by the linear subspace
 \[
p\left(\prod_{i\in\lvlset} (f_i^{*})^{-1}(a_i^{*}(V))\right)\subseteq \Hi[(0)](X;\CC)\times\prod_{i\in\lvlsetb} \PP\left( \Hi(X;\CC)^{GRC}\right) ,
 \]

where $p:\prod_{i\in\lvlset} \Hi(X;\CC)^{GRC}\to\Hi[(0)](X)\times\prod_{i\in\lvlsetb} \PP \left(\Hi(X;\CC)^{GRC}\right )$ is the natural quotient map.
\end{proposition}

\begin{proof}
We write $V=\ANN(E)$ where $E\subseteq H_1(X_{x_0}\setminus P_{x_0},Z_{x_0};\CC)$ are the linear equations defining $V$.  More explicitly $E=\left(H^1(X_{x_0}\setminus P_{x_0},Z_{x_0};\CC)/ V\right)^*$.

The restriction $E\cap W_i$ corresponds to all equations with top level $\leq i$ which do not pass through any horizontal nodes and $f_i(E\cap W_i)$  is then  generated by their restrictions to level $i$.
Note that $f_i(E\cap W_i)\subseteq  H_1\!\left(\Sigma^{cut}_{(i)}\!\setminus\! P,Z\cup \Lambda_{(i)}^{ver,+};\CC\right)\!/\!\GRC_{(i)}$
and 
\[
\ANN(f_i(E\cap W_i))\subseteq\left(\!H_1\!\left(\Sigma^{cut}_{(i)}\!\setminus\! P,Z\cup \Lambda_{(i)}^{ver,+};\CC\right)\!/\!\GRC_{(i)}\right )^*\!\!\!= \! \!H^1\!\left(\Sigma_{(i)}^{cut},Z\cup \Lambda_{(i)}^{+,ver};\CC\right)^{\GRC}\!\!= \! \Hi(X;\CC)^{\GRC}.
\]

By~\Cref{prop:LinContainment} and the explicit form of the equations~\cref{eq:LimEqVer}, we see that the defining equations for $\partial M$  near $x_0$ on the $i$-th level are exactly given by $f_i(E\cap W_i)$. In particular $\partial M$ is given by the linear subspace $p(\prod_{i\in\lvlset}\ANN(f_i(E\cap W_i)))$.

%It thus remains to show that
%\[
%\ANN(E^{res})=\cap_{i\in\lvlset} (f_i^{*})^{-1}(\alpha_i^{*}(V))).
%\]
%We recall some simple facts about annihilators. Firstly if $A,B$ are subspaces of a finite dimensional vector space $C$, then $\ANN(A + B)=\ANN(A)\cap\ANN(B)$.
%Now suppose $A,B$ and are finite-dimensional vector spaces with subspaces $C\subseteq B,\, D\subseteq B^*$ and $T:A\to B$ is a linear map.
%%Suppose $A\subseteq B, C, V\subseteq A^*$ are finite-dimensional vector spaces and $T:B\to C$ is a linear map. 
%Then $T^*(\ANN(C))=\ANN(T^{-1}(C))$ and $T^{-1}(\ANN(D))=\ANN(T^*(D))$.
It now follows from properties of the annihilator that

\[\begin{split}
\ANN(f_i(E\cap W_i))
&= \ANN(f_i(a_i^{-1}(E)))\\
&=  (f_i^{*})^{-1}(\ANN(a_i^{-1}(E)))\\
&= (f_i^{*})^{-1}(a_i^{*}(V))).
\end{split}
\]

\end{proof}

\begin{remark}\label{rem:TSConv} We now explain how \Cref{prop:CoordFree} relates to the results of \cite{MirzakhaniWrightWYSIWYG}.
Since $f_i$ is surjective, its dual $f_i^*$ is injective, and we can identify $(f_i^{*})^{-1}(a_i^{*}(V))$ with its image under $f_i^*$ inside $W_i^*$. We obtain
\[
(f_i^{*})^{-1}(a_i^{*}(V))\simeq a_i^{*}(V)\cap \im(f_i^{*})= a_i^{*}(V)\cap \ANN(\ker f_i)\subseteq W_i^*.
\]

In the special case where $\Stra$ is a stratum of holomorphic differentials and $\enhancG$ has no horizontal nodes, we have $H_1(\Sigma\setminus P, Z)=\LVLF[0]=\GRCF[0]$ and $\LVLF[1]=\GRCF[1]= \ker(f_0)$.
In this case the tangent space for the top level part only is given by
\[
(f_0^*)^{-1}(V)\simeq V\cap \ANN(W_1).
\]
This coincides with the tangent space description of \cite[Thm 2.9]{MirzakhaniWrightWYSIWYG}.
We stress that the space of vanishing cycles in \cite{MirzakhaniWrightWYSIWYG} coincides with $W_1$ in our notation and should not be confused with the collection of vanishing cycles $\Lambda$.
\end{remark}

\section{The boundary in  the \WYSI partial compactification }\label{sec:bdWYSI}
We now show how to quickly deduce \Cref{thm:WYSImain} from \Cref{prop:CoordFree}, reproving one of the main results of \cite{ChenWrightWYSIWYG}.

First we recall some of the setup. We work on a fixed stratum $\Stra$ which is allowed to be meromorphic (note that in \cite{ChenWrightWYSIWYG} only holomorphic strata are considered).
For a multi-scale differential $(X,\twistD)$ we define $\pi(X,\twistD)=(X_{(0)},\omega_{(0)})$ to be the top level projection and recall the partial compactification
\[
\WYSIC:=\MSDS/ (\pi(X,\twistD))\simeq \pi(X',\eta')).
\]
with the natural projection map $p:\MSDS\to \WYSIC$. The fibers of $p$ are compact, since they are given by a cover of a finite union $\cup_{\mu'}\PP\Xi\Mgn[g',n'](\mu')$.

Let $M\subseteq \Stra$ be a linear subvariety, potentially defined over the complex numbers.
We now fix a boundary point $(X_{\infty},\omega_{\infty})\in p(\partial M)$  in a stratum $\Stra[(\omega_{\infty})]$.
Our goal is to show that $p(\partial M)\cap \Stra[(\omega_{\infty})]$ is a finite union of linear subvarieties. While Chen and Wright \cite[Thm. 1.1]{ChenWrightWYSIWYG} show that $\WYSIC$ is not a complex analytic space and in general $p$ is only continuous, the restriction $p_{|p^{-1}(\Stra[(\omega_{\infty})])}:p^{-1}(\Stra[(\omega_{\infty})])\to \Stra[(\omega_{\infty})]$ is in fact, algebraic and proper. Thus $p(\partial M)\cap \Stra[(\omega_{\infty})]$ is algebraic.
Furthermore, we have computed at each point $x$ of $p^{-1}(X_{\infty},\omega_{\infty})\cap (\partial M\cap \bdComp)$ the linear equations defining $\partial M\cap \bdComp$. We now want to show that there are only finitely many different linear equations for the top level when we vary the point $x$ over  $p^{-1}(X_{\infty},\omega_{\infty})\cap \partial M$.
While so far we have worked on each open boundary component $\bdComp$ separately, $p$ maps different boundary strata into $\Stra[(\omega_{\infty})]$ and thus \cite[Thm 1.2]{ChenWrightWYSIWYG} does not just follow as part of the statements we proved.
Before we can proceed with the proof we thus need some preparation.

Given $(X,\eta)\in p^{-1}(X_{\infty},\omega_{\infty})\cap \bdComp$ our first goal is to determine which nearby boundary strata map into $p^{-1}(X_{\infty},\omega_{\infty})$.
 Let $B\subseteq \MSDS$ be a small chart containing $(X,\eta)$.

We define
\[
D_{-1}:=\{ t_{-1}=0\}\subseteq B,
\]
where we recall that $t_{-1}$ was defined in \Cref{section:ModelDomain}
Thus $D_{-1}$ is a union of boundary strata, corresponding to undegenerations of $\enhancG$ with the same top level graph as $\enhancG$.
Since on $\WYSIC$ all zeroes and preimages of nodes are marked, we have the following observation.

\begin{lemma}\label{lemma:TopUndeg}Let $(X,\eta)\in B\cap p^{-1}(X_{\infty},\omega_{\infty})$ as above. Then
\[
B\cap p^{-1}(X_{\infty},\omega_{\infty}) \subseteq D_{-1}.
\]
\end{lemma}

We need to study the asymptotic of log periods on $D_{-1}$. While the general asymptotic for undegenerations is complicated, for paths of top level $0$ there is a formula, similar to \Cref{thm:PerThm}. We only state it in the case of curves not crossing horizontal nodes but it can be adapted in general.

\begin{lemma}\label{lemma:PerUndeg}
Let $\pa$ be a path of top level $0$ not crossing any horizontal nodes.
Then
\[
\LogP(b)= \int_{\pa(b)_{\topl}}\twistD
\]
for all $b\in D_{-1}$.
\end{lemma}

Note that this shows that the top level equations for $\partial M \cap\bdComp$ are constant on $D_{-1}$, at least along a local irreducible component of $\partial M$.

\begin{proof}
The proof of \Cref{thm:PerThm} works almost verbatim, the only difference is that
$\stdFormB[b_0]^{-1}(\Nearby[e](b_0))=p_0$
is only true for nodes with $\ell(e+)=0$ since the modifying differential on the top level vanishes.

Thus at nodes with $\ell(e+)=0$ we have
 \[
\int^{\stdFormB[(\twistD,t,h)]^{-1}(\Nearby[](\twistD,t,h))}_{p_0}\twistD=O(t_{-1}).
\]
On lower levels, we only know that $\stdFormB[b_0]^{-1}(\Nearby[e](b_0))$ is bounded, but the integrand is divisible by $t_{-1}$ and thus
 \[
\int^{\stdFormB[(\twistD,t,h)]^{-1}(\Nearby[](\twistD,t,h))}_{p_0}(\starD{\twistD}+\modD)=O(t_{-1})
\]
as well.
\end{proof}

We have now assembled all the tools to give an independent proof of \Cref{thm:WYSImain}.

\begin{proof}[Proof of~\Cref{thm:WYSImain}]
Let $(X_{\infty},\omega_{\infty})\in \partial M\cap\WYSIC$. 
For each $y\in W\!:=p^{-1}(X_{\infty},\omega_{\infty})$ we let $\enhancG(y)$ be the associated enhanced level graph.
% For each $y\in W \setminus {\partial M}$ we choose  a neighborhood $U_y$ of $y$ in $\MSDS$ such that $U_y\cap \overline{M}=\emptyset$ and for each $y\in\overline{M}$ we choose $U_y$ such that $U_y\cap \overline{M}$ has only finitely many irreducible components and $U_y\cap\bdComp$ satisfies the conditions of \Cref{prop:LinContainment} for each local irreducible component.
If $y\in\overline{M}$ we choose $U_y$ such that $U_y\cap \overline{M}$ has only finitely many irreducible components and $U_y\cap\bdComp$ satisfies the conditions of \Cref{prop:LinContainment} for each local irreducible component.
 
We write $\partial M\cap \bdComp\cap U_y= \cup_{\beta} M_{\beta}(y)$ as a union of its local irreducible components. We choose arbitrary points $z_{\beta}(y)$ in $M_{\beta}(y)\cap\WYSIC$ and let $V_{\beta}(y)=T_{z_{\beta}(y)} M_{\beta}(y)$ be the corresponding tangent space.
Note that so far in \Cref{prop:LinContainment} we assumed that $z_{\alpha}(y)$ is a smooth point of $M$. We now claim that this assumption can be removed as follows.

 In a small period chart around $z_{\beta}(y)$ we can choose a smooth point of $M_{\beta}(y)$ and then transport the linear subspace for the given branch via the Gauss-Manin connection to $z_{\beta}(y)$.

We set $V_{\beta}^{(0)}(y):= (f_0^*)^{-1}(\alpha_0)^*(V_{\beta}(y))$.
 In the case $\omega_{\infty}$ has no simple poles, by  \Cref{rem:TSConv}, we can identify  $V_{\beta}^{(0)}(y)$ with the intersection of the tangent space to $M_{\beta}(y)$ at $z_{\beta}(y)$ and the tangent space to $\Stra[(\omega_{\infty})]$.

We can assume that $U_{y_{l}}=U'_{y_{l}}\times U''_{y_{l}}$ is a product of polydisks where $U'_{y_{l}}\subseteq H_1(X_{(0)}\setminus P_{(0)},Z_{(0)})$ is a polydisk in the coordinates of the top level.
By compactness of $W$ we can cover $W\cap \partial M$ by finitely many $U_{y_1},\ldots,U_{y_l}$ with $y_k\in \partial M$.
We set $U:= \cap_{k=1}^{l}U'_{y_k}$.

We claim that
\[
\partial M\cap \Stra[(\omega_{\infty})]\cap U= \bigcup_{k=1}^{l}\bigcup_{\beta} V^{(0)}_{\beta}(y_k)\cap U.
\]
Let $(x_n)$ be a sequence in $M$ converging to $(X'_{\infty},\omega'_{\infty})\in U\cap p(\partial M)$. After removing finitely many elements of the sequence, the sequence is contained in $\cup_{k=1}^{l}U_{y_k}$. We partition it into subsequences $(x_{n}^{(k)})$ contained in $U_{y_{k}}$.
After passing to a subsequence we can assume that $(x_n^{(k)})$ converges to $x^{(k)}\in \partial M\cap U_{y_k}$. Now if $x^{(k)}$ lies in the same open boundary component $\bdComp$ as $y_k$, it follows from \Cref{prop:CoordFree} that $p(x^{(k)})\in V^{(0)}_{\beta}(y_k))$ for some $\beta$. If $x$ lies in an undegeneration, then it follows first from \Cref{lemma:TopUndeg} that $x^{(l)}\in D_{-1}$ and then the claim follows from \Cref{lemma:PerUndeg}.

On the other hand if $x^{(k)}\in V^{(0)}_{\beta}(y_k)\cap U$, then we construct a multi-scale differential $(X,\twistD)$ by gluing $x^{(k)}$ and the lower level parts of $y_k$. Since the equations are level-wise, it follows that $(X,\omega)$ is contained in $U_{y_k}\cap \partial M$, and thus there exists a sequence $(x_n)$ in $M$ converging to $(X,\twistD)$ and by continuity $(x_n)$ converges to $x^{(k)}$ in $\WYSIC$.
This proves the first claim.

The second claim follows since we can choose $z_{\beta}(y)$ arbitrarily inside the branch $M_{\beta}(y)$ and thus can take $z_{\beta}(y)=x_n$.

\end{proof}

\end{document}